\documentclass[11pt]{article}

\usepackage{amsmath,amssymb,amsthm,amsfonts,mathrsfs}
\usepackage[pdftex]{graphicx,epsfig,color}
\usepackage[T1]{fontenc}
\usepackage{tikz}
\usepackage{xcolor}
\usepackage{float}

\def\R{\mathbb R}

\def\N{\mathbb N}

\def\<{\langle}
\def\>{\rangle}

\def\<{\langle}
\def\>{\rangle}

\numberwithin{equation}{section}

\newtheorem{theorem}{Theorem}

\newtheorem{lemma}[theorem]{Lemma}
\newtheorem{proposition}[theorem]{Proposition}
\newtheorem{corollary}[theorem]{Corollary}

\newtheorem{remark}[theorem]{Remark}

\numberwithin{theorem}{section}

\title{Subdifferentials and Minimizing Sard Conjecture in Sub-Riemannian Geometry}

\author{L.~Rifford\thanks{Universit\'e C\^ote d'Azur, CNRS, Labo.\ J.-A.\ Dieudonn\'e,  UMR CNRS 7351, Parc Valrose, 06108 Nice Cedex 02, France ({\tt ludovic.rifford@math.cnrs.fr})}  
}

\date{}

\makeindex

\begin{document}

\maketitle

{\centering\footnotesize Dedicated to Professor Francis Clarke's 75th Birthday\par}

\begin{abstract}
We use techniques from nonsmooth analysis and geometric measure theory to provide new examples of complete sub-Riemannian structures satisfying the Minimizing Sard conjecture. In particular, we show that complete sub-Riemannian structures associated with distributions of co-rank $2$ or generic distributions of rank $\geq 2$ satisfy the Minimizing Sard conjecture.
\end{abstract}


\section{Introduction}\label{sec:Intro}

Consider a smooth connected manifold $M$ of dimension $n\geq 3$ equipped with a sub-Riemannian structure $(\Delta,g)$ which consists of a totally nonholonomic smooth distribution $\Delta$ of rank $m<n$ and a smooth metric $g$ over $\Delta$. By the Chow-Rashevsky Theorem, such a structure makes $M$ horizontally connected, that is, for any $x,y \in M$ there is $\gamma:[0,1] \rightarrow M$, absolutely continuous with derivative in $\R^2$, called horizontal path, which joins $x$ to $y$ and satisfies 
\[
\dot{\gamma}(t) \in \Delta \left( \gamma(t)\right) \qquad \mbox{ for a.e. } t \in [0,1].
\]
Then, define the function $d_{SR}:M\times M \rightarrow \R$ by  
\[
d_{SR}(x,y) := \inf \Bigl\{ \mbox{length}^g(\gamma) \, \vert \, \gamma \in \Omega^{\Delta}_x, \, \gamma(1) = y\Bigr\} \qquad \forall (x,y) \in M \times M,
\]
where $\Omega^{\Delta}_x$ stands for the set of horizontal paths $\gamma:[0,1] \rightarrow M$ such that $\gamma(0)=x$. The function $d_{SR}$, called sub-Riemannian distance with respect to $(\Delta,g)$, makes $(M,d_{SR})$ a metric space which defines the same topology as the one of $M$ as a manifold, and furthermore, as in Riemannian geometry, thanks to  a sub-Riemannian version of Hopf-Rinow Theorem, the completeness of $(M,d_{SR})$ guarantees that for any pair $x,y\in M$ there is an horizontal path, called minimizing, which minimizes the sub-Riemannian distance between $x$ and $y$. We refer the reader to the monographs \cite{abb17,montgomery02,riffordbook} for further details on sub-Riemannian geometry and we assume from now that the metric space  $(M,d_{SR})$ is complete (the sub-Riemannian structure $(\Delta,g)$ is said to be complete).

Within the set of horizontal paths, the so-called singular minimizing horizontal paths play a significant role in sub-Riemannian geometry. Those are critical points of the End-Point mapping whose end-points may correspond to loss of regularity of the sub-Riemannian distance so a good understanding of the space filled by them is crucial. The Minimizing Sard Conjecture is precisely concerned with the size of the (closed) set of points, denoted by  $\mbox{Abn}^{min}(x)$,  that can be reached from a given point $x\in M$ through singular minimizing horizontal paths (see \cite[\S 10.2]{montgomery02}, \cite[Conjecture 1 p. 158]{rt05}, \cite[\S III.]{agrachev14} or \cite[\S 2.1]{br20}):\\

\noindent {\bf Minimizing Sard Conjecture.}  For every $x\in M$, $\mbox{Abn}^{min}(x)$ has Lebesgue measure zero in $M$.\\

Although various special cases of the Minimizing Sard Conjecture have been verified, the conjecture remains open in its full generality. The best result toward the conjecture, due to Agrachev \cite{agrachev09}, shows that $\mbox{Abn}^{min}(x)$  has always empty interior and to date, the Minimizing Sard Conjecture is known to hold true in the following cases:

\begin{itemize}
\item[(i)] The distribution $\Delta$ is {\it medium-fat}, that is, for every $x\in M$ and every smooth section $X$ of $\Delta$ with $X(x) \neq 0$, there holds
\begin{eqnarray*}
T_xM = \Delta(x)+ [\Delta,\Delta](x) + \bigl[X,[\Delta,\Delta] \bigr](x),
\end{eqnarray*}
where 
\[
[\Delta,\Delta] (x) := \mbox{Span} \Bigl\{ [Y,Z](x) \, \vert \,  Y, Z \mbox{ smooth sections of } \Delta \Bigr\}
\]
and
\[
 \bigl[X,[\Delta,\Delta] \bigr] (x) := \mbox{Span} \Bigl\{ \bigl[X,[Y,Z]\bigr] (x) \, \vert \,  Y, Z \mbox{ smooth sections of } \Delta \Bigr\}.
\]
The result follows from the lipschitzness properties of the sub-Riemannian distance obtained by Agrachev and Lee \cite{al09} which is itself a consequence of previous works on medium-fat distributions and Goh abnormals by Agrachev and Sarychev \cite{as99} (see also \cite{riffordbook}). It allows easily to obtain the conjecture for distributions which are medium fat almost everywhere such as  distributions of step $2$ almost everywhere ($\Delta(x)+[\Delta,\Delta](x)=T_xM$ for almost every $x\in M$) like for example co-rank $1$ distributions ($m=n-1$). 
\item[(ii)] The sub-Riemannian structure $(\Delta,g)$ has rank $m\geq 3$ and is generic in the set of sub-Riemannian structures of rank $m$ over $M$ endowed with the Whitney smooth topology. The result follows from the absence of Goh controls for generic distributions of rank $m\geq 3$ (as shown by Agrachev and Gauthier \cite[Theorem 8]{ag01} and Chitour, Jean and Trélat \cite[Corollary 2.5]{cjt06}) along with the fact that singular minimizing horizontal paths for generic sub-Riemannian structures are strictly abnormal (see \cite{cjt06}). 
\item[(iii)] The sub-Riemannian structure $(\Delta,g)$ over $M$ corresponds to a Carnot group of step $\leq 3$ (see \cite{lmopv16}) or of rank $2$ and step $4$ (see \cite{bv20}). The latter case verifies indeed the stronger Sard Conjecture. We refer the reader to  \cite{bnv22,lmopv16} for a few other specific examples of Carnot groups satisfying the Sard Conjecture and to \cite{agrachev14,br18,bfpr18,bpr22,bnv22,bv20,lmopv16,montgomery02,ov19,riffordbourbaki,rt05} for further details, results and discussions on that conjecture.
\end{itemize}

In this paper, we aim to provide new examples of complete sub-Riemannian structures satisfying the Minimizing Sard conjecture. In the spirit of a previous work by Trélat and the author \cite{rt05}, we are going to show that the sub-Riemannian structure $(\Delta,g)$ satisfies the Minimizing Sard conjecture whenever all pointed distances $d_{SR}(x,\cdot)$, with $x\in M$, are almost everywhere Lipschitz from below, and then in a second step, we shall give sufficient conditions for the latter property to hold true. Our results are as follows:\\

For every $x\in M$, we denote by $\mbox{Goh-Abn}^{min}(x)$ the set of $y\in M$ for which there is a minimizing horizontal path $\gamma \in \Omega^{\Delta}_x$ joining $x$ to $y$ which is singular and admits an abnormal lift $\psi:[0,1] \rightarrow \Delta^{\perp}$ satisfying the Goh condition
\begin{eqnarray}\label{Goh0}
\psi(t) \cdot [\Delta,\Delta](\gamma(t)) = 0 \qquad \forall t \in [0,1].
\end{eqnarray}
As shown by Agrachev and Sarychev (see \cite{as96,as99}), any singular minimizing horizontal path, which is not the projection of a normal extremal, does admit an abnormal lift satisfying the Goh condition (\ref{Goh0}) and moreover Agrachev and Lee \cite{al09} proved that the absence of abnormal lifts satisfying (\ref{Goh0}) along all minimizing paths from $x$ to $y$ guarantees that the pointed distance $d_{SR}(x,\cdot)$ is Lipschitz on a sufficiently smalll neighborhood of $y$. We refer the reader to Section \ref{SECGoh} for several comments and proofs regarding those results. By construction, $\mbox{Goh-Abn}^{min}(x)$ is a closed set satisfying 
\[
\mbox{Goh-Abn}^{min}(x) \subset \mbox{Abn}^{min}(x).
\]
We define the Goh-rank of an horizontal path $\gamma \in \Omega^{\Delta}_x$ , denoted by $\mbox{Goh-rank}(\gamma)$, as the dimension of the vector space of abnormal lifts  $\psi:[0,1] \rightarrow \Delta^{\perp}$ satisfying the Goh condition (\ref{Goh26sept}). Note that if $\gamma$ is not singular, then it admits no abnormal lifts so its Goh-rank has to be $0$. Our main result is the following:

\begin{theorem}\label{THM}
Let $M$ be a smooth manifold equipped with a complete sub-Riemannan structure $(\Delta,g)$ and $x\in M$ be fixed. If for almost every $y\in M$ all minimizing horizontal paths from $x$ to $y$ have Goh-rank at most $1$, then the closed set $\mbox{Abn}^{min}(x)$ has Lebesgue measure zero in $M$.
\end{theorem}

Although the assumptions of Theorem \ref{THM} seem to be not easily checkable, they are automatically satisfied in a certain number of cases that we proceed to describe. 

We say that a sub-Riemannian structure $(\Delta,g)$ has {\it minimizing co-rank $1$ almost everywhere} if for every $x\in M$ and almost every $y\in M$, every singular minimizing horizontal path from $x$ to $y$ has co-rank $1$.  Of course, a minimizing horizontal path of co-rank $1$ cannot have Goh-rank strictly more than $1$. Therefore, we have: 

\begin{corollary}\label{COR1}
Let $M$ be a smooth manifold equipped with a complete sub-Riemannan structure $(\Delta,g)$ having minimizing co-rank $1$ almost everywhere. Then the Minimizing Sard Conjecture is satisfied.
\end{corollary}

We say that the distribution $\Delta$ is {\it pre-medium fat almost everywhere} if for almost every $x\in M$ and every smooth section $X$ of $\Delta$ with $X(x) \neq 0$, there holds
\begin{eqnarray}\label{premedium}
\dim \left( \Delta(x)+ [\Delta,\Delta](x) + \bigl[X,[\Delta,\Delta] \bigr](x)\right) \geq n-1.
\end{eqnarray}
This is for example the case of totally nonholonomic smooth distributions of rank $m\geq n-2$. We can check by taking one derivative in (\ref{Goh0}) that if $\gamma \in \Omega_{x}^{\Delta}$ is a minimizing horizontal path from $x$ to $y\neq x$ in $M$, then any abnormal lift $\psi$ satisfying the Goh condition of $\gamma$ must verify 
\[
\psi(t) \cdot \left( \Delta(\gamma(t)) + [\Delta,\Delta](\gamma(t)) + [X^t(\gamma(t)),[\Delta,\Delta]] \right) = 0 \qquad \mbox{for a.e. } t \in [0,1],
\]
where $X^t$ is a smooth section of $\Delta$ defined on a neighborhood of $\gamma(t)$ such that $X(\gamma(t)) = \dot{\gamma}(t)$.  Therefore, if $y$ is a point where (\ref{premedium}) is satisfied then for almost every $t\in [0,1]$ close to $1$ any abnormal lift of $\gamma$ satisfying (\ref{Goh0}) must annihilate a vector hyperplane, which forces $\mbox{Goh-rank}(\gamma)$ to be at most $1$. In conclusion, we have the following:

\begin{corollary}\label{COR2}
Let $M$ be a smooth manifold equipped with a complete sub-Riemannan structure $(\Delta,g)$ with $\Delta$ pre-medium-fat almost everywhere. Then the Minimizing Sard Conjecture is satisfied.
\end{corollary}

We finish with a corollary which actually follows from Corollary \ref{COR1}. As shown by Chitour, Jean and Trélat (see \cite[Theorem 2.4]{cjt06}), there is an dense open set $O_m$ in the set $\mathcal{D}_m$ (with $m\geq 2)$ of smooth totally nonholonomic distributions of rank $m$ endowed with the Whitney $C^{\infty}$ topology such that for every $\Delta \in O_m$, every nontrivial singular horizontal path (w.r.t. $\Delta$) has co-rank $1$. Thus, for any $\Delta \in O_m$ and any smooth metric $g$ over $\Delta$ with $(\Delta,g)$ complete, the sub-Riemannian structure  $(\Delta,g)$ has co-rank $1$ almost everywhere. In conclusion, we have:

\begin{corollary}\label{COR3}
Let $M$ be a smooth manifold equipped with a complete sub-Riemannan structure $(\Delta,g)$ of rank $m\geq 2$ where $\Delta$ is generic in the set of smooth totally nonholonomic distrubutions of rank $m$ over $M$. Then the Minimizing Sard Conjecture is satisfied.
\end{corollary}

The proof of Theorem \ref{THM} consists in proving that for every $x\in M$, the pointed distance $d_{SR}(x,\cdot)$ admits at almost every $y\in M$ a support function from below which is Lipschitz. Roughly speaking, this result follows on the one hand from the fact that the assumption on Goh-ranks of minimizing geodesics provides  some lipschitzness property along an hypersurface $\mathcal{S}$ near $y$ while one the other hand the continuous function in one dimension given by the restriction of $d_{SR}(x,\cdot)$ to a curve transverse to $\mathcal{S}$ is at almost every point either differentiable or limit of many oscillations (see Section \ref{SECAlter}). The key result in the proof is this alternative satisfied almost everywhere by continuous functions in one dimension between differentiability and limits of points where the function reaches  local minima (see Proposition \ref{PROPDichotomy}). Since such an alternative is probably not available in higher dimension, our approach does not allow to treat the case of distributions with minimizing horizontal paths of Goh-rank $\geq 2$. \\

The paper is organized as follows: In Section \ref{SECSubLip}, we recall several notions of subdifferentials and prove several results of importance for the rest of the paper regarding Lipschitz-type properties of continuous functions. We explain in Section \ref{SEC3} how some properties satisfied by those subdifferentials interplay with the Minimizing Sard Conjecture; in particular, we provide in Proposition \ref{PROPcharac} various characterizations of the conjecture. Then, Section \ref{SECProofTHM} is devoted to the proof of Theorem \ref{THM} and we comment on our results in Section \ref{SECComments}. Finally, Appendix \ref{SECSecondOrder} contains a reminder on second order conditions for openness.\\

\noindent {\bf Acknowledgement.} The author is indebted to Aris Daniilidis and Alex Ioffe for fruitful discussions.

\section{Subdifferentials and lipschitzness}\label{SECSubLip}

Throughout this section, we consider a function $f:\mathcal{O} \rightarrow \R$ defined on an open set $\mathcal{O}\subset M$ and we suppose that $f$ is continuous on $\mathcal{O}$. We recall that $f$ is said to admit a {\it support function from below} $\varphi$ at some point $x\in \mathcal{O}$ if $\varphi:\mathcal{V} \rightarrow \R$ is defined on an open neighborhood $\mathcal{V}\subset \mathcal{O}$ of $x$ and satisfies 
\[
f(x) = \varphi(x)  \quad \mbox{and} \quad f(y) \geq \varphi(y) \quad \forall y \in \mathcal{V}.
\]  
We gather in the following sections several notions and results of non-smooth analysis that may be found  in \cite{clsw98} or \cite{rw98} in the Euclidean setting, we provide all proofs for sake of completeness. 

\subsection{Viscosity and proximal subdifferentials}

The {\it viscosity} (or {\it Fréchet}) {\it subdifferential} of $f$ at $x\in \mathcal{O}$, denoted by $\partial^-f(x)$, is defined as the set of $p\in T_x^*M$ for which $f$ admits a support function from below $\varphi$ at $x$ of class $C^1$ satisfying $d\varphi(x)=p$. Note that if $f$ admits a support function from below which is differentiable at $x$ then it admits a support function from below at $x$ which is of class $C^1$. Thus, we check easily that if $f$ is differentiable at $x$ then we have $\partial^-f(x)=\{df(x)\}$ and furthermore we observe that if $f$ attains a local minimum at $x$ then we have $0\in \partial^-f(x)$. The set $\partial^-f(x)$ is always a convex subset of $T_x^*M$ but it may be empty, as shown by the example of $f(x)=-|x|$ on the real line with $x=0$. Nonetheless, we have the following result:

\begin{proposition}\label{PROPnonemptydense}
The set of $x\in \mathcal{O}$ such that $\partial^-f(x)\neq \emptyset$ is a dense subset of $\mathcal{O}$. 
\end{proposition}

\begin{proof}[Proof of Proposition \ref{PROPnonemptydense}]
For every open set $\mathcal{O}' \subset \mathcal{O}$ with $\overline{\mathcal{O}'}\subset \mathcal{O}$, we can construct a smooth function $\beta:\mathcal{O}' \rightarrow [0,+\infty)$ which tends to $+\infty$ when approaching the boundary of $\mathcal{O}'$, so that $f+\beta$ attains its minimum at some point $x\in \mathcal{O}'$ where $\varphi := -\beta + f(x)+\beta(x)$ satisfies 
\[
f(x) = \varphi(x)  \quad \mbox{and} \quad f(y) \geq \varphi(y) \quad \forall y \in \mathcal{O}',
\]  
which shows that $d\varphi(x)=-d\beta(x)$ belongs to $\partial^-f(x)$.
\end{proof}

Note that this result is sharp, we cannot expect in general more than nonemptyness of the viscosity subdifferential for a dense set of points. Examples of functions in one variable having nonempty viscosity subdifferentials only for a countable set of points may be given by Weierstrass or Van der Waerden functions, see \cite{gs11}.\\

The {\it proximal subdifferential} of $f$ at $x\in \mathcal{O}$, denoted by $\partial^-_P f(x)$, is defined as the set of $p\in T_x^*M$ for which $f$ admits a support function from below $\varphi$ at $x$ of class $C^{2}$ on its domain and satisfying $d\varphi(x)=p$. Note that if $f$ admits a support function from below which is of class $C^2$ then it admits a support function from below at $x$ which is of class $C^{\infty}$. The set $\partial_P^-f(x)$ is a convex subset of $T_x^*M$ that may be empty and which is contained in $\partial^-f(x)$. Note that the inclusion $\partial_P^-f(x)\subset \partial^-f(x)$ may be strict as shown by the example $f(x)=-|x|^{3/2}$ at $x=0$ on the real line. The proof of Proposition \ref{PROPnonemptydense} shows that the set of $x\in \mathcal{O}$ such that $\partial_P^-f(x)\neq \emptyset$ is a dense subset of $\mathcal{O}$. In fact, we have much more than this, we can show that elements of $\partial^-f$ can be approximated by elements of $\partial^-_Pf$. More precisely, we have:

\begin{proposition}\label{PROPidem}
For every $x\in \mathcal{O}$, every $p\in \partial^-f(x)$ and every neighborhood $\mathcal{W}$ of $(x,p)$ in $T^*M$, there is $y \in \mathcal{O}$ and $q\in \partial^-_P f(y)$ such that $(y,q)\in \mathcal{W}$. 
\end{proposition}

We do not give the full proof of Proposition \ref{PROPidem}, we are just going to show how to deduce the result from a theorem by Subbotin \cite{subbotin95}. We follow the proof given by Clarke, Ledyaev, Stern and Wolenski in \cite[Proposition 4.5 p. 138]{clsw98}.

\begin{proof}[Proof of Proposition \ref{PROPidem}]

The Subbotin Theorem reads as follows (its proof can be found in  \cite[Theorem 4.2 p. 137]{clsw98}).
\begin{lemma}\label{LEMSubbotin}
Let $h:O \rightarrow \R$ be a continuous function on an open set $O\subset \R^n$, $x\in \R^n$, and let $\rho \in \R$ be such that
\[
Dh(x;v) := \lim_{\stackrel{w\rightarrow v}{t\downarrow 0}} \frac{h(x+tw)-h(x)}{t} > \rho \qquad \forall v \in \bar{B}_1,
\]
where $\bar{B}_1$ stands for the closed unit ball in $\R^n$. Then, for any $\epsilon >0$, there exist $z\in B_{\epsilon}(x)$ and $\zeta \in \partial^-_Ph(z)$ such that
\begin{eqnarray*}
|f(z)-f(x)| < \epsilon \quad
\mbox{and} \quad \zeta \cdot v > \rho \qquad \forall v \in \bar{B}_1.
\end{eqnarray*}
\end{lemma}
To prove Proposition \ref{PROPidem}, we consider a point $x\in \mathcal{O}$, a co-vector $p\in \partial^-f(x)$ and a neighborhood $\mathcal{W}$ of $(x,p)\in T^*M$. In fact, up to taking a chart, we can assume that work in $\R^n$ with a function $f$ defined on an open set $O\subset \R^n$, with $x\in O$ and $p\in (\R^n)^*$, and with a neighborhood $\mathcal{W}$ of $(x,p)\in T^*(\R^n)$ which contains a set of the form $B_{\delta}(x) \times B_{\delta}^*(p)$ with $\delta >0$. Let $h: O \rightarrow \R$ be the continuous function  defined by 
\[
h(y):=f(y)-p\cdot y \qquad \forall y \in O.
\]
Since $p\in \partial^-f(x)$, we check easily that 
\[
Dh(x;v) \geq 0 \qquad \forall v \in \bar{B}_1.
\]
Thus, by applying Lemma \ref{LEMSubbotin} with $\rho=-\delta$ and $\epsilon=\delta$, there are $z \in B_{\delta}(x)$ and $\zeta \in \partial^-_Ph(z)$ such that $|h(z)-h(x)| < \delta$ and
\begin{eqnarray*}
 \zeta \cdot v > -\delta \qquad \forall v \in \bar{B}_1.
\end{eqnarray*}
We infer that $|\zeta|<\delta$ and $p+\zeta \in \partial_P^-f(z)$. 
\end{proof}

\subsection{Lipschitz points}

We assume in this section that $M$ is equipped with a smooth Riemannian metric $h$ whose geodesic distance is denoted by $d^h$ and for which at every $x\in M$ the associated norm in $T_xM$ is denoted by $|\cdot|_x$ and the norm of some $p\in T_x^*M$ is defined by $|p|_x:=|v|_x$ where $p=h_x(v,\cdot)$ (we refer the reader to \cite{sakai96} for further details of Riemannian geometry).  Then, for every $x\in M$, the pointed distance $d^h_x:=d^h(x,\cdot)$ is $1$-Lipschitz with respect to $d^h$, there is an open neighborhood $\mathcal{V}$ of $x$ such that $d^h_x$ is smooth in $\mathcal{V} \setminus \{x\}$ with a differential of norm $1$ and we have
\begin{eqnarray}\label{22avril1}
\partial^-d_x^h(x) =\Bigl\{ p\in T_x^*M \, \vert \, |p|_x\leq 1\Bigr\}.
\end{eqnarray}
As shown by the following result, the lipschitzness of $f$ is controlled by the size of co-vectors in $\partial^- f$. We recall that a set $\mathcal{C}\subset M$ is said to be {\it convex} with respect to $h$ if any minimizing geodesic between two points of $\mathcal{C}$ is contained in $\mathcal{C}$.

\begin{proposition}\label{PROPSubLip}
Let $\mathcal{C} \subset \mathcal{O}$ be an open convex set (with respect to $h$) and $K\geq 0$ be fixed, then the following properties are equivalent:
\begin{itemize}
\item[(i)] $f$ is $K$-Lipschitz (with respect to $d^h$) on $\mathcal{C}$.
\item[(ii)] For every $x\in \mathcal{C}$ and every $p\in \partial^-f(x)$, $|p|_x\leq K$. 
\end{itemize}
\end{proposition}

\begin{proof}[Proof of Proposition \ref{PROPSubLip}]
Assume that (i) is satisfied and fix $x\in \mathcal{C}$ and $p\in \partial^-f(x)$ (if $\partial^-f(x)$  is nonempty). By assumption $f$ admits a support function from below $\varphi:\mathcal{V}\rightarrow \R$ with $d\varphi(x)=p$ and by $K$-Lipschitzness of $f$ we have
\[
\varphi(y)  \leq f(y) \leq f(x) + Kd^h(x,y) = \varphi(x) + Kd_x^h(y) \qquad \forall y\in \mathcal{V} \cap \mathcal{C}.
\]  
We infer that $p=d\varphi(x)$ belongs to $\partial^- (Kd_x^h)(x)=K\partial^-d_x^h(x)$, which by (\ref{22avril1}) gives $|p|_x\leq K$. \\

Let us now assume that (ii) is satisfied and fix $\bar{x}\in \mathcal{C}$, $\bar{\delta}>0$ with $\bar{B}^h(\bar{x},\bar{\delta})\subset \mathcal{C}$ and some constant $K'>K$. Pick a smooth convex  function $\phi : [0,\bar{\delta}) \rightarrow [0,+\infty)$ such that 
\begin{eqnarray}\label{22sept1}
\phi(t) = K't \qquad \forall t\in [0,\bar{\delta}/4]
\end{eqnarray}
and
\begin{eqnarray}\label{22sept2}
\phi(5\bar{\delta}/16) + \min_{x\in \bar{B}^h(\bar{x},\bar{\delta})} \left\{f(x)\right\} \geq  f(\bar{x})+ \phi(\bar{\delta}/8),
\end{eqnarray}  
fix $y,z$ in $B(\bar{x},\bar{\delta}/8)$, and define $g:B^h(\bar{x},\bar{\delta}/2) \rightarrow \R$ by
\[
g(x):= f(x)+\phi \left(d^h(x,y) \right) \qquad \forall x \in B^h(\bar{x},\bar{\delta}/2).
\]
The function $g$ is continuous on $B(\bar{x},\delta/2)$ and satisfies (by (\ref{22sept2}) and the fact that $\phi$ is increasing) for every $x\in B^h(\bar{x},\bar{\delta}/2) \setminus B^h(\bar{x},7\bar{\delta}/16)$,
\[
g(x) > f(x) + \phi(5\bar{\delta}/16) \geq  f(\bar{x})+ \phi(\bar{\delta}/8) \geq g(\bar{x}),
\]
hence it attains a minimum at some point $x_g\in B(\bar{x},\bar{\delta}/2)$. If $x_g\neq y$, since $x\mapsto \phi(d^h(x,y))$ is $C^1$ on $B(\bar{x},\bar{\delta}/2)\setminus \{y\}$ (with differential of norm $1$), this means that 
\[
-\phi' \left(d^h(x_g,y)\right) dd_y^h(x_g) \in \partial^-f(x_g),
\]
thus $|\phi'(d(x_g,y))|\leq K$ which contradicts the properties satisfied by $\phi$ ((\ref{22sept1}) and the convexity). In consequence $x_g=y$. Therefore, 
\[
f(y)=g(x_g) \leq g(z) =f(z)+\phi \left(d^h(z,y)\right) \leq f(z) +K'd^h(z,y),
\]
where we have used that $d^h(y,z)\leq \bar{\delta}/8$ and (\ref{22sept1}). Since $y,z$ are arbitrary points in $B(\bar{x},\bar{\delta}/8)$ and $K'$ any constant $>K$, we are done. 
\end{proof}

We say that $f$ is {\it Lipschitz} at some point $x\in \mathcal{O}$ if there is a smooth Riemannian metric on $M$ such that $f$ is Lipschitz, with respect to that metric, on an open neighborhood of $x$, and we denote by $\mbox{Lip}\,(f)$ the set of such points. Of course, if $f$ is Lipschitz at some point $x\in \mathcal{O}$ with respect to some metric then it is Lipschitz with respect to any other metric. Proposition \ref{PROPSubLip} yields the following characterization:

\begin{proposition}\label{PROPSubLipEQ}
For every $x\in \mathcal{O}$, the following properties are equivalent:
\begin{itemize}
\item[(i)] $x \in \mbox{Lip}\,(f)$.
\item[(ii)] $\partial^-f$ is bounded in a neighborhood of $x$, that is, there is a neighborhood $\mathcal{V}\subset \mathcal{O}$ of $x$ such that the set of $(y,q)$ with $y\in \mathcal{V}, q\in \partial^-f(y)$ is relatively compact in $T^*M$.
\end{itemize}
\end{proposition}

Finally, we note that by construction the set $\mbox{Lip}(f)$ is open and $f$ is locally Lipschitz on $\mbox{Lip}(f)$, thus Rademacher's Theorem (see {\it e.g.} \cite{eg15,federer69}) implies that $f$ is differentiable almost everywhere on $\mbox{Lip}(f)$.

\subsection{Lipzchitz points from below and limiting subdifferentials}

We say that $f$ is {\it Lipschitz from below} at some point $x\in M$ if it admits a support function from below at $x\in M$ which is Lipschitz on its domain, and we denote by $\mbox{Lip}^-(f)$ the set of such points. Moreover, we call {\it limiting subdifferential} of $f$ at $x$, denoted by $\partial^-_Lf(x)$, the set of $p\in T_x^*M$ for which there is a sequence $\{(x_k,p_k)\}_{k\in \N}$ converging to $(x,p)$ in $T^*M$ such that $p_k\in \partial^-f(x_k)$ for all $k\in \N$. As shown by the following result, Lipschitz points from below belong to the domain of the limiting subdifferential.

\begin{proposition}\label{PROPLip-Lim}
Let $h$ be a smooth Riemaniann metric on $M$, $K>0$ and $x\in \mathcal{O}$ be such that $f$ admits a support function from below at $x\in M$ which is $K$-Lipschitz (with respect to $d^h$) on its domain, then there is $p\in \partial^-_Lf(x)$ with $|p|_x\leq K$.
\end{proposition}

\begin{proof}[Proof of Proposition \ref{PROPLip-Lim}]
Let $x\in \mathcal{O}$ be such that $f$ admits a support function from below $\varphi$ at $x\in M$ which is $K$-Lipschitz with respect to a smooth Riemannian metric $h$. Up to considering a chart and extending properly the restrictions of $f_{\vert \mathcal{V}}$ and $\varphi_{\vert \mathcal{V}}$ on a small open neighborhood $\mathcal{V}$ of $x\in \R^n$ to the whole $\R^n$, we may assume that we work in $\R^n$ with $\tilde{f}:\R^n \rightarrow \R$ continuous such that $\tilde{f}_{\vert \mathcal{V}}=f_{\vert \mathcal{V}}$ (so that $\partial^-_L\tilde{f}(x)=\partial^-_Lf(x)$) and a support function $\tilde{\varphi}:\R^n \rightarrow \R$ of $\tilde{f}$ (or $f$) at $x$ which is $\tilde{K}$-Lipschitz with respect to the Euclidean metric for some $\tilde{K}>K$ as close to $K$ as we want. Thus, we have by assumption
\begin{eqnarray}\label{11mai1}
\tilde{f}(x) = \tilde{\varphi}(x)  \quad \mbox{and} \quad \tilde{f}(y) - \tilde{\varphi}(y) \geq 0 \quad \forall y \in \R^n.
\end{eqnarray}
For every positive integer $k$, define the function $\psi_k :\R^n \times \R^n \rightarrow \R$ by
\[
\psi_k(y,z) := \tilde{f}(y) - \tilde{\varphi}(z) + k |y-z|^2 + |z-x|^2 \qquad \forall (y,z) \in \R^n \times \R^n
\]
and set
\[
m_k := \inf \Bigl\{ \psi_k(y,z) \, \vert \, (y,z) \in \R^n \times \R^n \Bigr\}.
\]
By $\tilde{K}$-Lipschitzness of $\tilde{\varphi}$ and (\ref{11mai1}),  we have for every $k\in \N^*$ and any $y,z\in \R^n$, 
\begin{eqnarray*}
\psi_k(y,z)  & = & \tilde{f}(y) - \tilde{\varphi}(y) + \left( \tilde{\varphi}(y) - \tilde{\varphi}(z) \right) + k |y-z|^2 + |z-x|^2 \\
& \geq & -\tilde{K} |y-z| + k |y-z|^2 +  |z-x|^2.
\end{eqnarray*}
Thus, since $m_k\leq \psi_k(x,x)=0$, we infer that for every $k\in \N^*$ the infimum in the definition of $m_k$ is attained at some $(y_k,z_k)$, {\it i.e.} $m_k = \psi_k(y_k,z_k)$, and there holds
\[
\lim_{k\rightarrow +\infty} y_k = \lim_{k\rightarrow +\infty} z_k = x. 
\]
Given $k\in \N^*$, we note that $\psi_{k}(y_k,\cdot) \geq m_k$ gives 
\[
\tilde{f}(y_k) - \tilde{\varphi}(z) + k |y_k-z|^2 + |z-x|^2 \geq \tilde{f}(y_k) - \tilde{\varphi} (z_k) + k |y_k-z_k|^2 + |z_k-x|^2
\]
for all $z\in \R^n$, which means that the function $- \tilde{\varphi}$ admits the smooth function
\[
 z \longmapsto  - \tilde{\varphi}(z_k) + k |y_k-z_k|^2 -  k |y_k-z|^2 + |y_k-x|^2 - |z-x|^2
 \]
 as a support function from below at $z_k$ and so yields 
 \[
 -2k\left( z_k-y_k\right)^* -2 \left( z_k-x\right)^* \in \partial^- (-\tilde{\varphi})(z_k).
 \]
Similarly, $\psi_{k}(\cdot,z_k) \geq m_k$ gives 
\[
 p_k:= -2k\left( y_k-z_k\right)^* \in \partial^- \tilde{f}(y_k) \qquad \forall k \in \N^*.
 \]
Since $-\tilde{\varphi}$ is $\tilde{K}$-Lipschitz with respect to the Euclidean metric, we have by Proposition \ref{PROPSubLip}
\[
\left|-2k\left( z_k-y_k\right)^* -2 \left( z_k-x\right)^* \right| \leq \tilde{K} \quad \mbox{for all } k.
\]
In conclusion, we have shown that for all $k$, we have $p_k \in  \partial^- \tilde{f}(y_k)$ with $\lim_{k\rightarrow +\infty} y_k=x$ and $|p_k| \leq \tilde{K} +o(1)$ for $k$ large. Since $\tilde{f}_{\vert \mathcal{V}}=f_{\vert \mathcal{V}}$, this implies that $p_k \in  \partial^- f(y_k)$ for $k$ large enough, which by compactness gives some $\tilde{p} \in \partial^-_L f(x)$ with $|p|\leq \tilde{K}$. We conclude by letting $\tilde{K}$ tend to $K$.  
\end{proof}

The following result is an easy consequence of Rademacher's Theorem, it shows that viscosity subdifferentials of $f$ are nonempty almost everywhere over  $\mbox{Lip}^-(f)$.

\begin{proposition}\label{PROPLip-Sub}
There is a set $\mathcal{N}\subset \mbox{Lip}^-(f)$ of Lebesgue measure zero in $M$ such that $\partial^-f(x)\neq \emptyset$ for every $x\in \mbox{Lip}^-(f) \setminus \mathcal{N}$.
\end{proposition}

\begin{proof}[Proof of Proposition \ref{PROPLip-Sub}]
Without loss of generality, up to considering charts, we may assume that $M=\R^n$. Pick a sequence $\{x_k\}_{k\in \N^*}$ which is dense in $\R^n$ and set for every $k,l \in \N^*$,
\[
A_{k,l} :=\Bigl\{ x\in  \mbox{Lip}^-(f) \cap B(x_k,1/l) \, \vert \, f(y)-f(x) \geq -l|y-x|, \, \forall y\in B(x_k,1/l)\Bigr\}.
\]
Since $\mbox{Lip}^-(f)=\cup_{l,k\in \N^*}A_{l,k}$, it is sufficient to show that for every $k,l\in \N^*$, $\partial^-f(x)\neq \emptyset$ for almost every $x\in A_{k,l}$. Fix $k,l\in \N^*$ such that $A_{k,l} \neq \emptyset$ and define the function $\phi_{k,l}:A_{k,l} \rightarrow \R$ by
\[
\phi_{k,l}(x) := \sup \Bigl\{\phi(x) \, \vert \, \phi: B(x_k,1/l)\rightarrow \R  \mbox{ is $l$-Lipschitz and } \phi\leq f \mbox{ on }  B(x_k,1/l)\Bigr\}.
\] 
By construction, $\phi_{k,l}$ is finite (because $A_{k,l} \neq \emptyset$), $l$-Lipschitz and equal to $f$ on $A_{k,l}$. By Rademacher's Theorem (see {\it e.g.} \cite{eg15,federer69}), we infer that for almost every $x\in A_{k,l}$, $\phi_{k,l}$ is differentiable at $x$ which means that $\partial^-\phi_{k,l}(x)$ is a singleton and $\phi_{k,l}$ admits a support function from below differentiable at $x$ and shows that $\partial^- f(x)$ is nonempty.
\end{proof}

\subsection{On the domain of the limiting subdifferential}\label{SECAlter}

We may wonder whether limiting subdifferentials are nonempty almost everywhere. The following result gives a positive answer in one dimension: 

\begin{proposition}\label{PROPDichotomy}
Let $a,b\in \R$ with $a<b$ and $\varphi :(a,b) \rightarrow \R$ be a continuous function. Then, for almost every $x\in (a,b)$, one of the following property is satisfied:
\begin{itemize}
\item[(i)] $\varphi$ is differentiable at $x$.
\item[(ii)] There is a sequence $\{x_k\}_{k\in \N}$ converging to $x$ such that $0\in \partial^-\varphi(x_k)$ for all $k\in \N$, in particular we have $0\in \partial^-_L \varphi(x)$.
\end{itemize}
\end{proposition}

\begin{proof}[Proof of Proposition \ref{PROPDichotomy}]
The result will be an easy corollary of the following lemma:

\begin{lemma}\label{LEMDim1}
Let $I=[c,d] \subset (a,b)$ with $c<d$ be a closed interval such that 
\[
\min \left\{\varphi(z) \, \vert \, z \in [x,y]\right\} = \min \left\{\varphi(x),\varphi(y)\right\} \qquad \forall x<y \mbox{ in } I.
\]
Then there is $e \in I$ such that $\varphi$ is monotone over $[c,e]$ and $[e,d]$. 
\end{lemma}
\begin{proof}[Proof of Lemma \ref{LEMDim1}]
Let $\bar{x},e \in \R$ be such that 
\[
\varphi(\bar{x}) = \min \left\{\varphi(z) \, \vert \, z \in [c,d]\right\} \quad \mbox{and} \quad \varphi(e)= \max  \left\{\varphi(z) \, \vert \, z \in [c,d]\right\}.
\] 
Since by assumption the minimum of $\varphi$ over $I$ is equal to the the minimum of $\varphi(c)$ and $\varphi(d)$, we may assume that $\bar{x}=c$ or $\bar{x}=d$. If $\bar{x}=c$, we claim that $\varphi$ is nondecreasing on $[c,e]$ and nonincreasing on $[e,d]$. If not, there are $x<y$ in $[c,e]$ such that $\varphi(x)>\varphi(y)$, so that
\[
\min  \left\{\varphi(z) \, \vert \, z \in [x,e]\right\} \leq \varphi(y)  < \varphi(x) \leq \varphi(e),
\]
which contradicts the assumption. The rest of the proof is left to the reader.
\end{proof}
The set $S$ of $x\in (a,b)$, for which there is a sequence $\{x_k\}_{k\in \N}$ converging to $x$ such that $0\in \partial^-\varphi(x_k)$ for all $k\in \N$, is closed in $(a,b)$. We need to show that $\varphi$ is differentiable almost everywhere in $(a,b)\setminus S$.  Suppose for contradiction that there is a set $E \subset (a,b)\setminus S$ of positive Lebesgue measure such that $\varphi$ is not differentiable at any $x\in E$ and fix $\bar{x}$ a density point of $E$. We claim that there are $c, d, e \in (a,b)\setminus S$ with $\bar{x}\in (c,d) \subset (a,b)\setminus S$ and $e\in [c,d]$ such that  $\varphi$ is monotone on $[c,e]$ and $[e,d]$. As a matter of fact, otherwise, for every $k\in \N^*$ large enough,  the assumption of Lemma \ref{LEMDim1} is not satisfied over the interval $I_k:=[\bar{x}-1/k,\bar{x}+1/k] \subset (a,b)\setminus S$ so there are $x_k, y_k, z_k\in I_k$ with $x_k< z_k< y_k$ such that
\[
\varphi(z_k) = \min \left\{\varphi(z) \, \vert \, z \in [x_k,y_k] \right\} < \min \left\{\varphi(x_k),\varphi(y_k)\right\},
\]
which means that $\varphi$ attains a local minimum at $z_k$ so that $0\in \partial^-f(z_k)$ with $z_k$ converging to $x$ as $k$ tends to $+\infty$, a contradiction. In conclusion, since monotone functions are differentiable almost everywhere, we infer that $\varphi$ is differentiable almost everywhere in a neighborhood of $\bar{x}$, which contradicts the  fact that $\bar{x}$ is a density point of $E$. 
\end{proof}

\begin{remark}\label{RemDYS}
The Denjoy-Young-Saks Theorem allows to make more precise assertion (ii). Given $\varphi$ as in Proposition \ref{PROPDichotomy}, the {\it Dini derivatives} $D^+\varphi, D_+\varphi,D^-\varphi, D_-\varphi:(a,b) \rightarrow \R\cup \{\pm \infty\}$ of $\varphi$ at $x\in (a,b)$ are defined by 
 \[
 D^+\varphi(x) = \limsup_{h\rightarrow 0^+} \frac{\varphi(x+h)-\varphi(x)}{h}, \quad D_+\varphi(x) = \liminf_{h\rightarrow 0^+} \frac{\varphi(x+h)-\varphi(x)}{h}
 \] 
\[
 D_-\varphi(x) = \liminf_{h\rightarrow 0^+} \frac{\varphi(x)-\varphi(x-h)}{h}, \quad D^-\varphi(x) = \limsup_{h\rightarrow 0^+} \frac{\varphi(x)-\varphi(x-h)}{h}.
\]
Denjoy-Young-Saks' Theorem (see \cite{hanson34} and references therein) asserts that for almost every $x\in (a,b)$, one of the following assertions holds:
\begin{itemize}
\item[(1)] $D^+\varphi(x)= D_+\varphi(x)=D^-\varphi(x)= D_-\varphi(x) \in \R$, {\it i.e.} $f$ is differentiable at $x$,
\item[(2)] $D^+f(x)=D^-f(x)=+\infty$ and $D_+f(x)=D_-f(x)=-\infty$,
\item[(3)] $D^+f(x)=+\infty, D_-f(x)=-\infty$ and $D_+f(x)=D^-f(x)\in \R$,
\item[(4)] $D^-f(x)=+\infty, D_+f(x)=-\infty$ and $D_-f(x)=D^+f(x)\in \R$.
\end{itemize}
As a consequence, we may also suppose in Proposition \ref{PROPDichotomy} (ii) that one of the above assertions (2), (3), (4) is satisfied.
 \end{remark}

Proposition \ref{PROPDichotomy} implies that the limiting subdifferential of a continuous function in dimension one is nonempty almost everywhere, we do not know if this result holds true in higher dimension (see Section \ref{SEC8jan}). 

\subsection{Projective limiting subdifferentials}

We now introduce an object that allows to capture limits of elements of $\partial^-f$ going to infinity. For every $x\in \mathcal{O}$, we call {\it projective limiting subdifferential} of $f$ at $x$, denoted by $\partial^-_{PL}f(x)$, the set of $p\in T_x^*M\setminus \{0\}$ for which there are sequences $\{(x_k,p_k)\}_{k\in \N}$ in $T^*M$ and $\{\lambda_k\}_{k\in \N}$ in $(0,+\infty)$ satisfying
\[
\lim_{k\rightarrow +\infty} \left|p_k\right|_{x_k} = +\infty, \quad \lim_{k\rightarrow +\infty} \left( x_k,\lambda_k p_k\right) = (x,p)
\quad
\mbox{and} \quad p_k \in \partial^-f(x_k) \quad \forall k \in \N.
\]
By construction, the set $\partial^-_{PL}f(x)$ is a positive cone, that is, if $p\in \partial^-_{PL}f(x)$ then $\lambda p\in \partial^-_{PL}f(x)$ for all $\lambda >0$. Proposition \ref{PROPSubLipEQ} implies the following:

\begin{proposition}\label{PROPLip-PL}
For every $x\in \mathcal{O}$, $\partial^-_{PL}f(x)=\emptyset$ if and only if $x\in \mbox{Lip}(f)$.
\end{proposition}

As we shall see in the next section, viscosity and limiting subdifferentials of pointed sub-Riemannian distances come along with minimizing normal extremals while projective limiting subdifferentials are associated with minimizing abnormal extremals.  

\section{Minimizing Sard Conjecture, subdifferentials and lipschitzness}\label{SEC3}

Throughout this section we consider a point $x\in M$ and fix a smooth Riemannian metric $h$ on $M$. Then, we denote by $d_{SR}^x:=d_{SR}(x,\cdot)$ the pointed sub-Riemannian distance and we define $f_x:M\rightarrow \R$ by 
\begin{eqnarray}\label{19oct1}
f_x(y) := \frac{1}{2} d_{SR}^x(y)^2 \qquad \forall y \in M.
\end{eqnarray}
We check easily by definition of the viscosity subdifferential that we have for every $y\in M\setminus \{x\}$ and every $p\in T_y^*M$, 
\begin{eqnarray}
p \in \partial^-d_{SR}^x(y) \quad \Longleftrightarrow \quad d_{SR}(x,y) p \in \partial^-f_x(y).
\end{eqnarray}

The aim of this section is to explain the link between subdifferentials of $f_x$ and minimizing geodesics and eventually to provide several characterization of the Minimizing Sard Conjecture.

\subsection{Abnormal and normal extremals}\label{SECExtremals}

We explain below how the notions of abnormal and normal extremals emerge in sub-Riemannian geometry and introduce several notations that will be used in the next sections, we refer the reader to \cite{riffordbook} for further details. \\

Let us consider $\bar{y}\neq x$ in $M$ and  $\bar{\gamma} : [0,1] \rightarrow M$ a minimizing geodesic from $x$ to $\bar{y}$, that is, a horizontal path that minimizes ($|v|^2=g_x(v,v)$)
\[
\mbox{energy}^g(\gamma)= \int_0^1 |\dot{\gamma}(t)|^2\, dt,
\]
among all paths $\gamma\in \Omega_x^{\Delta}$ verifying $\gamma(1)=\bar{y}$. Then, consider an orthonormal family $\mathcal{F} = \{X^1, \ldots, X^m\}$ of smooth vector fields which parametrizes $\Delta$ on an open neighborhood $\mathcal{V}\subset M$ of $\bar{\gamma}([0,1])$  and denote by $E$ the end-point mapping associated with $x$ and $\mathcal{F}$ defined by
$$
E(u) :=  \gamma^u(1) \qquad \forall u \in L^2([0,1],\R^m),
$$
where $\gamma^u$ is the curve in $\Omega_x^{\Delta}$ solution to the Cauchy problem
$$
\dot{\gamma}^u (t) = \sum_{i=1}^m u_i(t) \, X^i \left(\gamma^u(t)\right) \, \mbox{ for a.e. } t \in [0,1], \quad \gamma^u(0)=x.
$$
By construction, there is a control $\bar{u} \in L^2([0,1],\R^m)$ such that $\bar{\gamma}=\gamma^{\bar{u}}$ and $E$ is well-defined and smooth on  an open set $U\subset  L^2([0,1],\R^m)$ containing $\bar{u}$.  Then, we define the smooth mappings $C: U \rightarrow \R$ and $F: U \rightarrow M\times \R$ by  
\begin{eqnarray}\label{10oct1}
C(u):= \frac{1}{2} \| u\|_{L^2}^2 \quad \mbox{and} \quad F(u) := \left( E (u), C(u) \right)  \qquad \forall u \in U
\end{eqnarray}
and note that for every $u\in U$ we have, because $\mathcal{F}$ is orthonormal,
\[
C(u) = \frac{1}{2} \mbox{energy}^g(\gamma^u).
\]
In particular, by construction we have 
\[
C(\bar{u}) = \frac{1}{2} \mbox{energy}^g(\bar{\gamma}) = \frac{1}{2} d_{SR}(x,\bar{y})^2 
\quad
\mbox{and} \quad \left| \bar{u}(t) \right| = d_{SR}(x,\bar{y}) \quad \mbox{for a.e. } t \in [0,1].
\]
Furthermore, we observe that since $\bar{\gamma}$ has no self-intersection, we may assume by shrinking $\mathcal{V}$ if necessary,  that $\mathcal{V}$ is smoothly diffeomorphic to the open unit ball $B^n(0,1) \subset \R^n$ through a diffeomorphism $\Phi: \mathcal{V} \rightarrow B^n(0,1)$ satisfying
\begin{eqnarray}\label{ESTDiffeo}
\frac{1}{K}|p| \leq \left| p \cdot d_x\Phi \right|^*_x \leq K |p| \qquad \forall x \in \mathcal{V}, \, \forall p \in (\R^n)^*
\end{eqnarray}
for some constant $K>0$ depending on $\bar{\gamma}$, so that we may assume from now that we are in $\R^n$. 

Let us now consider a local minimizer $u\in U$ of $C$ with end-point $y:=E(u)$, that is, such that
\[
C(u) = \min \Bigl\{C(u') \, \vert \, u' \in U, \, E(u') = y \Bigr\}.
\]
By the above construction, this means that the horizontal path $\gamma^u$ minimizes the energy $\mbox{energy}^g(\gamma)$ among all paths $\gamma\in \Omega_x^{\Delta}$ sufficiently close to $\gamma^u$ verifying $\gamma(1)=y$. By the Lagrange Multiplier Theorem, there are $\bar{p}^u \in T_y^*M$ and $\bar{p}_0^u \in \{0,1\}$ with $(\bar{p}^u,\bar{p}_0^u) \neq (0,0)$ such that (we denote here the differentials of smooth function with uppercase letter "D")
\begin{eqnarray}\label{Lagrange}
\bar{p}^u \cdot D_{u} E  = \bar{p}_0^u \, D_{u} C,
\end{eqnarray}
where the differentials of $E$ and $C$ at $u$ are respectively given by 
\begin{eqnarray}\label{dE}
D_{u}E(v) = \int_0^1 S^u(1) S^u(t)^{-1} B^u(t) v(t)\, dt \qquad \forall v \in L^2( [0,1],\R^m)
\end{eqnarray}
and
\begin{eqnarray}\label{dC}
D_{u}C(v)  =   \int_0^1\langle u(t),v(t)\rangle \, dt \qquad \forall v \in L^2( [0,1],\R^m),
\end{eqnarray}
where we have defined $B^u : [0,1] \rightarrow M_{n,m}(\R)$ by
\begin{eqnarray}\label{Bu}
B^u(t) := \left(X^{1}(\gamma^{u}(t),\cdots,X^{m}(\gamma^{u}(t)  \right) \qquad \forall t \in [0,1]
\end{eqnarray}
and $S^u:[0,1] \rightarrow M_n(\R)$ as the solution to the Cauchy problem
\begin{eqnarray}\label{5avril3}
\dot{S}^u(t)=A^u(t)S^u(t) \quad \mbox{for a.e. } t \in [0,1], \quad S^u(0)=I_{n},
\end{eqnarray}
with $A^u:[0,1] \rightarrow M_n(\R)$ given by ($J_{X^i}$ is the Jacobian matrix of $X^i$)
\begin{eqnarray}\label{5avril4}
A^u(t) := \sum_{i=1}^m u_{i}(t) J_{X^{i}}\left(\gamma^u(t)\right) \qquad \mbox{for a.e. } t \in [0,1].
\end{eqnarray} 
Then we define the extremal $\psi^{u,\bar{p}^u} : [0,1] \rightarrow T^*M$ by 
\begin{eqnarray}\label{formulaextremal}
\psi^{u,\bar{p}^u}(t) := \left(\gamma^u(t),p^{u,\bar{p}^u}(t)\right) :=  \left(\gamma_u(t), \bar{p}^u \cdot S^u(1)S^u(t)^{-1}\right)  \qquad \forall t \in [0,1].
\end{eqnarray}
By construction, $\psi^{u,\bar{p}^u}$ is a lift of $\gamma^u$ verifying $\psi^{u,\bar{p}^u}(1)=(y,\bar{p}^u)$ and we have  the following result:

\begin{proposition}\label{PROPp0}
Depending on the value of $\bar{p}_0^u \in \{0,1\}$ in (\ref{Lagrange}), we have:
\begin{itemize}
\item[(i)] If $\bar{p}_0^u=0$, then $\psi^{u,\bar{p}^u}$ is an abnormal extremal (and $\gamma^{u}$ is a singular).
\item[(ii)] If $\bar{p}_0^u=1$, then  $\psi^{u,\bar{p}^u}$ is a normal extremal and we have 
\[
u(t) = B^u(t)^{*} p(t)^* \qquad \mbox{for a.e. } t \in [0,1].
\]
\end{itemize}
\end{proposition}

This being said, we explain in the next section the link between subdifferentials of $f_x$ and  extremals. We keep the same notations as above. 

\subsection{Subdifferentials and extremals}

The key result to connect subdifferentials of $f_x$ to normal extremals, due to Trélat and the author \cite{rt05}, is the following:
 
\begin{proposition}\label{PROPsub1}
Let $y\in M\setminus \{x\}$, $p\in \partial^- f_x(y)$ and $u\in U$ be a local minimizer of $C$ with end-point $E(u)=y$, then we have 
\begin{eqnarray}\label{26sept1}
p\cdot D_{u}E = D_{u}C.
\end{eqnarray}
In particular, there is a unique minimizing geodesic from $x$ to $y$, it is given by the projection of the normal extremal $\psi:[0,1] \rightarrow T^*M$ satisfying $\psi(1)=(y,p)$. Moreover, if $\psi(0)$ is not a critical point of the exponential mapping $\exp_x$, then $y$ does not belong to $\mbox{Abn}^{min}(x)$.
\end{proposition}

\begin{proof}[Proof of Proposition \ref{PROPsub1}]
Let $y\in M\setminus \{x\}$, $p\in \partial^- f_x(y)$, and $u\in U$ be a local minimizer of $C$ with end-point $E(u)=y$ and let $\varphi:M\rightarrow \R$ be a support function from below of class $C^1$ with $d\varphi(x)=p$. By assumption, we have 
\[
C(u) = f_x(y)=\varphi(y) \quad \mbox{and} \quad C(u') \geq f_x\left(E(u')\right) \geq \varphi \left(E(u')\right) \quad \forall u'\in U,
\]
which means that the function $C-\varphi\circ E$ attains a local minimum at $u$. Hence we have (\ref{26sept1}) and by Proposition \ref{PROPp0} we infer that $\gamma^u$ has to be the projection of the normal extremal $\psi^{u,\bar{p}^{u}}$ with $\bar{p}^{u}:=d_{y}\varphi =p$. If $\psi(0)$ is not a critical point of $\exp_x$, then the horizontal path $\gamma^u$ is not singular and and there is no other minimizing geodesic from $x$ to $y$, we infer that  $y\notin \mbox{Abn}^{min}(x)$.
\end{proof}

\begin{remark}\label{REM28oct}
We note that if $y\in M\setminus \{x\}$, $p\in \partial^-_P f_x(y)$ and $u\in U$ is a local minimizer of $C$ with end-point $E(u)=y$, then we have (\ref{26sept1}) and moreover since the function $C-\varphi\circ E$ with a local minimum at $u$ is of class $C^2$ (where $\varphi:M\rightarrow \R$ is a support function from below of class $C^2$ with $d\varphi(x)=p$), we have
\begin{eqnarray*}
 D^2_uC (v) - p\cdot D^2_uE(v) \geq 0 \qquad \forall v \in \mbox{\rm Ker}(D_uE),
\end{eqnarray*}
where $D^2_uC$ and $D^2_uE$ stand for the quadratic forms defined respectively by the second order differentials of $C$ and $E$ at $u$.
\end{remark}

As a consequence, Proposition \ref{PROPsub1} allows to associate to each $p\in \partial^-_L f_x(y)$ a normal extremal whose projection is minimizing.
 
\begin{proposition}\label{PROPsub2}
Let $y\in M\setminus \{x\}$ and $p\in \partial^-_L f_x(y)$, then the projection of the normal extremal $\psi:[0,1] \rightarrow T^*M$ satisfying $\psi(1)=(y,p)$ is a minimizing geodesic from $x$ to $y$.
\end{proposition}

\begin{proof}[Proof of Proposition \ref{PROPsub2}]
Let $y\in M\setminus \{x\}$, $p\in \partial^-_L f_x(y)$ and $\{(y_k,p_k)\}_{k\in \N}$ be a sequence converging to $(y,p)$ in $T^*M$ such that $p_k\in \partial^-u(y_k)$ for all $k\in \N$. By Proposition \ref{PROPsub1}, for every $k\in \N$, the projection $\gamma_k:[0,1] \rightarrow M$ of the normal extremal $\psi_k:[0,1] \rightarrow T^*M$ verifying $\psi_k(1)=(y_k,p_k)$ is a minimizing geodesic from $x$ to $y_k$. By regularity of the Hamiltonian flow, the sequence $\{\psi_k\}_{k\in \N}$ converges to the normal extremal $\psi:[0,1] \rightarrow T^*M$ verifying $\psi(1)=(y,p)$ and its projection is minimizing as a limit of minimizing geodesics from $x$ to $y_k$ which converges to $y$.
\end{proof}

If however we consider a co-vector $p$ in $\partial^-_{PL} f_x(y)$ then Proposition \ref{PROPsub1} allows to obtain a minimizing singular geodesic associated with $p$, more precisely we have:

\begin{proposition}\label{PROPsub3}
Let $y\in M\setminus \{x\}$ and $p\in \partial^-_{PL} f_x(y)$, then there is a singular minimizing geodesic from $x$ to $y$. Moreover, for all sequences $\{(y_k,p_k)\}_{k\in \N}$ in $T^*M$ and $\{\lambda_k\}_{k\in \N}$ in $(0,+\infty)$ satisfying
\[
\lim_{k\rightarrow +\infty} \left|p_k\right|_{y_k} = +\infty, \quad \lim_{k\rightarrow +\infty} \left( y_k,\lambda_k p_k\right) = (y,p)
\quad
\mbox{and} \quad p_k \in \partial^-f(y_k) \quad \forall k \in \N,
\]
any uniformly convergent subsequence of the sequence of minimizing geodesics $\{\gamma_k\}_{k\in \N}$ given by the projections of the normal extremals  $\psi_k:[0,1] \rightarrow T^*M$ verifying $\psi_k(1)=(y_k,p_k)$, converges to a singular minimizing geodesic admitting an abnormal extremal $\psi:[0,1] \rightarrow T^*M$ such that $\psi(1)=(y,p)$.
\end{proposition}

\begin{proof}[Proof of Proposition \ref{PROPsub3}]
Let $y\in M\setminus \{x\}$, $p\in \partial^-_{PL} f_x(y)$ and two sequences $\{(y_k,p_k)\}_{k\in \N}$ in $T^*M$ and $\{\lambda_k\}_{k\in \N}$ in $(0,+\infty)$ satisfying the properties given in the statement. Assume that some subsequence $\{\gamma_{k_l}\}_{l\in \N}$ of  the sequence $\{\gamma_k\}_{k\in \N}$ given by the projections of the normal extremals  $\psi_k:[0,1] \rightarrow T^*M$ verifying $\psi_k(1)=(y_k,p_k)$ converges uniformly (in $W^{1,2}([0,1],M)$) to some minimizing geodesic $\bar{\gamma}$. Using the notations of the previous section, we write $\gamma=\gamma^{\bar{u}}$ with $\bar{u} \in L^2( [0,1],\R^m)$ and notice that for an index $l$ large enough, $l\geq L$,  we can write $\gamma_{k_l}$ as $\gamma^{u_l}$ for some $u_l\in U$ where the subsequence $\{u_l\}_{l\geq L}$ tends to $\bar{u}$ in $L^2( [0,1],\R^m)$ as $l$ tends to $\infty$. Then we have for all index $l\geq L$,
\[
p_{k_l}\cdot D_{u_l}E = D_{u_l}C
\]
which gives
\[
\lambda_{k_l} p_{k_l}\cdot D_{u_l}E = \lambda_{k_l} D_{u_l}C \qquad \forall l \geq L,
\]
where
\[
\lim_{l\rightarrow +\infty} \lambda_{k_l} p_{k_l} = p \quad \mbox{and} \quad \lim_{l\rightarrow +\infty} \lambda_{k_l} =0.
\]
By passing to the limit, we obtain $p \cdot D_{\bar{u}}E=0$ which, by Proposition \ref{PROPp0}, shows that $\gamma=\gamma^{\bar{u}}$ is a singular minimizing geodesic admitting an abnormal extremal $\psi:[0,1] \rightarrow T^*M$ such that $\psi(1)=(y,p)$.

\end{proof}

\subsection{Non-Lipschitz points admit Goh abnormals}\label{SECGoh}

For every $x\in M$, $\mbox{Goh-Abn}^{min}(x)$ stands for the set of $y\in M$ for which there is a minimizing path $\gamma \in \Omega^{\Delta}_x$ joining $x$ to $y$ which is singular and admits an abnormal lift $\psi:[0,1] \rightarrow \Delta^{\perp}$ satisfying the Goh condition
\begin{eqnarray}\label{Goh26sept}
\psi(t) \cdot [\Delta,\Delta](\gamma(t)) = 0 \qquad \forall t \in [0,1].
\end{eqnarray}
By construction, $\mbox{Goh-Abn}^{min}(x)$ is a closed set satisfying 
\[
\mbox{Goh-Abn}^{min}(x) \subset \mbox{Abn}^{min}(x).
\]
Agrachev and Sarychev \cite{as99} showed that any singular minimizing geodesic with no normal extremal lifts does admit an abnormal lift satisfying the Goh condition and Agrachev and Lee \cite{al09} proved that the absence of Goh abnormal extremals imply lipschitzness properties for $f_x$. In some sense, the following result combines the two results. 

\begin{proposition}\label{PROPGohLip}
For every $x\in M$, we have $M\setminus \mbox{Lip}(d_x) \subset \mbox{Goh-Abn}^{min}(x)$. Moreover, for any $y\in M \setminus \mbox{Lip}(f_x) $, any $p \in \partial^-_{PL}f_x(y)$ and any sequences $\{(y_k,p_k)\}_{k\in \N}$ in $T^*M$ and $\{\lambda_k\}_{k\in \N}$ in $(0,+\infty)$ satisfying
\[
\lim_{k\rightarrow +\infty} \left|p_k\right|_{y_k} = +\infty, \quad \lim_{k\rightarrow +\infty} \left( y_k,\lambda_k p_k\right) = (y,p)
\quad \mbox{and} \quad p_k \in \partial^-f_x(y_k) \quad \forall k \in \N,
\]
any minimizing horizontal path from $x$ to $y$, obtained as the uniform limit of a subsequence of the sequence of minimizing geodesics $\{\gamma_k\}_{k\in \N}$ given by the projections of the normal extremals  $\psi_k:[0,1] \rightarrow T^*M$ verifying $\psi_k(1)=(y_k,p_k)$, does admit an abnormal extremal $\psi:[0,1] \rightarrow T^*M$ with $\psi(1)=(y,p)$ which satisfies the Goh condition.
\end{proposition}

Let us fix $\bar{y}\neq x$ in $M$ and a minimizing geodesic $\bar{\gamma} : [0,1] \rightarrow M$ from $x$ to $\bar{y}$ and by considering all notations introduced in Section \ref{SECExtremals} (especially (\ref{10oct1})) write $\bar{\gamma}=\gamma^{\bar{u}}$ for some $\bar{u}\in U$. Proposition \ref{PROPGohLip} will be an easy consequence of the following result (we refer the reader to Appendix \ref{SECSecondOrder} for further details on negative indices of quadratic forms ($ \mbox{ind}_-$)): 

\begin{proposition}\label{PROPGohNormal}
Assume that $\|\bar{u}\|_{L^{\infty}}<L$ for some $L>0$. Then for every $\kappa>0$ and every $N\in \N$, there are $\rho, \Lambda>0$ such that the following property is satisfied: For every $u\in U$ such that
\begin{eqnarray}\label{HypGohNormal1}
\left\|u-\bar{u}\right\|_{L^2}< \rho, \quad \left\|u\right\|_{L^{\infty}}< 2L,
\end{eqnarray}
and every $\bar{p}^u \in T_y^*M$ and $\bar{p}_0^u \in \R$, with $y:=E(u)$ and $(\bar{p}^u,\bar{p}_0^u) \neq (0,0)$, satisfying 
\begin{eqnarray}\label{Lagrange2}
\bar{p}^u \cdot D_{u} E  = \bar{p}_0^u \, D_{u} C,
\end{eqnarray}
 together with
\begin{eqnarray}\label{HypGohNormal2}
 \left| \bar{p}^u\right|_y \geq \Lambda \, \left| \bar{p}_0^{u}\right|
\end{eqnarray}
and
\begin{eqnarray}\label{HypGohNormal3}
 \mbox{\em ind}_- \left( \left(\bar{p}^u,-\bar{p}_0^u\right) \cdot \left( D^2_{u} F \right)_{\vert \mbox{\em Ker} (D_{u}F)}  \right) < N,
\end{eqnarray}
the extremal $\psi^{u,\bar{p}^u} = (\gamma^u(t),p^{u,\bar{p}^u}(t)): [0,1] \rightarrow T^*M$ defined by (\ref{formulaextremal})  satisfies 
\begin{eqnarray}\label{AlmostGoh}
\left| p^{u,\bar{p}^u}(t) \cdot \left[ X^i,X^j\right] \left(\gamma^u(t)\right)\right|  \leq \kappa \left| \bar{p}^u\right|_y \qquad \forall t \in [0,1], \, \forall i,j = 1, \ldots, m.
\end{eqnarray}
\end{proposition} 
 
\begin{proof}[Proof of Proposition \ref{PROPGohNormal}]
The formula of $D_uE$ has been recalled in (\ref{dE}) and the quadratics forms defined respectively by the second order differentials of $C$ and $E$ are given by  
\begin{eqnarray}\label{D2uC}
D_{u}^2C(v)=  \|v\|_{L^2}^2 \qquad \forall v \in L^2( [0,1],\R^m)
\end{eqnarray}
and
\begin{eqnarray}\label{d2uE}
D_{u}^2 E(v) =  2 \int_0^1 S^u(1) S^u(t)^{-1} \left[ A^u_v(t)+D^u_v(t)\right] \,   dt \qquad \forall v \in L^2( [0,1],\R^m),
\end{eqnarray}
where for every $v\in L^2([0,1],\R^m)$ the functions $A^u_v, D^u_v:[0,1] \rightarrow M_n(\R)$ are defined for almost every $t\in [0,1]$ by
\begin{eqnarray}\label{5avril6}
A^u_v(t) :=  \sum_{i=1}^m v_i(t) D_{\gamma^u(t)} X^i  \left( \delta_v^1(t)\right), \quad D_v^u(t) := \frac{1}{2} \sum_{i=1}^m u_i(t) D^2 _{\gamma^u(t)} X^i  \left( \delta_v^1(t)\right)
\end{eqnarray}
with
\begin{eqnarray}\label{5avril7}
\delta_v^1(t) := \int_0^t S^u(t) S^u(s)^{-1} B^u(s) v(s)\, ds.
\end{eqnarray}
We refer the reader to \cite{riffordbook} for the above formulas. The proof of Proposition \ref{PROPGohNormal} will follow from the following lemma (compare \cite[Lemma 2.21 p. 63]{riffordbook}): 

\begin{lemma}\label{LEM8sept}
There are $\rho, K>0$ such that for any $u\in U$ with $\|u-\bar{u}\|_{L^2}< \rho$, $\|u\|_{L^{\infty}}< 2L$ and any $\bar{t}, \delta>0$ with $[\bar{t},\bar{t}+\delta]\subset [0,1]$, the following property holds: For every $v\in \mbox{\em Ker} (D_{u} E)$ with $\mbox{\em Supp} (v)\in  [\bar{t},\bar{t}+\delta]$ and every $\bar{p}^u\in T_y^*M$, we have
\begin{eqnarray}\label{8sept3}
\left| \bar{p}^u \cdot D_{u}^2 E (v)  - Q^{u,\bar{p}^u}_{\bar{t},\delta} (v) \right| \leq K \, \|v\|_{L^2}^2 \, \left|\bar{p}^u\right|_{y} \,  \delta^2, 
\end{eqnarray}
where $Q^{u,\bar{p}^u}_{\bar{t},\delta}: L^2 \left( [0,1],\R^m\right) \rightarrow \R^n$ is the quadratic form defined by  
\begin{eqnarray}\label{8septQ}
Q^{u,\bar{p}^u}_{\bar{t},\delta}(v) :=  \int_{\bar{t}}^{\bar{t}+\delta} \int_{\bar{t}}^t        \langle v(s), M_{\bar{t}}^{u,\bar{p}^u}v(t)\rangle \, ds \, dt \qquad \forall v\in L^2 \left( [0,1],\R^m\right)
\end{eqnarray}
with
\begin{eqnarray}\label{27avril1}
\left(  M_{\bar{t}}^{u,\bar{p}^u}\right)_{i,j} = 2 \, p^{u,\bar{p}^u} (\bar{t}) \cdot   D_{\gamma^u(\bar{t})} X^i \left( X^j\bigl(\gamma^u(\bar{t}) \bigr)\right) \qquad \forall i,j =1, \ldots,m.
\end{eqnarray}
\end{lemma}

\begin{proof}[Proof of Lemma \ref{LEM8sept}]
Let $\bar{t}, \delta>0$ with $[\bar{t},\bar{t}+\delta]\subset [0,1]$ and $v\in \mbox{Ker} (D_{u} E)$ with $\mbox{Supp} (v)\in  [\bar{t},\bar{t}+\delta]$ be fixed. By (\ref{formulaextremal}) and (\ref{d2uE}), we have
\begin{eqnarray}\label{6avril33}
  \bar{p}^u \cdot \left( D^2_{u} E \right)_{\vert \mbox{Ker} (D_{u} E)}  (v) = 2  \int_0^1 p^{u,\bar{p}^u}(t) \cdot \left[ A^u_v(t)+ D^u_v(t)\right] \,   dt.
\end{eqnarray}
Since $v\in \mbox{Ker} (D_{u} E)$ and $\mbox{Supp} (v)\in  [\bar{t},\bar{t}+\delta]$, we have $\delta_v^1(t)=0$ for every $t\in [0,\bar{t}] \cup [\bar{t}+\delta,1]$ and by Cauchy-Schwarz's inequality, we have for every $t\in [\bar{t},\bar{t}+\delta]$,
\begin{eqnarray*}
\left| \bar{\delta}_v^1(t) \right| \leq  \sup_{s \in [0,1]} \Bigl\{ \bigl\|  S^u(t)S^u(s)^{-1}B^u(s)     \bigr\| \Bigr\} \,  \sqrt{t-\bar{t}} \, \|v\|_{L^2} \leq K_1  \,  \|v\|_{L^2} \, \sqrt{\delta},
\end{eqnarray*}
where $K_1$ is a constant depending only upon the sizes of $S^u, (S^u)^{-1}, B^u$ for $u$ close to $\bar{u}$. Then we have 
 $$
A^u_v(t)= D^u_v(t)=0 \qquad \forall t \in  t \in \bigl[0,\bar{t}\bigr] \cup \bigl[\bar{t}+\delta,1\bigr],
 $$
 and
 $$
 \bigl| D^u_v (t) \bigr| \leq K_2  \,  \|v\|_{L^2}^2 \, \bigl\| u\bigr\|_{L^{\infty}} \, \delta \qquad \forall t\in \bigl[\bar{t},\bar{t}+\delta\bigr],
 $$
 which gives
\begin{eqnarray}\label{6avril1}
 \left| \int_0^1\bar{p}^{u,\bar{p}^u}(t) \cdot D^u_v(t)\, dt  \right| \leq K_3 \, \|v\|_{L^2}^2  \left|\bar{p}^u\right|_y \, \delta^{2} ,
\end{eqnarray}
where $K_2, K_3$ are some constants depending on $K_1$, on the size of the $D^2X^j$'s and $L$ (remember that $\|u\|_{L^{\infty}}<2L$).  Note that since we can write for every $t \in[\bar{t},\bar{t}+\delta]$,
\begin{eqnarray*}
& \quad & \delta_v^1 (t) -\int_{\bar{t}}^t \sum_{j=1}^m v_j(s) X^j  \bigl(\gamma^u(\bar{t})\bigr) \, ds \\
& = & \delta_v^1 (t) -\int_{\bar{t}}^t \sum_{j=1}^m v_j(s) X^j  \bigl(\gamma^u(s)\bigr) \, ds + \int_{\bar{t}}^t \sum_{j=1}^m v_j(s) \left[ X^j  \bigl(\gamma^u(s)\bigr) - X^j  \bigl(\gamma^u(\bar{t})\bigr)\right]  \, ds \\
& = & \int_0^t S^u(t) S^u(s)^{-1} B^u(s) v(s) - B^u(s) v(s) \, ds  \\
& \quad & \qquad \qquad \qquad \qquad \qquad \qquad + \int_{\bar{t}}^t \sum_{j=1}^m v_j(s) \left[ X^j  \bigl(\gamma^u(s)\bigr) - X^j  \bigl(\gamma^u(\bar{t})\bigr)\right]  \, ds \\
& = & \int_{\bar{t}}^t \left( S^u(t) -S^u(s) \right) S^u(s)^{-1} B^u(s) v(s) \, ds  \\
& \quad & \qquad \qquad \qquad \qquad \qquad \qquad + \int_{\bar{t}}^t \sum_{j=1}^m v_j(s) \left[ X^j  \bigl(\gamma^u(s)\bigr) - X^j  \bigl(\gamma^u(\bar{t})\bigr)\right]  \, ds,
 \end{eqnarray*}
we have (since $\|u\|_{L^{\infty}}< 2L$, $S^u$ and $\gamma^u$ are both Lipschitz)
\begin{multline}\label{10sept4}
\left| \delta_v^1 (t) -\int_{\bar{t}}^t \sum_{j=1}^m v_j(s) X^j \bigl(\gamma^u(\bar{t})\bigr) \,ds \right| \leq K_4  \, \|v\|_{L^2} \, \delta^{\frac{3}{2}}, \quad \forall   t \in \bigl[\bar{t},\bar{t}+\delta\bigr],
\end{multline} 
where $K_4$ is a constant depending only upon the sizes of $S^u, (S^u)^{-1}, B^u$ (for $u$ close to $\bar{u}$), upon the Lipschitz constants of the $X^j$'s in a neighborhood of the curve $\bar{\gamma}$ and upon $L$. Moreover, by noting that $Q^{u,\bar{p}^u}_{\bar{t},\delta}$ satisfies for every $v\in L^2([0,1],\R^m)$ (by (\ref{8septQ})-(\ref{27avril1})),
\[
Q^{u,\bar{p}^u}_{\bar{t},\delta} = 2 \int_{\bar{t}}^{\bar{t}+\delta} p^{u,\bar{p}^u}(t) \cdot  \sum_{i=1}^m v_i(t) D_{\gamma^u(\bar{t})} X^i  \left( \int_{\bar{t}}^t \sum_{j=1}^m v_j(s) X^j(\gamma^u(\bar{t})) \, ds \right) \, dt
\]
the first equality in (\ref{5avril6}) gives
\begin{eqnarray*}
& \quad & 2  \int_{0}^{1} p^{u,\bar{p}^u}(t) \cdot A^u_v(t)\, dt - Q^{u,\bar{p}^u}_{\bar{t},\delta}(v) \\
& = & 2 \int_{\bar{t}}^{\bar{t}+\delta} p^{u,\bar{p}^u}(t) \cdot  \left( \sum_{i=1}^m v_i(t) D_{\gamma^u(t)} X^i \left[\delta_v^1(t)\right] \right. \\
& \quad & \qquad \qquad \qquad \qquad \left. - \sum_{i=1}^m v_i(t) D_{\gamma^u(\bar{t})} X^i  \left[ \int_{\bar{t}}^t \sum_{j=1}^m v_j(s) X^j\bigl(\gamma^u(\bar{t}) \bigr) \, ds \right] \right) \, dt \\
& = & 2 \int_{\bar{t}}^{\bar{t}+\delta} p^{u,\bar{p}^u}(t) \cdot  \left( \sum_{i=1}^m v_i(t) D_{\bar{\gamma}(t)} X^i  \right) \left[ \delta_v^1(t)  - \int_{\bar{t}}^t \sum_{j=1}^m v_j(s) X^j\bigl(\gamma^u(\bar{t}) \bigr) \, ds \right] \, dt.
\end{eqnarray*}
By (\ref{10sept4}), we infer that 
$$
\left| 2  \int_{0}^{1} p^{u,\bar{p}^u}(t) \cdot A^u_v(t)\, dt - Q^{u,\bar{p}^u}_{\bar{t},\delta}(v) \right|  \leq K_5 \, \|v\|_{L^2}^2 \, \left|\bar{p}^u\right|_{y} \, \delta^{2},
$$
for some constant $K_5$ depending on the datas, which by (\ref{6avril1}) gives 
$$
\left| 2  \int_0^1 p^{u,\bar{p}^u}(t) \cdot \left[ A^u_v(t)+ D^u_v(t)\right] \,   dt  - Q^{u,\bar{p}^u}_{\bar{t},\delta}(v) \right| \leq K_6  \, \|v\|_{L^2}^2 \, \left|\bar{p}^u\right|_{y} \,  \delta^{2},
$$
for some constant $K_6$ depending on the datas. We conclude easily by (\ref{6avril33}).
\end{proof}

Returning to the proof of Proposition \ref{PROPGohNormal}, we fix $\kappa>0$, $N\in \N$, $u\in U$ with $\|u-\bar{u}\|_{L^2}< \rho$, $\|u\|_{L^{\infty}}< 2L$ and we suppose for contradiction that there are $\bar{t} \in (0,1)$ and $\bar{i}\neq \bar{j} \in \{1,\cdots ,m\}$ such that 
\begin{eqnarray}\label{27oct1}
P_{\bar{i},\bar{j}} (\bar{t}) & := &  \frac{ p^{u,\bar{p}^{u}}( \bar{t})}{|\bar{p}^u|_{y}} \cdot \left[X^{\bar{i}},X^{\bar{j}}\right] \bigl(\gamma^u(\bar{t}) \bigr) > \kappa.
\end{eqnarray}
Let us consider $\delta>0$ with $[\bar{t},\bar{t}+\delta] \subset [0,1]$ to be chosen later, we note that we have for every $v \in L^2\bigl([0,1],\R^m\bigr)$,
\begin{eqnarray}\label{EQ27avril}
Q^{u,\bar{p}^u}_{\bar{t},\delta}(v) & = & \int_{\bar{t}}^{\bar{t}+\delta} \int_{\bar{t}}^t   \langle v(s), M_{\bar{t}}^{u,\bar{p}^u}v(t)\rangle \, ds \, dt \nonumber\\
& = & \int_{\bar{t}}^{\bar{t}+\delta} \langle w^v(t), M_{\bar{t}}^{u,\bar{p}^u}v(t)\rangle \, dt  \nonumber\\
& = & \int_{\bar{t}}^{\bar{t}+\delta} \sum_{i,j=1}^m w^v_i(t) \left( M_{\bar{t}}^{u,\bar{p}^u}\right)_{i,j} v_j(t) \, dt,
\end{eqnarray}
with
\[
w^v(t) := \int_{\bar{t}}^t v(s) \, ds \qquad \forall t \in [\bar{t},\bar{t}+\delta].
\]
Set $\bar{N}:=N+n+1$  and denote by $L= L_{\bar{t},\delta}$ the vector space in $L^2\bigl( [0,1],\R^m\bigr)$ of all controls $v$ for which there is a sequence $\{a_1, \ldots, a_{\bar{N}}\}$ such that 
\[
\left\{
\begin{array}{rcl}
v_{\bar{i}}(t) & = & \sum_{k=1}^{\bar{N}} a_k \cos \left(k  \frac{(t-\bar{t}) 2\pi}{\delta}\right) \\
 v_{\bar{j}}(t) & = & \sum_{k=1}^{\bar{N}} a_k \sin \left(k  \frac{(t-\bar{t}) 2\pi}{\delta}\right)
 \end{array}
 \right.
 \qquad \forall t \in \bigl[ \bar{t},\bar{t}+\delta\bigr],
\]
\[
\mbox{and} \quad v_i(t) =0, \quad \forall i\neq \bar{i},\bar{j} \qquad \forall t \in [0,1].
\]
Then, we have for every $v\in L$,
\[
\left\{
\begin{array}{rcl}
w^v_{\bar{i}}(t) & = & \frac{\delta}{2\pi} \, \sum_{k=1}^{\bar{N}} \frac{a_k}{k}  \sin \left(k  \frac{(t-\bar{t}) 2\pi}{\delta}\right) \\
w^v_{\bar{j}}(t) & = &  \frac{\delta}{2\pi} \,  \sum_{k=1}^{\bar{N}} \frac{a_k}{k}   \Bigl(1 -   \cos \left(k  \frac{(t-\bar{t}) 2\pi}{\delta}\right) \Bigr),
\end{array}
 \right.
 \qquad \forall t \in \bigl[ \bar{t},\bar{t}+\delta\bigr],
\]
\[
\mbox{and} \quad w^v_i(t) =0, \quad \forall i\neq \bar{i},\bar{j} \qquad \forall t \in [0,1],
\]
which gives
 \[
\int_{\bar{t}}^{\bar{t}+\delta} w_{\bar{i}}^v(t) v_{\bar{j}}(t) \, dt = \sum_{k=1}^{\bar{N}} \frac{\delta^2 a_k^2}{4 \pi k}, \quad 
\int_{\bar{t}}^{\bar{t}+\delta} w_{\bar{j}}^v(t) v_{\bar{i}}(t) \, dt = - \sum_{k=1}^{\bar{N}} \frac{\delta^2 a_k^2}{4 \pi k}
\]
and
\[
\int_{\bar{t}}^{\bar{t}+\delta}  w_{\bar{i}}^v (t) v_{\bar{i}}(t) \, dt = \int_{\bar{t}}^{\bar{t}+\delta}  w_{\bar{j}}^v (t)  v_{\bar{j}}(t) \, dt=0.
\]
In conclusion, by (\ref{EQ27avril}) and by noting that
\[
P_{\bar{i},\bar{j}} (\bar{t})  =  \frac{1}{2|\bar{p}^u|_{y}} \cdot \left( \left(  M_{\bar{t}}^{u,\bar{p}^u}\right)_{\bar{j},\bar{i}} -  \left(  M_{\bar{t}}^{u,\bar{p}^u}\right)_{\bar{i},\bar{j}} \right)
\]
we infer that
\begin{eqnarray}\label{1oct1}
Q^{u,\bar{p}^u}_{\bar{t},\delta}(v)
& = &  \int_{\bar{t}}^{\bar{t}+\delta}  w^v_{\bar{i}}(t) \left( M_{\bar{t}}^{u,\bar{p}^u}\right)_{\bar{i},\bar{j}} v_{\bar{j}}(t) +  w^v_{\bar{j}}(t) \left( M_{\bar{t}}^{u,\bar{p}^u}\right)_{\bar{j},\bar{i}} v_{\bar{i}}(t) \, dt\nonumber\\
 &  = &  - \sum_{k=1}^{\bar{N}} \frac{\delta^2 a_k^2|\bar{p}^u|_{y}}{2 \pi k}  P_{\bar{i},\bar{j}} (\bar{t}),
\end{eqnarray}
where we have
\begin{eqnarray}\label{1oct2}
\|v\|_{L^2}^2 =  \delta  \sum_{k=1}^{\bar{N}} a_k^2.
\end{eqnarray}
Let us now distinguish between the cases $\bar{p}_0^u=0$ and $\bar{p}_0^u\neq0$.\\

\noindent Case 1: $\bar{p}_0^u=0$.\\
If $\delta>0$ is small enough then (\ref{8sept3}), (\ref{1oct1}) and (\ref{1oct2}) imply that 
\[
 \bar{p}^u \cdot D_{u}^2 E (v) < 0 \qquad \forall v \in L\setminus\{0\},
\]
where the vector space $L\subset L^2\bigl([0,1],\R^m\bigr)$ has dimension $\bar{N}=N+n+1$. Since $\bar{p}_0^{u}=0$ and $\mbox{codim}(\mbox{Ker} (D_{u}E))\leq n$, this contradicts (\ref{HypGohNormal3}). \\

\noindent Case 2: $\bar{p}_0^{u}\neq 0$.\\
By (\ref{27oct1}), (\ref{1oct1}) and (\ref{1oct2}) and by noticing that
\[
\sum_{k=1}^{\bar{N}} a_k^2 \leq \bar{N} \sum_{k=1}^{\bar{N}} \frac{a_k^2}{k}, 
\]
we note that we have for every $v\in L\setminus \{0\}$,
\[
\frac{Q^{u,\bar{p}^u}_{\bar{t},\delta}(v) - \bar{p}_0^u  \|v\|_{L^2}^2}{\|v\|_{L^2}^2 |\bar{p}^u|_{y} \delta^2} = \frac{- P_{\bar{i},\bar{j}}  \bigl( \bar{t}\bigr) }{2\pi}    \frac{\sum_{k=1}^{\bar{N}} \frac{a_k^2}{k}}{ \delta  \sum_{k=1}^{\bar{N}} a_k^2} - \frac{\bar{p}_0^u}{|\bar{p}^u|_{y} \delta^2} \leq -\frac{\kappa}{2\pi\bar{N}\delta} + \frac{\left|\bar{p}_0^u\right|}{|\bar{p}^u|_{y} \delta^2}.
\]
Take $\delta>0$ small such that
\[
- \frac{\kappa}{2\pi \bar{N}\delta} +K <0
\]
and suppose that 
\[
\frac{\left|\bar{p}_0^u\right|}{|\bar{p}^u|_{y} \delta^2}< \left| - \frac{\kappa }{2\pi \bar{N}\delta} +K\right| \quad \Leftrightarrow \quad |\bar{p}^u|_y > \left| K\delta^2 - \frac{\kappa \delta}{2\pi \bar{N}}  \right|^{-1} \, \left| \bar{p}_0^u\right|.
\]
Then by  (\ref{8sept3}) and the above inequality, since $D^2_uC=\|\cdot\|_{L^2}$, we infer that
\[
 \bar{p}^u \cdot D_{u}^2 E (v) - \bar{p}_0^u D_u^2C(v)< 0 \qquad \forall v \in L\setminus\{0\},
\]
where the vector space $L\subset L^2\bigl([0,1],\R^m\bigr)$ has dimension $\bar{N}=N+n+1$. Since $\mbox{codim}(\mbox{Ker} (D_{u}F))\leq n+1$, this contradicts (\ref{HypGohNormal3}).
\end{proof}

\begin{remark}\label{20oct1}
We note that in the case where $u=\bar{u}$ and $\bar{p}_0^u=0$, Proposition \ref{PROPGohNormal} asserts that if we have 
\begin{eqnarray*}
 \mbox{\em ind}_- \left( \bar{p}^{\bar{u}} \cdot \left( D^2_{\bar{u}} E \right)_{\vert \mbox{\em Ker} (D_{\bar{u}}E)}  \right) < +\infty
\end{eqnarray*}
for some $ \bar{p}^{\bar{u}}\neq 0$ such that $\bar{p}^{\bar{u}} \cdot D_{\bar{u}} E  =0$, that is $\bar{p}^{\bar{u}}\in \mbox{Im}(D_{\bar{u}}E)^{\perp}$, then the corresponding abnormal extremal $\psi^{\bar{u},\bar{p}^{\bar{u}}}$ satisfies the Goh condition. This is the classical Goh Condition for Abnormal Geodesics due to  Agrachev and Sarychev, see \cite[Proposition 3.6 p. 389]{as99} and \cite[Theorem 2.20 p. 61]{riffordbook}.
\end{remark}

\begin{proof}[Proof of Proposition \ref{PROPGohLip}]
Let $x\in M$, $y\in M \setminus \mbox{Lip}(f_x) $, $p \in \partial^-_{PL}f_x(y)$, two sequences $\{(y_k,p_k)\}_{k\in \N}$ in $T^*M$ and $\{\lambda_k\}_{k\in \N}$ in $(0,+\infty)$ satisfying
\[
\lim_{k\rightarrow +\infty} \left|p_k\right|_{y_k} = +\infty, \quad \lim_{k\rightarrow +\infty} \left( y_k,\lambda_k p_k\right) = (y,p)
\quad \mbox{and} \quad p_k \in \partial^-f_x(y_k) \quad \forall k \in \N,
\]
and  $\gamma :[0,1] \rightarrow M$ be an horizontal minimizing path from $x$ to $y$ obtained as the uniform limit of a subsequence of the sequence of minimizing geodesics $\{\gamma_k\}_{k\in \N}$ given by the projections of the normal extremals  $\psi_k:[0,1] \rightarrow T^*M$ verifying $\psi_k(1)=(y_k,p_k)$. By Proposition \ref{PROPidem}, we may indeed assume that there are two sequences $\{(\tilde{y}_k,\tilde{p}_k)\}_{k\in \N}$ in $T^*M$ and $\{\tilde{\lambda}_k\}_{k\in \N}$ in $(0,+\infty)$ satisfying
\[
\lim_{k\rightarrow +\infty} \left| \tilde{p}_k\right|_{\tilde{y}_k} = +\infty, \quad \lim_{k\rightarrow +\infty} \left( \tilde{y}_k,\tilde{\lambda}_k \tilde{p}_k\right) = (y,p)
\quad \mbox{and} \quad \tilde{p}_k \in \partial^-_Pf_x(\tilde{y}_k) \quad \forall k \in \N,
\]
and that $\gamma :[0,1] \rightarrow M$ is the uniform limit of a subsequence of the sequence of minimizing geodesics $\{\tilde{\gamma}_k\}_{k\in \N}$ given by the projections of the normal extremals  $\tilde{\psi}_k:[0,1] \rightarrow T^*M$ verifying $\tilde{\psi}_k(1)=(\tilde{y}_k,\tilde{p}_k)$.
Assume as in Section \ref{SECExtremals} that $\gamma=\gamma ^{\bar{u}}$ for some $\bar{u}$ in an open set $U$ of $L^2([0,1],\R^m)$ where the end-point mapping $E$ is well-defined and smooth. Then, for $k$ large enough there is $\tilde{u}_k \in U$ such that $\tilde{\gamma}_k=\bar{\gamma}^{\tilde{u}_k}$ and moreover we can without loss of generality parametrize $\bar{\gamma}^{\tilde{u}_k}$ by arc-length and assume that $|\tilde{u}_k(t)|=d_{SR}(x,\tilde{y}_k)$ almost everywhere in $[0,1]$, so that $\|\tilde{u}_k\|_{\infty}<\infty$. By Remark \ref{REM28oct}, we have 
\[
 D^2_{\tilde{u}_k}C (v) - \tilde{p}_k \cdot D^2_{\tilde{u}_k}E\geq 0 \qquad \forall v \in \mbox{Ker} (D_{\tilde{u}_k} E)
\]
so we have 
\[
\mbox{ind}_- \left( \left( -\tilde{p}_k,1\right) \cdot \left( D^2_{\tilde{u}_k} F \right)_{\vert \mbox{Ker} (D_{\tilde{u}_k} F)}  \right) =0.
\]
Therefore, by Proposition \ref{PROPGohNormal}, we have 
\[
\lim_{k \rightarrow +\infty} \frac{ \tilde{\psi}_k(t) }{ |\tilde{p}_k|_{\tilde{y}_k} } \cdot \left[ X^i,X^j\right] \left(\tilde{\gamma}_k(t)\right) = 0.
\]
By a compactness argument, we conclude that, up to a subsequence, the sequence of extremals $\{\tilde{\lambda}_k\tilde{\psi}_k \}_k$ converges to an abnormal extremal satisfying the required property (see {\it e.g.} \cite[Lemma 4.8]{trelat00}). 
\end{proof}

\subsection{Characterizations of the Minimizing Sard Conjecture}

We are now ready to give several different characterizations of the Minimizing Sard Conjecture (recall that $f_x$ has been defined in (\ref{19oct1})). 

\begin{proposition}\label{PROPcharac}
For every $x\in M$, the following properties are equivalent:
\begin{itemize}
\item[(i)] The closed set $\mbox{Abn}^{min}(x)$ has Lebesgue measure zero in $M$.
\item[(ii)] The closed set $\mbox{Goh-Abn}^{min}(x)$ has Lebesgue measure zero in $M$.
\item[(iii)] For almost every $y\in M$, $\partial^-_{PL}f_x(y)= \emptyset$. 
\item[(iv)] For almost every $y\in M$, $\partial^-f_x(y)\neq \emptyset$. 
\item[(v)] The set $\mbox{Lip}^-(f_x)$ has full Lebesgue measure in $M$. 
\item[(vi)] The function $d_x$ is differentiable almost everywhere in $M$. 
\item[(vii)] The open set $\mbox{Lip}\,(f_x)$ has full Lebesgue measure in $M$.
\item[(viii)] For almost every $y\in M$, $\partial^-_Pf_x(y)\neq \emptyset$. 
\item[(ix)] The function $f_x$ is smooth on an open subset of $M$ of full Lebesgue measure in $M$. 
\end{itemize}
\end{proposition}

\begin{proof}[Proof of Proposition \ref{PROPcharac}]
The implications (ix) $\Rightarrow$ (viii) and (ix) $\Rightarrow$ (vii) are immediate, (viii) $\Rightarrow$ (iv) follows from the inclusion $\partial^-_Pf_x\subset \partial^-f_x$, (vii) $\Rightarrow$ (vi) is a consequence of Rademacher's Theorem, (vii) $\Rightarrow$ (v) follows from the inclusion $\mbox{Lip}(d_x) \subset \mbox{Lip}^-(d_x)$, (vii) $\Leftrightarrow$ (iii) is a consequence of Proposition \ref{PROPLip-PL}, (v) $\Rightarrow$ (iv) is a corollary of Proposition \ref{PROPLip-Sub}, and (vi) $\Rightarrow$ (iv) follows by definition of $\partial^-d_x$. Moreover, (iv) $\Rightarrow$ (i) follows by Proposition \ref{PROPsub1} together with Sard's Theorem applied to the mapping $\exp_x$, (ii) $\Rightarrow$ (vii) is a consequence of Proposition \ref{PROPGohLip} and (i) $\Rightarrow$ (ii) follows from the inclusion $\mbox{Goh-Abn}^{min}(x) \subset \mbox{Abn}^{min}(x)$. Thus it remains to show that (i) $\Rightarrow$ (ix), this proof can be found in \cite[\S 2.1]{br20}.
\end{proof}

\section{Proof of Theorem \ref{THM}}\label{SECProofTHM}

Let $x\in M$ be fixed and $M_x$ be the set of points $y\in M\setminus\{x\}$ for which all minimizing geodesics have Goh-rank at most $1$, which by assumption has full Lebesgue measure in $M$. The first step of the proof of Theorem \ref{THM} consists in showing that we may indeed assume that we work in the open unit ball of $\R^n$ with a distribution that admits a global parametrization by $m$ smooth vector fields. For every $y \in \R^n$ and $r>0$, we denote by $B_r(y)$ the open unit ball in $\R^n$ centered at $y$ with radius $r$. Moreover, we recall that a point $y\in \R^n$ is a Lebesgue density point (or a point of density $1$) for a measurable set $E\subset \R^n$ if
\[
\lim_{r\rightarrow 0} \frac{\mathcal{L}^n\left(B_r(y)\cap E \right)}{\mathcal{L}^n\left(B_r(y)\right)}=1,
\]
where $\mathcal{L}^n$ stands for the Lebesgue measure in $\R^n$ and that if $E$ has positive Lebesgue measure then almost every point of $E$ is a Lebesgue density point for $E$ (see {\it e.g.} \cite{eg15}). Our first result is the following:

\begin{proposition}\label{PROPSTEP1}
If $\mbox{\rm Abn}^{min}(x)$ has positive Lebesgue measure in $M$, then there are a complete sub-Riemannian structure $(\bar{\Delta},\bar{g})$ of rank $m$ on $B_1(0)$ generated by an orthonormal family of smooth vector fields $\bar{\mathcal{F}}=\{\bar{X}^1, \ldots, \bar{X}^m\}$ on $B_1(0)$ along with a point $\bar{y} \in B_1(0)\setminus \{0\}$ such that the following properties are satisfied:
\begin{itemize}
\item[(i)] all minimizing horizontal paths from $0$ to $\bar{y}$ have Goh-rank at most $1$,
\item[(ii)] $\bar{y}$ is a Lebesgue density point of $\mbox{\rm Abn}^{min}(0)$. 
\end{itemize}
\end{proposition}

\begin{proof}[Proof of  Proposition \ref{PROPSTEP1}]
Let us first recall a result of compactness for minimizing geodesics. For every $y\in M$, we denote by $\Gamma^y \subset W^{1,2}([0,1],M)$  the set of minimizing geodesics from $x$ to $y$. We refer the reader to \cite{agrachev01,riffordbook} for the proof of the following result:

\begin{lemma}\label{STEP1_LEM1}
For every compact set $\mathcal{K}\subset M$, the set of $\gamma \in \Gamma^y$ with $y\in \mathcal{K}$ is a compact subset of $W^{1,2}([0,1],M)$ and the mapping $y\in M \mapsto \Gamma^y \in \mathcal{C}(W^{1,2}([0,1],M))$ has closed graph (here $\mathcal{C}(W^{1,2}([0,1],M))$ stands for the set of compact subsets of $W^{1,2}([0,1],M)$ equipped with the Hausdorff topology).
\end{lemma}

Assume now that $\mbox{Abn}^{min}(x)$ has positive Lebesgue measure in $M$ and pick a Lebesgue density point $\bar{y}$ of the set $\mbox{Abn}^{min}(x)\cap M_x$. Lemma \ref{STEP1_LEM1} along with the parametrization that can be made along every minimizing geodesic as shown in Section \ref{SECExtremals} yields the following result:

\begin{lemma}\label{STEP1_LEM2}
There are an open neighborhood $\mathcal{B}$ of $\bar{y}$, a positive integer $N$, $N$ minimizing geodesics $\bar{\gamma}^1, \ldots$, $\bar{\gamma}^N \in \Gamma^{\bar{y}}$, $N$ open sets $\mathcal{V}^1, \ldots, \mathcal{V}^N \subset M$ diffeomorphic to $B_1(0)$ containing respectively $\bar{\gamma}^1([0,1]), \ldots, \bar{\gamma}^N([0,1])$, $N$ orthonormal  families of smooth vector fields $\mathcal{F}^1 = \{X^{1,1}, \ldots, X^{1,m}\}, \ldots, \mathcal{F}^N = \{X^{N,1}, \ldots, X^{N,m}\}$ defined respectively on $\mathcal{V}^1, \ldots, \mathcal{V}^N$ and $N$ open sets $\mathcal{W}^1, \ldots, \mathcal{W}^N$ containing $\mathcal{B}$ with 
\[
\bar{\gamma}^1([0,1]) \subset \overline{\mathcal{W}}^1 \subset \mathcal{V}^1, \ldots, \bar{\gamma}^N([0,1]) \subset \overline{\mathcal{W}}^N \subset \mathcal{V}^N
\]
such that the following properties are satisfied:
\begin{itemize}
\item[(i)] For every $k=1, \ldots, N$ and every $z\in \mathcal{V}^k$, $\Delta(z) = \mbox{Span}\{X^{k,1}(z), \ldots, X^{k,m}(z)\}$. 
\item[(ii)] For every $y\in \mathcal{B}$ and every $\gamma \in \Gamma^y$, there is $k_{\gamma} \in \{1, \ldots,N\}$  such that 
\[
  \gamma([0,1]) \subset \mathcal{W}^{k_{\gamma}}.
\]
\end{itemize}
\end{lemma}

We now need to consider a family of sub-Riemannian distances from $x$ whose minimum over $\mathcal{B}$ coincides with $d_{SR}(x,\cdot)$ (the sub-Riemannian distance with respect to $(\Delta,g)$). For this purpose, we state the following lemma whose proof is left to the reader.

\begin{lemma}\label{STEP1_LEM3}
For every $k=1, \ldots,N$, there is a smooth function $\psi^k : \mathcal{V}^k \rightarrow [1,+\infty)$ satisfying the following properties:
\begin{itemize}
\item[(i)] $\psi^k(z)=1$ for all $z\in \mathcal{W}^k$.
\item[(ii)] $\psi^k(z)>1$ for all $z\in \mathcal{V}^k \setminus \overline{\mathcal{W}}^k$.
\item[(iii)] The sub-Riemannian structure $(\Delta,g^k:=\psi^k g)$ on $\mathcal{V}^k$ is complete. 
\end{itemize}
\end{lemma}

For every $k=1, \ldots,N$, we define the function $F^k:\mathcal{V}^k \rightarrow [0,+\infty)$ by
\[
F^k(y):= \frac{1}{2} d_{SR}^{g^k}(x,y)^2 \qquad \forall y \in \mathcal{V}^k,
\]
where $d_{SR}^{g^k}$ stands for the sub-Riemannian distance associated with $(\Delta,g^k)$. Since each sub-Riemannian structure $(\Delta,g^k)$ is complete on $\mathcal{V}_k$, each $F^k$ is continuous. Moreover, since $\mathcal{B}$ is contained in all $\mathcal{W}^k$, we have thanks to Lemma \ref{STEP1_LEM2} (ii) and Lemma \ref{STEP1_LEM3} (i),
\[
f_x(y) := \frac{1}{2} d_{SR}(x,y)^2 = \min \Bigl\{ F^1(y), \ldots, F^N(y) \Bigr\} \qquad \forall y \in \mathcal{B}.
\]
By permuting the indices $1, \ldots, N$ if necessary, we may indeed assume that there is $\bar{N} \in \{1, \ldots, N\}$ such that 
\[
f_x(\bar{y}) = F^1(\bar{y})= \cdots = F^{\bar{N}}(\bar{y}),
\]
so that there is an open neighborhood $\mathcal{B}'\subset \mathcal{B}$ of $\bar{y}$ satisfying
\[
f_x(y)= \min \Bigl\{ F^1(y), \ldots, F^{\bar{N}}(y) \Bigr\} \qquad \forall y \in \mathcal{B}'.
\]
Furthermore, if for every $y\in \mathcal{B}'$ and every $k\in \{1,\ldots, \bar{N}\}$, we denote by $\Gamma^y_k \subset W^{1,2}([0,1],M)$  the set of minimizing geodesics from $x$ to $y$ with respect to $(\Delta,g^k)$, then we have (by Lemma \ref{STEP1_LEM2} (ii) and Lemma \ref{STEP1_LEM3} (i)-(ii))
\[
\Gamma^{\bar{y}}_k = \Gamma^{\bar{y}} \qquad \forall k=1, \ldots, \bar{N},
\]
and in addition, since  $\bar{y}\in M_x$, all minimizing geodesics from $x$ to $\bar{y}$ have Goh-rank at most $1$. As a consequence, since $\Gamma^{\bar{y}}$ is a compact subset of $W^{1,2}([0,1],M)$ and the mappings $y\in \mathcal{B}' \mapsto \Gamma^y_k \in \mathcal{C}(W^{1,2}([0,1],M))$, with $k=1, \ldots, \bar{N}$, have closed graph and since the mapping $\gamma \in \Omega^{\Delta}_x \mapsto \mbox{Goh-rank}(\gamma)\in \N$ (recall that $\Omega^{\Delta}_x$ stands for the set of horizontal paths $\gamma:[0,1] \rightarrow M$ such that $\gamma(0)=x$) is upper semi-continuous, we infer that there is an open neighborhood $\mathcal{B}''\subset \mathcal{B}'$ of $\bar{y}$ such that 
\[
\mbox{Goh-rank}(\gamma)\leq 1 \qquad \forall y \in \mathcal{B}'', \, \forall k=1, \ldots,\bar{N}, \, \forall \gamma \in \Gamma^{y}_k. 
\]
We are ready to complete the proof of Proposition \ref{PROPSTEP1}. We claim that there is $\bar{k}$ in $\{1,\ldots, \bar{N}\}$ such that $\mathcal{B}'' \setminus \mbox{Lip}^-(F^{\bar{k}})$ has positive Lebesgue measure. As a matter of fact, otherwise there is a set $\tilde{\mathcal{B}}$ of full Lebesgue measure in $\mathcal{B}''$ such that every point in $\tilde{\mathcal{B}}$ belongs to $\mbox{Lip}^-(F^k)$ for all $k=1, \ldots, \bar{N}$. Then, for every $y\in \tilde{\mathcal{B}}$,  each $F^k$ admits a support function $\varphi^k$ from below at $y$ which is Lipschitz on its domain and so the function $\min \{\varphi^k\, \vert \, k=1,\ldots,N\}$ is a support function from below for $f_x$ at $y$ which is Lipschitz on its domain (given by the intersection of the domains of $\varphi^1, \ldots, \varphi^N$). Therefore, $\tilde{\mathcal{B}}$ is contained in $\mbox{Lip}^-(f_x)$ and as a consequence the proof of Proposition \ref{PROPcharac} shows that $\mbox{Abn}^{min}(x)\cap \mathcal{B}''$ has Lebesgue measure zero. This contradicts the fact that $\bar{y}\in \mathcal{B}$ is a Lebesgue density point of $\mbox{Abn}^{min}(x)$. We conclude the proof by considering a Lebesgue density point $\hat{y}$ of the set 
\[
\mathcal{B}'' \setminus \mbox{Lip}^-(F^{\bar{k}})\subset \mathcal{B}'' \setminus \mbox{Lip}(F^{\bar{k}}) \subset \mbox{Abn}_{\bar{k}}^{min}(x)\cap \mathcal{B}''
\]
(the inclusion follows by Proposition \ref{PROPLip-PL} and Proposition \ref{PROPsub3}, and $\mbox{Abn}_k^{min}(x)$ stands for the set of points that can be reached from $x$ through singular minimizing horizontal paths with respect to $(\Delta,g^{\bar{k}})$)
and by pushing forward the sub-Riemannian structure $(\Delta,g^{\bar{k}})$, the family of  vector fields $\mathcal{F}^{\bar{k}} = \{X^{\bar{k},1}, \ldots, X^{\bar{k},m}\}$ on $\mathcal{V}^{\bar{k}}$ and $\hat{y}$ to $B_1(0)$ by using the diffeomorphism from $\mathcal{V}^{\bar{k}}$ to $B_1(0)$ (note that we have to multiply each vector field by $1/\sqrt{\psi^{\bar{k}}}$ to obtain an orthonormal family).
\end{proof}

Now, we suppose for  contradiction that $\mbox{Abn}^{min}(x)$ has positive Lebesgue measure in $M$ and we consider the sub-Riemannian structure $(\bar{\Delta},\bar{g})$ on $B_1(0)$, the orthonormal family of smooth vector fields $\bar{\mathcal{F}}$ on $B_1(0)$  and the Lebesgue density point $\bar{y}$ (of $\mbox{Abn}^{min}(0)$) given by Proposition \ref{PROPSTEP1}. We define the continuous function $F:B_1(0)\rightarrow [0,+\infty)$ by 
\[
F(y):= \frac{1}{2} d_{SR} (0,y)^2 \qquad \forall y \in B_1(0),
\]
where $d_{SR}$ stands for the sub-Riemannian metric associated with $(\bar{\Delta},\bar{g})$ and we denote respectively by $\exp_0$ the exponential mapping associated with $(\bar{\Delta},\bar{g})$, by $\mathcal{C}_0$ the set of its critical values, by $M_0$  the set of points $y\in B_1(0)\setminus\{0\}$ for which all minimizing geodesics from $0$ to $y$ have Goh-rank at most $1$, and by $\Gamma^y \subset W^{1,2}([0,1],B_0(1))$  the set of minimizing geodesics from $0$ to $y\in B_1(0)$. It follows from the proof of Proposition \ref{PROPcharac} and the upper semi-continuity of the mapping $\gamma \in \Omega^{\bar{\Delta}}_0 \mapsto \mbox{Goh-rank}(\gamma)\in \N$ (where $\Omega^{\bar{\Delta}}_0$ stands for the set of horizontal paths $\gamma:[0,1] \rightarrow B_0(1)$ such that $\gamma(0)=0$) that the set $\mathcal{A} \subset \mathcal{B}_1(0)$ defined by 
\begin{eqnarray}\label{20janv3}
\mathcal{A} := M_0  \setminus \left( \mbox{\rm Lip}^-(F) \cup \mathcal{C}_0 \right)
\end{eqnarray}
has positive Lebesgue measure. We pick a Lebesgue density point $\hat{y}$ of $\mathcal{A}$. 

Then, we denote by $U$ the open set of $u \in L^2([0,1],\R^m)$ for which the solution $\gamma^u:[0,1] \rightarrow B_1(0)$  to the Cauchy problem 
\[
\dot{\gamma}^u (t) = \sum_{i=1}^m u_i(t) \, \bar{X}^i \left(\gamma^u(t)\right) \, \mbox{ for a.e. } t \in [0,1], \quad \gamma^u(0)=0
\]
is well-defined, we denote by $E:U \rightarrow B_1(0)$ the end-point mapping associated with $0$ and $\bar{\mathcal{F}}$ defined by
\[
E(u) :=  \gamma^u(1) \qquad \forall u \in U,
\]
and we recall that the function $C: L^2([0,1],\R^m)\rightarrow [0,+\infty)$ has been defined in Section \ref{SECExtremals} by 
\[
C(u) :=\frac{1}{2} \|u\|_{L^2}^2 \qquad \forall u \in L^2([0,1],\R^m).
\]
Moreover, for every vector space $V \subset\R^n$, we denote by $\mbox{Proj}_{V}^{\perp}:\R^n \rightarrow V$ the orthogonal projection to $V$ (with respect to the Euclidean metric) and for every $y\in \R^n$ we define the affine space $V(y)$ by
\[
V(y) := \{y\} + V.
\]
The second step of the proof of Theorem \ref{THM} consists in proving the following result which is a preparatory result for the next step.

\begin{proposition}\label{PROPSTEP2}
There are $\delta, r, \rho \in (0,1)$, $K>0$, a control $\hat{u}$ in $U$ with $\gamma^{\hat{u}}\in \Gamma^{\hat{y}}$, a linear hyperplane $V\subset \R^n$ and a smooth function $h:[0,\delta^2/4) \rightarrow [0,+\infty)$ such that the following properties are satisfied: 
\begin{itemize}
\item[(i)] For every $u\in U$ with $\|u-\hat{u}\|_{L^2} < \delta$ and every $z\in V(\gamma^u(1)) \cap B_{\rho}(\gamma^u(1))$, there is $v \in U$ such that 
 \begin{eqnarray}\label{PROPSTEP2_1}
 \mbox{\em Proj}_{V(\gamma^u(1))}^{\perp} \left(\gamma^{v}(1)\right)=z,
 \end{eqnarray}
 \begin{eqnarray}\label{PROPSTEP2_2} 
 \left| \gamma^v(1)-z \right| \leq K \left| z-\gamma^u(1)\right|,
 \end{eqnarray}
 \begin{eqnarray}\label{PROPSTEP2_3}
 C(v) \leq C(u) + K \left|z - \gamma^u(1)\right|,
 \end{eqnarray}
 \begin{eqnarray}\label{PROPSTEP2_4}
 \left| \| v-\hat{u}\|_{L^2}^2 -  \| u-\hat{u}\|_{L^2}^2 \right| \leq  K \left| z-\gamma^u(1)\right|.
 \end{eqnarray}
 \item[(ii)] The function $h$ is smooth, nondecreasing and satisfies 
 \[
h(\alpha) = 0 \, \, \forall \alpha \in [0,\delta^2/4], \quad h(\alpha) > 0 \, \,\forall \alpha \in (\delta^2/4,\delta^2), \quad \lim_{\alpha \rightarrow \delta^2} h(\alpha)= +\infty.
\]
\item[(iii)] The function $W: B_r(\hat{y}) \rightarrow [0,+\infty)$ defined by
\[
W(y) := \inf \Bigl\{ C(u) + h\left(\|u-\hat{u}\|_{L^2}^2\right) \, \vert \, u \in U \mbox{ s.t. } \gamma^{u}(1)=y \Bigr\} \qquad \forall y \in B_r(\hat{y}) 
\]
is continuous and the set 
\[
\mathcal{K} := \Bigl\{y\in B_{r/2}(\hat{y})\setminus \mbox{\rm Lip}^-(W)\, \vert \, W(y)=F(y) \Bigr\} 
\]
have positive Lebesgue measure. 
\end{itemize}
\end{proposition}

\begin{proof}[Proof of Proposition \ref{PROPSTEP2}]
We start with the following lemma which is an easy consequence of the construction of $\hat{y}$. 

\begin{lemma}\label{STEP2_LEM1}
For  every $\gamma \in \Gamma^{\hat{y}}$, one of the following situation occurs: 
\begin{itemize}
\item[(i)] $\gamma$ is not a singular horizontal path. 
\item[(ii)] $\gamma$ is a singular horizontal path and $\mbox{Goh-rank}(\gamma)=1$. 
\end{itemize} 
\end{lemma}

\begin{proof}[Proof of Lemma \ref{STEP2_LEM1}]
Let $\gamma \in \Gamma^{\hat{y}}$ be fixed. If $\gamma$ is the projection of a normal extremal (w.r.t $(\bar{\Delta},\bar{g})$), then there is $p\in T_0^*B_1(0)=(\R^n)^*$ such that $y=\exp_0(p)$. Since $\hat{y}\in \mathcal{A}\subset B_1(0) \setminus \mathcal{C}_0$, the mapping $\exp_0$ is a submersion at $p$ and as a consequence $\gamma$ is not a singular horizontal path. Otherwise, $\gamma$ is not projection of a normal extremal and so it is singular and admit a abnormal lift satisfying the Goh condition (see Remark \ref{20oct1}). Since $\hat{y}$ belongs to $M_0$, we have $\mbox{Goh-rank}(\gamma)=1$. 
\end{proof}

We need to consider all minimizing geodesics in $\Gamma^{\hat{y}}$ and so to distinguish between the cases (i) and (ii) of Lemma \ref{STEP2_LEM1}.  The following result follows easily from the Inverse Function Theorem (see {\it e.g.} \cite[Proof of Theorem 3.14 p. 99]{riffordbook}). 

\begin{lemma}\label{STEP2_LEM2}
Let $u \in U$ with $\gamma^u \in \Gamma^{\hat{y}}$ and $\gamma^u$ non-singular be fixed. Then there are $\delta, \rho \in (0,1)$ and $K>0$ such that the following property is satisfied: For every $v\in U$ with $\|v-u\|_{L^2} < \delta$ and every $z\in B_{\rho}(\gamma^v(1))$, there is $w \in U$ such that
 \begin{eqnarray}\label{STEP2_LEM2_1}
 \gamma^w(1) =z, 
 \end{eqnarray}
 \begin{eqnarray}\label{STEP2_LEM2_2}
 C(w) \leq C(v) + K \left|z-\gamma^v(1)\right|,
 \end{eqnarray}
 \begin{eqnarray}\label{STEP2_LEM2_3}
 \| w-v\|_{L^2} \leq K \left|z - \gamma^v(1)\right|,
\end{eqnarray}
where the last inequality implies
 \begin{eqnarray}\label{STEP2_LEM2_4}
  \left|\| w-u\|_{L^2}^2 -  \| v-u\|_{L^2}^2\right| \leq  K(K+2) \left| z-\gamma^v(1)\right|.
 \end{eqnarray}
\end{lemma}

For every singular minimizing geodesic $\gamma^u \in \Gamma^{\hat{y}}$ with $u\in U$ (since $\hat{y}\in M_0$, all those $\gamma^u$ verify $\mbox{Goh-rank}(\gamma^u)= 1$), we denote by $P^u$ the vector line of co-vectors $p\in (\R^n)^*$  for which there is an abnormal lift $\psi$ of $\gamma^u$ satisfying the Goh condition and $\psi(1)=(\hat{y},p)$, and we define the hyperplane $V^u\subset \R^n$ by 
\begin{eqnarray}\label{10oct2}
V^{u} := \left( P^{u}\right)^{\perp} := \Bigl\{ v \in \R^n \, \vert \, p\cdot v=0, \, \forall p \in P^{u} \Bigr\}.
\end{eqnarray}
Note that since any co-vector in $P^{u}$ annihilates the image of the end-point mapping $E$ from $0$ associated with $\mathcal{F}$ in $U$ (see Section \ref{SECExtremals}), we have   
\begin{eqnarray}\label{21oct7}
\mbox{Im} \left(D_{u}E\right) \subset V^{u}.
\end{eqnarray}
The following lemma follows from an extension of a theorem providing a sufficient condition for local openness of a mapping at second order due to Agrachev and Sachkov \cite{as04}, we refer the reader to Appendix  \ref{SECSecondOrder} for the statement of the result and its proof. 

\begin{lemma}\label{STEP2_LEM3}
 Let $u \in U$ with $\gamma^u \in \Gamma^{\hat{y}}$ and $\gamma^u$ singular of Goh-rank $1$, be fixed. Then there are $\delta, \rho \in (0,1)$ and $K>0$ such that the following properties are satisfied: For every $v\in U$ with $\|v-u\|_{L^2} < \delta$ and every $z\in V^{u}(\gamma^v(1)) \cap B_{\rho}(\gamma^v(1))$, there is $w \in U$ such that 
 \begin{eqnarray}\label{STEP2_LEM3_1}
 \mbox{\em Proj}_{V^{u}(\gamma^v(1))}^{\perp} \left(\gamma^{w}(1)\right)=z, 
 \end{eqnarray}
 \begin{eqnarray}\label{STEP2_LEM3_2}
 \left| \gamma^w(1)-z \right| \leq K \left| z-\gamma^v(1)\right|,
 \end{eqnarray}
 \begin{eqnarray}\label{STEP2_LEM3_3}
 C(w) \leq C(v) + K \left|z - \gamma^v(1)\right|,
 \end{eqnarray}
 \begin{eqnarray}\label{STEP2_LEM3_4}
  \left|\| w-u\|_{L^2}^2 -  \| v-u\|_{L^2}^2\right| \leq  K \left| z-\gamma^v(1)\right|.
 \end{eqnarray}
 \end{lemma}

\begin{proof}[Proof of Lemma \ref{STEP2_LEM3}]
 Let $u \in U$, with $\gamma^u \in \Gamma^{\hat{y}}$ singular of Goh-rank $1$, be fixed and let $\mathcal{E}:U \rightarrow V^u(\hat{y})$ be the smooth mapping defined by
\[
\mathcal{E}(v) := \mbox{Proj}_{V^u(\hat{y})}^{\perp} \left( E(v) \right) \qquad \forall v \in U.
\]
We need to consider two different cases.  \\

\noindent First case: $ \mbox{Im} ( D_{u} \mathcal{E})=V^{u}$.\\
There are $u^1, \ldots, u^{n-1} \in L^2([0,1],\R^m)$ such that the linear operator
\[
\begin{array}{rcl}
\R^{n-1} & \longrightarrow & V^{u} \\
\alpha & \longmapsto & \sum_{i=1}^{n-1} \alpha_i D_u\mathcal{E} \left(u^{i}\right)
\end{array}
\] 
is invertible. Thus, by the Inverse Function Theorem, the smooth mapping $\mathcal{G}^u: \R^{n-1} \rightarrow V^{u}(\hat{y})$ defined in a neighborhood of the origin by
\[
\mathcal{G}^u (\alpha) := \mathcal{E} \left( u + \sum_{i=1}^{n-1}  \alpha_i u^{i}\right)
\]
admits an inverse $\mathcal{H}^u$ of class $C^1$ on an open neighborhood of $u$. In fact, by $C^1$ regularity of $E$, this property holds uniformly on a neighborhood of $u$. There are $\delta, \rho \in (0,1)$ and $K>0$ such that for every $v\in U$ with $\|v-u\|_{L^2}< \delta$, the smooth mapping $\mathcal{G}^v: \R^{n-1} \rightarrow V^{u}(\hat{y})$ defined in a neighborhood of the origin by
\[
\mathcal{G}^v (\alpha) := \mathcal{E} \left( v + \sum_{i=1}^{n-1}  \alpha_i u^{i}\right)
\]
admits an inverse $\mathcal{H}^v=(\mathcal{H}^v_1,\ldots, \mathcal{H}^v_{n-1}): V^u(\hat{y}) \cap B_{\rho}(\mathcal{E}(v)) \rightarrow \R^{n-1}$ of class $C^1$ such that 
\[
\left| \mathcal{H}^v(z)-\mathcal{H}^v(z') \right| \leq K |z-z'| \qquad \forall z,z' \in V^u(\hat{y}) \cap B_{\rho}(\mathcal{E}(v)).
\]
As a consequence, if $v$ belongs to $U$ with $\|v-u\|_{L^2} < \delta$ and $z$ belongs to $V^{u}(\gamma^v(1)) \cap B_{\rho}(\gamma^v(1))$, then we have 
\[
z' :=  \mbox{Proj}_{V^{u}(\hat{y})}^{\perp} (z) \in V^u(\hat{y}) \cap B_{\rho}(\mathcal{E}(v)).
\]
Hence, by the above result, the control $w\in U$ defined by 
\begin{eqnarray}\label{20janv1}
w:=v + \sum_{i=1}^{n-1}  \mathcal{H}_i^v(z') u^{i}
\end{eqnarray}
satisfies 
\[
\mbox{Proj}_{V^u(\hat{y})}^{\perp} \left( \gamma^w(1)\right)= \mathcal{E}(w) =z',
\]
which implies
\[
\mbox{Proj}_{V^u(\gamma^v(1))}^{\perp} \left( \gamma^w(1)\right)= z,
\]
and, by noting that 
\[
\mathcal{H}^v(\mbox{Proj}_{V^u(\hat{y})}^{\perp}\gamma^v(1))= \mathcal{H}^v(\mathcal{E}(v))=0 \quad \mbox{and} \quad \left|z'-\mathcal{E}(v)\right| = \left| z-\bar{\gamma}^v(1)\right|,
\]
we also have 
\begin{eqnarray*}
 \|w-v\|_{L^2}  & = &  \left\|  \sum_{i=1}^{n-1}  \mathcal{H}_i^v(z') u^{i} -  \sum_{i=1}^{n-1}  \mathcal{H}^v_i(\mathcal{E}(v)) u^i\right\|_{L^2}  \\
& \leq & \left| \mathcal{H}^v(z') - \mathcal{H}^v(\mathcal{E}(v))\right|  \sum_{i=1}^{n-1}   \left\|u^i\right\|_{L^2}  \\
& \leq & K \left|z' -\mathcal{E}(v) \right|  \sum_{i=1}^{n-1}   \left\|u^i\right\|_{L^2}  \\
& = & K \left|z -\gamma^v(1)\right|  \sum_{i=1}^{n-1}   \left\|u^i\right\|_{L^2} =: KK' \left|z -\gamma^v(1)\right|.
\end{eqnarray*}
So, by considering local Lipschitz constants $L_E, L_C$ respectively for $E$ and $C$, we infer that for any $v\in U$ with $\|v-u\|_{L^2} < \delta$ and any $z\in V^{u}(\gamma^v(1)) \cap B_{\rho}(\gamma^v(1))$, the control $w$ given by (\ref{20janv1}) satisfies (\ref{STEP2_LEM3_1}), 
\begin{eqnarray*}
\left| \gamma^w(1)-z \right| & \leq & \left| \gamma^w(1)-\gamma^v(1) \right| + |\gamma^v(1)-z| \\
& \leq &  L_E \|w-v\|_{L^2}  + |\gamma^v(1)-z| \\
& = & (L_EKK'+1)  \left|z -\gamma^v(1)\right|
\end{eqnarray*}
which gives (\ref{STEP2_LEM3_2}) (for a certain constant), 
\[
C(w) \leq C(v) + L_C \|w-v\|_{L^2} \leq C(v) + L_CK K' \left|z - \gamma^v(1)\right|
\]
which gives (\ref{STEP2_LEM3_3}) (for a certain constant), and (because $\|v-u\|_{L^2} < \delta<1$  and $|z-\gamma^v(1)|<\rho<1$)
\begin{eqnarray*}
\left| \| w- u\|_{L^2}^2 -  \| v-u\|_{L^2}^2 \right| & = & \left| \| w-v\|_{L^2}^2   + 2 \langle w-v,v-u\rangle_{L^2}\right| \\
& \leq &    \| w-v\|_{L^2}^2 + 2 \| w-v\|_{L^2}\\
& \leq &  KK' (KK'+2)\left|z - \gamma^v(1)\right|
\end{eqnarray*}
which gives (\ref{STEP2_LEM3_4}) (for a certain constant). We complete the proof by considering the maximum of the constants above.\\

\noindent Second case: $ \mbox{Im} ( D_{u} \mathcal{E})\neq V^{u}$.\\
We claim that
\begin{eqnarray}\label{21oct8}
\mbox{ind}_- \left( \lambda \cdot \left( D^2_{u} \mathcal{E} \right)_{\vert \mbox{Ker} (D_{u} \mathcal{E})}  \right) =+\infty \qquad \forall \lambda \in \left( \mbox{Im} \left( D_{u} \mathcal{E} \right)\right)^{\perp^{V^u}} \setminus \{0\},
\end{eqnarray} 
where $( \mbox{Im} ( D_{u} \mathcal{E}))^{\perp^{V^u}}$ stands for the set of linear forms on $V^u$ which annihilate $ \mbox{Im} ( D_{u} \mathcal{E})$. To prove the claim, we observe that if (\ref{21oct8}) is false then there is a linear form $\lambda \neq 0$ in  $( \mbox{Im} ( D_{u} \mathcal{E}))^{\perp^{V^u}}$ such that
\[
\mbox{ind}_- \left( \lambda \cdot \left( D^2_{u} \mathcal{E} \right)_{\vert \mbox{Ker} (D_{u} \mathcal{E})}  \right) < +\infty.
\]
Extend $\lambda$ into a non-zero linear form $\tilde{\lambda}$ on $\R^n$ by setting  $\tilde{\lambda} := \lambda \cdot \mbox{Proj}_{V^u}^{\perp}$. Since $\mbox{Im} \left(D_{u}E\right) \subset V^{u}$ (by (\ref{21oct7})), the linear form $\tilde{\lambda}$ belongs to  $( \mbox{Im} ( D_{u}E))^{\perp}$ and we have $\mbox{Ker}(D_{u} \mathcal{E})=\mbox{Ker}(D_{u}E)$. As a consequence, since
\[
D^2_{u} E = D^2_{u} \mathcal{E} + \left( D^2_{u} E -D^2_{u} \mathcal{E} \right),
\]
where the image of $D^2_{u} E -D^2_{u} \mathcal{E} $ is orthogonal to $V^{u}$, we infer that
\[
\mbox{ind}_- \left( \lambda \cdot \left( D^2_{u} \mathcal{E} \right)_{\vert \mbox{Ker} (D_{u} \mathcal{E})}  \right) = \mbox{ind}_- \left( \tilde{\lambda} \cdot \left( D^2_{u} E \right)_{\vert \mbox{Ker} (D_{u} E)}  \right).
\]
Thus if the claim is false, then by Remark \ref{20oct1}, the horizontal path $\gamma^{u}$ admits an abnormal extremal $\psi$ satisfying the Goh condition with $\psi(1)=\tilde{\lambda}$. Since $\tilde{\lambda}\notin P^{u}$, this is a contradiction. \\
We can now conclude the proof of the second case by applying Theorem \ref{THMopenquant} at $\bar{u}=u$ with 
\[
(X,\|\cdot\|)=\left(L^2([0,1],\R^m),\|\cdot\|_{L^2}\right), \quad F=\mathcal{E}:U \longrightarrow V^u(\hat{y}),
\]
\[ 
\mbox{and} \quad G=(E,C,\bar{C}): U \longrightarrow \R^{n+2} \quad \mbox{with} \quad \bar{C}=\|\cdot-u\|_{L^2}^2: U \longrightarrow \R.
\]
 By (\ref{21oct8}), there exist $\delta, \rho\in (0,1) $ and $K>0$ such that for any $v \in U$ and any $z'\in V^u(\hat{y})$ with
\[
\left\| v - u\right\|_{L^2} < \delta \quad \mbox{and} \quad |z'-\mathcal{E}(v)|<\rho,
\]
there are $w_1, w_2 \in L^2([0,1],\R^m)$ such that $v+w_1+w_2 \in U$,
\[
z'=\mathcal{E}(v+w_1+w_2),
\]
\[
w_1 \in \mbox{Ker}\left(D_v\mathcal{E}\right) \cap \mbox{Ker}\left(D_vG\right) 
\]
\[
\mbox{and} \quad  \|w_1\|_{L^2}<K \sqrt{|z'-\mathcal{E}(v)|}, \quad \|w_2\|_{L^2}<K |z'-\mathcal{E}(v)|.
\]
By setting $w:=v+w_1+w_2$ and by noting that $w_1 \in \mbox{Ker}(D_vG)\subset \mbox{Ker}(D_vE)$, the Taylor formula gives
\begin{eqnarray*}
\gamma^w(1)-\gamma^v(1) & = & E(w)-E(v) \\
 & =& D_vE (w_1+w_2) + D^2_vE(w_1+w_2) + o \left( \|w_1+w_2\|_{L_2}^2\right)\\
 & = & D_vE (w_2) + D^2_vE(w_1+w_2) + o \left( \|w_1+w_2\|_{L_2}^2\right)
\end{eqnarray*}
so that 
\[
\left| \gamma^w(1)-\gamma^v(1) \right| \leq KK'|z'-\mathcal{E}(v)|
\] 
for some $K'>0$. Therefore, if $z'$ is given by
\[
z' =  \mbox{Proj}_{V^{u}(\hat{y})}^{\perp} (z) \in V^u(\hat{y}) \cap B_{\rho}(\mathcal{E}(v)) \quad \mbox{with} \quad z \in V^{u}(\gamma^v(1)) \cap B_{\rho}(\gamma^v(1)),
\]
then we have (\ref{STEP2_LEM3_1}), if we denote by $L_E$ a local Lipschitz constant for $E$ then we have (because $|z'-\mathcal{E}(v)|=|z-\gamma^v(1)|$)
\begin{eqnarray*}
\left| \gamma^w(1)-z \right| & \leq & \left| \gamma^w(1)-\gamma^v(1) \right| + |\gamma^v(1)-z| \\
& \leq &  (L_EKK'+1)  \left|z -\gamma^v(1)\right|
\end{eqnarray*}
which gives (\ref{STEP2_LEM3_2}) (for a certain constant), by noting that  $w_1 \in \mbox{Ker}(D_vG)\subset \mbox{Ker}(D_vC)$ and $|z'-\mathcal{E}(v)|=|z-\gamma^v(1)|<\rho < 1$ we have
\begin{eqnarray*}
2C(w) & = &\|v+w_1+w_2\|_{L^2}^2 \\
& = & 2C(v) + 2 \langle v,w_1+w_2\rangle_{L^2} + \|w_1+w_2\|_{L^2}^2 \\
& = & 2C(v)   + 2 \langle v,w_2\rangle_{L^2} + \|w_1+w_2\|_{L^2}^2\\
& \leq & 2C(v)   + 2 \|v\|_{L^2} \|w_2\|_{L^2} + \left( \|w_1\|_{L^2} + \|w_2\|_{L^2}\right)^2 \\
& \leq & 2C(v)   + 2 \left(\|u\|_{L^2}+\delta\right) K \left|z-\gamma^v(1)\right| + \left(2K\sqrt{\left|z-\gamma^v(1)\right|}\right)^2\\
& = & 2C(v)   + K \left(2 \|u\|_{L^2}+2+4K\right) \left|z-\gamma^v(1)\right|,
\end{eqnarray*}
which gives (\ref{STEP2_LEM3_3}) (for a certain constant), and finally by noting that  $w_1 \in  \mbox{Ker}(D_vG)\subset \mbox{Ker}(D_v\bar{C})$ we also have
\begin{eqnarray*}
\left| \| w-u\|_{L^2}^2 -   \| v-u\|_{L^2}^2\right| & = &\left|  \| w-v\|_{L^2}^2 + 2 \langle w-v,v-u\rangle_{L^2} \right| \\
& \leq &   \| w_1+w_2 \|_{L^2}^2   + 2 \left|\langle w_2 ,v-u\rangle_{L^2}\right|  \\
& \leq & \left(2K\sqrt{\left|z-\gamma^v(1)\right|}\right)^2 + 2K \left|z-\gamma^v(1)\right|\\
& = &  2K(2K+1)\left|z - \gamma^v(1)\right|.
\end{eqnarray*}
The proof of Lemma \ref{STEP2_LEM3} is complete.
\end{proof}
 
The compactness results of Lemma \ref{STEP1_LEM1} together with the results of Lemmas \ref{STEP2_LEM2} and \ref{STEP2_LEM3} yield the following result:

\begin{lemma}\label{STEP2_LEM4}
There are $\delta, r , \rho \in (0,1)$, $K>0$, a positive integer $N$, $N$ controls $u^1, \ldots, u^{N}$ in $U$ with $\gamma^{u^l}\in \Gamma^{\hat{y}}$ for $l=1, \ldots, N$ and $N$ linear hyperplanes $V^{u^1}, \ldots V^{u^{N}}$ in $\R^n$ such that the following properties are satisfied: 
\begin{itemize}
\item[(i)]  For every $l\in \{1, \ldots N\}$, every $v\in U$ with $\|v-u^l\|_{L^2} < \delta$ and every $z\in V^{u^l}(\gamma^v(1)) \cap B_{\rho}(\gamma^v(1))$, there is $w \in U$ such that 
 \begin{eqnarray}\label{24oct1}
 \mbox{\em Proj}_{V^{u^l}(\gamma^v(1))}^{\perp} \left(\gamma^{w}(1)\right)=z, 
 \end{eqnarray}
 \begin{eqnarray}\label{24oct1bis}
 \left| \gamma^w(1)-z \right| \leq K \left| z-\gamma^v(1)\right|,
 \end{eqnarray}
 \begin{eqnarray}\label{24oct2}
 C(w) \leq  C(v) + K \left|z - \gamma^v(1)\right|,
 \end{eqnarray}
 \begin{eqnarray}\label{24oct3}
  \left| \| w-u^l\|_{L^2}^2 -  \| v-u^l\|_{L^2}^2 \right| \leq K \left| z-\gamma^v(1)\right|.
 \end{eqnarray}
\item[(ii)]  For any $y \in B_r(\hat{y})$ and $v \in \Gamma^y$, there is $l\in \{1, \ldots, N\}$, such that $\|v-u^l\|_{L^2} < \delta/2$.
\end{itemize}
 \end{lemma}

Pick a smooth nondecreasing function $h:[0,\delta^2/4) \rightarrow [0,+\infty)$ such that 
\[
h(\alpha) = 0 \, \forall \alpha \in [0,\delta^2/4], \quad h(\alpha) > 0 \, \forall \alpha \in (\delta^2/4,\delta^2) \quad \mbox{and} \quad \lim_{\alpha \rightarrow \delta^2} h(\alpha)= +\infty.
\]
Then, for every $l=1, \ldots, N$, define the function $W^l: B_r(\hat{y}) \rightarrow [0,+\infty)$ by
\[
W^l(y) := \inf \left\{ C(u) + h\left(\|u-u^l\|_{L^2}^2\right) \, \vert \, u \in U \mbox{ s.t. } \bar{\gamma}^{u}(1)=y \right\} \qquad \forall y \in B_r(\hat{y}). 
\]
By construction, the functions $W^1, \ldots, W^{\hat{N}}$ are continuous on their domain (remember that the end-point mapping $E:U \rightarrow M$ is open, see {\it e.g.}  \cite[\S 1.4]{riffordbook}) and we have (by Lemma \ref{STEP2_LEM4} (ii) and the construction of $h$)
\[
F(y) =  \min \left\{W^1(y), \ldots, W^{N}(y)\right\} \qquad \forall y \in B_r(\hat{y})
\]
and
\[
F(\hat{y}) = W^1(\hat{y}) = \cdots = W^{N}(\hat{y}). 
\]
Then, for every set $I\subset  \{1, \ldots, N\}$, we define the set $\mathcal{Z}^{I}\subset  B_{r/2}(\hat{y})$ by
\[
\mathcal{Z}^{I} := \Bigl\{ y \in B_{r/2}(\hat{y}) \, \vert \, F(y)=W^k(y) \, \forall k \in I \mbox{ and } F(y)<W^k(y) \, \forall k \notin I \Bigr\}
\]
and we check easily that
\[
B_{r/2}(\hat{y}) = \bigcup_{I\subset \{1,\ldots,N\}} \mathcal{Z}^{I} 
\]
and that we have for every $y\in B_{r/2}(\hat{y})$ and every $I\subset \{1,\ldots,N\}$ (see the end of the proof of Proposition \ref{PROPSTEP1})
\[
y \in \mathcal{Z}^{I} \setminus \mbox{Lip}^-(F) \quad \Longrightarrow \quad y \in \bigcup_{k\in I} \left( \mathcal{Z}^{I} \setminus \mbox{Lip}^-(W^k)\right). 
\]
Since by construction the point $\hat{y}$ is a Lebesgue density point of $\mathcal{A}$ given by the set (\ref{20janv3}) the set 
\[
B_{r/2} (\bar{y}) \setminus \mbox{Lip}^-(F) = \bigcup_{I\subset \{1,\ldots,N\}} \left( \mathcal{Z}^{I} \setminus \mbox{Lip}^-(F) \right) \subset  \bigcup_{I\subset \{1,\ldots,N\}}  \bigcup_{k\in I} \left( \mathcal{Z}^{I} \setminus \mbox{Lip}^-(W^k)\right)
\]
has positive Lebesgue measure and we conclude easily. 
\end{proof}

We denote by $P\subset \R^n$ the vector line orthogonal to $V$, for every $a$ in $B_{r/2}(\hat{y})\cap V(\hat{y})$ we define the piece of affine line $P_a$ given by
\[
P_a := (a+P) \cap B_{r/2}(\hat{y}),
\]
and we denote by $W_{|P_a}$ the restriction of $W$ to $P_a$. By construction, each $W_{|P_a}$ is a continuous function on an open set of dimension $1$. The third step of the proof of Theorem \ref{THM}, which can be seen as the core of the proof, consists in showing that the functions $W$ and $W_{|P_a}$  admit Lipschitz functions from below at the same points.  

\begin{proposition}\label{PROPSTEP3}
There exists $\check{K}>0$ such that for every $a\in B_{r/2}(\hat{y})\cap V(\hat{y})$ and every $\check{y} \in P_a$ the following property is satisfied: If $W_{|P_a}$ admits a support function from below $\varphi$ at $\check{y}$ which is Lipschitz on its domain (with Lipschitz constant $\mbox{\rm Lip}(\varphi)$), then there is $\nu>0$ such that
\[
W(y) \geq W(\check{y}) - \check{K} \left(1+\mbox{\rm Lip}(\varphi)\right) \left|y-\check{y}\right| \qquad \forall y \in B_{\nu}(\check{y});
\] 
in particular, $\check{y}$ belongs to $\mbox{\rm Lip}^-(W)$.
 \end{proposition}

 \begin{proof}[Proof of Proposition \ref{PROPSTEP3}]
 First, we note that since $W$ is well-defined and continuous on the closed ball $\bar{B}_{r/2}(\hat{y})$, there is $A>0$ such that $W(y)\leq A$ for all $y\in \bar{B}_{r/2}(\hat{y})$. As a consequence, if we consider some $y\in B_{r/2}(\hat{y})$ and $u\in U$ such that 
\begin{eqnarray}\label{17janv1}
W(y) = C(u) +  h\left(\|u-\hat{u}\|_{L^2}^2\right) \quad \mbox{and} \quad \gamma^u(1)=y,
\end{eqnarray}
then we have $h(\|u-\hat{u}\|_{L^2}^2)\leq A$, hence (because $h(\alpha)$ goes to $+\infty$ as $\alpha$ increases to $\delta^2$) there is $\bar{\delta}\in (0,\delta)$ such that $\|u-\hat{u}\|_{L^2}< \bar{\delta}$. We denote by $L>0$ the Lipschitz constant of $h$ on the set 
$[0,\check{\delta}^2]$ with $\check{\delta}^2:=\bar{\delta}^2+(\bar{\delta}^2+\delta^2)/2$ (remember that $h$ is smooth on its domain).

Let $a\in B_{r/2}(\hat{y})\cap V(\hat{y})$, $\check{y} \in P_a$ and $\varphi$ a function which is Lipschitz (with Lipschitz constant $\mbox{Lip}(\varphi)$) on an open segment of $P_a$ containing $\check{y}$ such that $W_{|P_a}$ admits $\varphi$ as support function from below at $\check{y}$. Given $y \in B_{\rho}(\check{y}) \cap B_{r/2}(\hat{y})$, we consider a control $u\in U$ satisfying (\ref{17janv1}) and we set (see Figure \ref{fig1})
\[
z:=\mbox{Proj}_{P_a}^{\perp}(y).
\]

\begin{figure}[H]
\begin{center}
\includegraphics[width=7cm]{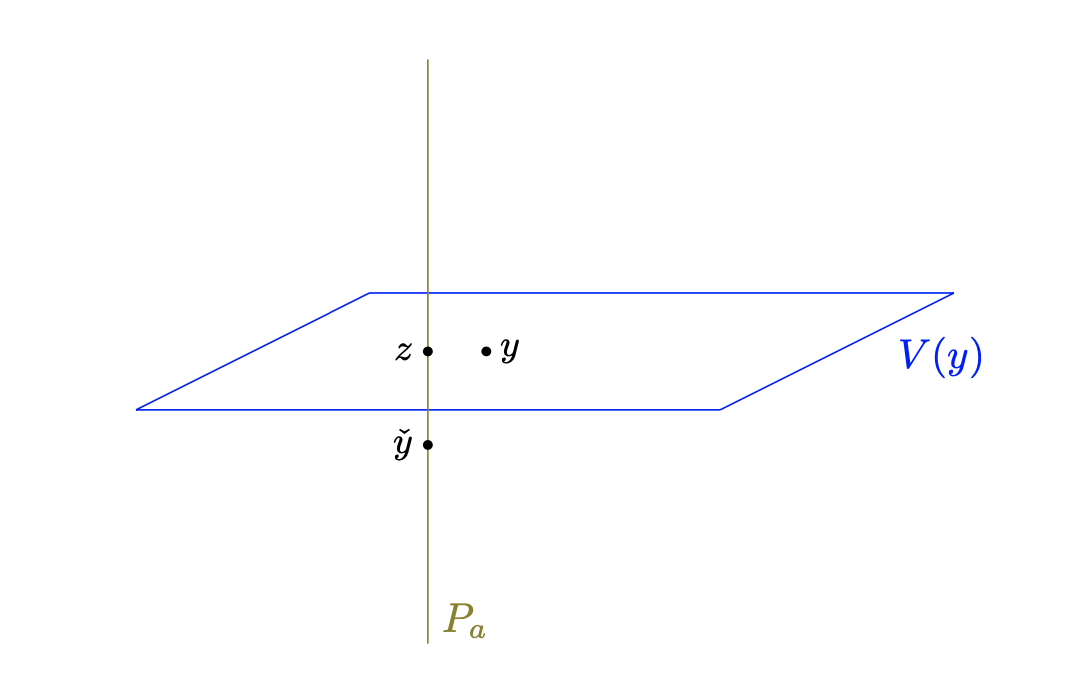}
\caption{$z$ belongs to $P_a\cap V(y)$ \label{fig1}}
\end{center}
\end{figure}

By construction, we have $|y-z| \leq |y-\check{y}| < \rho $ with $y=\gamma^u(1)$ and since $h(\|u-\hat{u}\|_{L^2}^2)<+\infty$ we have $\|u-\hat{u}\|_{L^2}<\delta$.  Therefore, by Proposition \ref{PROPSTEP2}, there is $v\in U$ such that
 \[
 \mbox{Proj}_{V(y)}^{\perp} \left(\gamma^{v}(1)\right)=z, \quad 
 \left|\gamma^v(1)-z \right| \leq K \left| z-y\right|,
 \]
 \[
 C(v) \leq  C(u) + K |z - y| \quad \mbox{and} \quad  \left| \| v-\hat{u}\|_{L^2}^2 -  \| u-\hat{u}\|_{L^2}^2 \right| \leq   K \left| z-y\right|.
 \]
 The properties $\mbox{Proj}_{V(y)}^{\perp}(\gamma^{v}(1))=z$ and $z\in P_a$ imply that $\gamma^{v}(1)\in P_a$. Moreover, we note that if $y$ belongs to the ball center at $\check{y}$ with radius less than $(\check{\delta}^2-\bar{\delta}^2)/K$ then we have 
 \[
 \| u-\hat{u}\|_{L^2}^2 \leq \bar{\delta}^2 \leq \check{\delta}^2 \quad
  \mbox{and} \quad \|v-\hat{u}\|_{L^2}^2 \leq  \| u-\hat{u}\|_{L^2}^2+K |z-y| \leq \bar{\delta}^2+K |y-\check{y}| \leq  \check{\delta}^2,
 \]
 so that
 \[
  h\left(\|v-\hat{u}\|_{L^2}^2\right)\leq h\left(\|u-\hat{u}\|_{L^2}^2\right) + L \left| \| v-\hat{u}\|_{L^2}^2 -  \| u-\hat{u}\|_{L^2}^2 \right| \leq KL |z-y|.
 \]
Consequently, for every $y$ in a ball center at $\check{y}$ with radius less than $(\check{\delta}^2-\bar{\delta}^2)/K$, we have
 \begin{eqnarray*}
  W_{|P_a}  \left(\gamma^v(1)\right) & \leq & C(v) +  h\left(\|v-\hat{u}\|_{L^2}^2\right) \\
 & \leq &  C(u) + K |z -y| +  h\left(\|u-\hat{u}\|_{L^2}^2\right) + KL |z-y|\\
 & \leq & W (y) + K(L+1) \left|z - y\right|.
 \end{eqnarray*}
Moreover, there is $\nu >0$ such that if $y$ belongs to $B_{\nu}(\check{y})$ then $\gamma^v(1)$ belongs to the domain of $\varphi$ (because $|\gamma^v(1)-\check{y}| \leq |\gamma^v(1)-z| +|z-y| + |y-\check{y}| \leq K | z-y| +2|y-\check{y}| \leq (K+2) |y-\check{y}|$), thus if $y\in B_{\nu}(\check{y})$ then we have
 \begin{eqnarray*}
 \varphi \left( \gamma^v(1)\right)\leq W_{|P_a}  \left(\gamma^v(1)\right) \leq W (y) + K(L+1) \left|z - y\right|.
 \end{eqnarray*}
In conclusion, we obtain that for any $y$ sufficiently close to $\check{y}$, there holds (we set $\hat{K}:=K(L+1)$)
 \begin{eqnarray*}
 W(y) & \geq & \varphi \left(\gamma^v(1)\right) - \hat{K} \left| z-y\right|\\
 & \geq & \varphi \left( \check{y}\right) - \mbox{Lip}(\varphi) \left| \gamma^v(1)-\check{y}\right| - \hat{K} \left| z-y\right|\\
 & \geq & \varphi \left( \check{y}\right) - \mbox{Lip}(\varphi)\left( \left|\gamma^v(1)-z\right| + \left|z-y\right| + \left|y-\check{y}\right| \right)- \hat{K} \left| z-y\right|\\
 & \geq & \varphi \left( \check{y}\right) - \mbox{Lip}(\varphi)( \hat{K}+2)  \left|y-\check{y}\right|- \hat{K} \left|y-\check{y}\right|\\
 & \geq & \varphi \left( \check{y}\right) - \check{K} \left( 1+\mbox{Lip}(\varphi) \right) \left|y-\check{y}\right|,
 \end{eqnarray*}
by setting $\check{K}:=\hat{K}+2$. 
\end{proof}

Proposition \ref{PROPSTEP3} allows us to distinguish between two cases. For each $a\in B_{r/2}(\hat{y})\cap V(\hat{y})$ we denote by $\mathcal{K}_a$ the intersection of $\mathcal{K}$ (introduced in  Proposition \ref{PROPSTEP2} (iii)) with $P_a$, that is,
\[
\mathcal{K}_a := \mathcal{K} \cap P_a \quad \forall a\in B_{r/2}(\hat{y})\cap V(\hat{y}).
\]
By Proposition \ref{PROPSTEP2} (iii) the set $\mathcal{K}\subset B_{r/2}(\hat{y})$ has positive Lebesgue measure and moreover by Proposition \ref{PROPDichotomy}, for every $a\in B_{r/2}(\hat{y})\cap V(\hat{y})$ there are two measurable sets $\mathcal{K}_{a}^{i}, \mathcal{K}_{a}^{ii} \subset \mathcal{K}_{a}$ with 
\[
\mathcal{L}^1 \left(  \mathcal{K}_{a}^{i}\cup  \mathcal{K}_{a}^{ii}\right) = \mathcal{L}^1\left( \mathcal{K}_{a}\right)
\]
satisfying the following properties:
\begin{itemize}
\item[(i)] For every $y\in \mathcal{K}_{a}^{i}$, the function $W_{|P_a}$ is differentiable at $y$.
\item[(ii)] For every $y \in \mathcal{K}_{a}^{ii}$, the function $W_{|P_a}$ is not differentiable at $y$ and there is a sequence $\{y_k\}_{k\in \N}$ in $P_a$ converging to $y$ such that $0\in \partial^-W_{|P_a}(y_k)$ for all $k\in \N$, in particular we have $0\in \partial^-_L W_{|P_a}(y)$.
\end{itemize}
We set
\[
\mathcal{K}^{i}:= \bigcup_{a\in B_{r/2}(\hat{y})\cap V(\hat{y}} \mathcal{K}_{a}^{i} \quad \mbox{and} \quad \mathcal{K}^{ii}:= \bigcup_{a\in B_{r/2}(\hat{y})\cap V(\hat{y})} \mathcal{K}_{a}^{ii}
\]
and we note that by Fubini's Theorem, we have
\[
\mathcal{L}^n (\mathcal{K})= \mathcal{L}^n\left(\mathcal{K}^{i}\cup \mathcal{K}^{ii}\right) >0.
\]
Two different cases have to be distinguished: $ \mathcal{L}^n(\mathcal{K}^{i})>0$ or $ \mathcal{L}^n(\mathcal{K}^{i})=0$ and $ \mathcal{L}^n(\mathcal{K}^{ii})>0$. The first case ($\mathcal{L}^n(\mathcal{K}^{i})>0$) leads easily to a contradiction because $\mathcal{L}(\mathcal{K}^{i})>0$ together with Proposition \ref{PROPSTEP3} imply that any point of $\mathcal{K}^{i}\subset \mathcal{K}$ belongs to $\mbox{Lip}^-(W)$ which contradicts Proposition \ref{PROPSTEP2} (iii). Therefore, we assume from now that
\[
\mathcal{L}^n(\mathcal{K}^{i})=0 \quad \mbox{and} \quad \mathcal{L}^n(\mathcal{K}^{ii})>0
\]
and we explain how to get a contradiction. \\

By Fubini's Theorem there is $\check{a}\in B_{r/2}(\hat{y})\cap V(\hat{y})$ such that 
\begin{eqnarray}\label{lastcontradiction}
\mathcal{L}^1 \left( \mathcal{K}_{\check{a}}^{ii}\right) >0. 
\end{eqnarray}
We pick a Lebesgue density point $\check{y}$ of $\mathcal{K}_{\check{a}}^{ii}$ in $P_{\check{a}}$ (w.r.t. $\mathcal{L}^1$) and we notice that by construction of $\mathcal{A}$ (see (\ref{20janv3}) the point $\check{y}$ does not belongs to  $\mathcal{C}_0$ the set of critical values of $\exp_0$. Then we pick a unit vector $\vec{v} \in \R^n$ tangent to $P_{\check{a}}$ and we define the continuous function $W^{\check{a}}: (-r/10,r/10) \rightarrow [0,+\infty)$ by 
\[
W^{\check{a}}(t) := W_{|P_{\check{a}}} \left( \check{y} + t \vec{v}\right)- W_{|P_{\check{a}}} \left( \check{y}\right)\qquad \forall t \in (-r/10,r/10).
\]
The next lemma follows essentially from the construction of $\check{y}$,  Proposition \ref{PROPSTEP3}, the Inverse Function Theorem and the Denjoy-Young-Saks Theorem (see Remark \ref{RemDYS}). 

\begin{proposition}\label{PROPSTEP4}
There are $\mu, \sigma >0$  with $B_{\mu}(\check{y}) \subset B_{r/2}(\hat{y})$ such that the following properties are satisfied: 
\begin{itemize}
\item[(i)] For any $t\in (-\mu,\mu)$ and any $p \in \partial^-W^{\check{a}}(t)$, there holds
\begin{eqnarray*}
|p| \leq 1 \quad \Longrightarrow \quad \Bigl( W^{\check{a}}(s) \leq W^{\check{a}}(t) + p (s-t)+\sigma (s-t)^2 \quad \forall s \in (-\mu,\mu)\Bigr).
\end{eqnarray*} 
\item[(ii)] $D^-W^{\check{a}}(0)=+\infty, D_+W^{\check{a}}(0)=-\infty$ and $D_-W^{\check{a}}(0)=D^+W^{\check{a}}(0)\in \R$.
\end{itemize}
\end{proposition} 

\begin{proof}[Proof of Proposition \ref{PROPSTEP4}]

For every $y\in M$, we denote by $\Gamma^y_W$  the set of controls $u \in L^2([0,1],\R^m)$ such that 
\[
W(y) = C(u) + h\left(\|u-\hat{u}\|_{L^2}^2\right).
\]
The proof of the following result is a consequence of the continuity of $W$ and the fact that $h$ is nondecreasing.

\begin{lemma}\label{STEP4_LEM1}
For every compact set $\mathcal{K}\subset B_{r/2}(\hat{y})$, the set of controls $u \in \Gamma^y_W$ with $y\in \mathcal{K}$ is a compact subset of $L^{2}([0,1],\R^m)$ and the mapping $y\in B_{r/2}(\hat{y}) \mapsto \Gamma^y_W \in \mathcal{C}(L^{2}([0,1],\R^m))$ has closed graph (here $\mathcal{C}(L^{2}([0,1],\R^m))$ stands for the set of compact subsets of $L^{2}([0,1],\R^m)$ equipped with the Hausdorff topology).
\end{lemma}

\begin{proof}[Proof of Lemma \ref{STEP4_LEM1}]
Let $\mathcal{K}$ be a compact subset of $B_{r/2}(\hat{y})$ and $\{u_k\}_{k\in \N}, \{y_k\}_{k\in \N}$ be two sequences respectively in  $L^{2}([0,1],\R^m)$ and $\mathcal{K}$ such that $u_k\in \Gamma_W^{y_k}$ for all $k\in \N$. The sequence $\{y_k\}_{k\in \N}$ is valued in $\mathcal{K}$ compact and since $W$ is bounded on $\mathcal{K}$ (because it is continuous) and
\[
W(y_k) = \frac{1}{2} \|u_k\|_{L^2}^2 + h\left( \|u_k-\hat{u}\|_{L^2}^2\right) \quad \forall k\in \N,
\]
the sequence $\{u\}_{k\in \N}$ is bounded in $L^2([0,1],\R^m)$, so there is an increasing subsequence $\{k_l\}_{l\in \N}$ such that $\{y_{k_l}\}_{l\in \N}$ tends to some $\bar{y}\in \mathcal{K}$ at infinity and $\{u_{k_l}\}_{l\in \N}$ weakly converges to some $\bar{u}\in L^{2}([0,1],\R^m)$. We note that by the Hahn-Banach Separation Theorem (see {\it e.g.} \cite{clarke13}) we have 
\[
\|\bar{u}\|_{L^2} \leq \liminf_{l\rightarrow +\infty} \|u_{k_l}\|_{L^2} \quad \mbox{and} \quad \|\bar{u}-\hat{u}\|_{L^2} \leq  \liminf_{l\rightarrow +\infty} \|u_{k_l}-\hat{u}\|_{L^2},
\]
therefore, since $h$ is nondecreasing, we obtain
\begin{eqnarray*}
\frac{1}{2} \|\bar{u}\|_{L^2}^2 + h\left( \|\bar{u}-\hat{u}\|_{L^2}^2\right)  & \leq &  \liminf_{l\rightarrow +\infty} \frac{1}{2} \|u_{k_l}\|_{L^2}^2 + \liminf_{l\rightarrow +\infty} h\left(\|u_{k_l}-\hat{u}\|_{L^2}^2\right)\\
& \leq &  \liminf_{l\rightarrow +\infty} \left\{\frac{1}{2} \|u_{k_l}\|_{L^2}^2 +  h\left(\|u_{k_l}-\hat{u}\|_{L^2}^2\right)\right\}\\
& = &  \lim_{l\rightarrow +\infty} W(y_{k_l}) = W(\bar{y}). 
\end{eqnarray*}
This shows that $\bar{u}$ belongs to $\Gamma_W^{\bar{y}}$ and that, up to a subsequence, the sequence $\{u_{k_l}\}_{l\in \N}$  convergences strongly to $\bar{u}$ in $L^2([0,1],\R^m)$ (because $\{u_{k_l}\}_{l\in \N}$ converges weakly to $\bar{u}$ and $\|\bar{u}\|_{L^2} = \liminf_{l\rightarrow +\infty} \|u_{k_l}\|_{L^2}$), which concludes the proof of the first part. The second part is left to the reader. 
\end{proof}

We now consider the set $\Theta \subset  \Gamma_W^{\check{y}} \subset L^2([0,1],\R^m)$ defined by
\begin{multline*}
\Theta := \left\{ u \in \Gamma_W^{\check{y}}\, \vert \, \exists  \{y_k\}_{k\in \N} \mbox{ in } B_{r/2}(\hat{y}),  \{p_k\}_{k\in \N} \mbox{ in } (\R^n)^*,  \{u_k\}_{k\in \N} \mbox{ in } L^2([0,1],\R^m)  \mbox{ s.t. }  \right. \\
\left. \lim_{k\rightarrow +\infty}y_k = \check{y}, \,  \lim_{k\rightarrow +\infty}u_k = u \mbox{ and } u_k\in \Gamma_W^{y_k}, \, p_k \in \partial^- W(y_k), \, |p_k|\leq 2\check{K}+1 \, \forall k \in \N\right\}.
\end{multline*}
 In the following result, the first part follows from the property (ii) above together with Lemma \ref{STEP4_LEM1} and the second part is due to the fact that $\check{y}$ is in $\mathcal{K}_{\check{a}}^{ii}$ and so does not belong to  $\mathcal{C}_0$.

\begin{lemma}\label{STEP4_LEM2}
The set $\Theta$ is a nonempty compact subset of $L^2([0,1],\R^m)$ and for every $u\in \Theta$ there are  $v^1, \ldots, v^n \in L^2([0,1],\R^m)$ such that the linear mapping 
\[
\lambda =(\lambda_1, \ldots, \lambda_n) \in \R^n \, \longmapsto \, D_uE \left( \sum_{i=1}^m \lambda_i v^{i}\right) \in \R^n
\]
is invertible at $\lambda=0$.
\end{lemma}
\begin{proof}[Proof of Lemma \ref{STEP4_LEM2}]
By construction of $\check{y}$ and the property (ii) above, there is a sequence $\{\bar{y}_k\}_{k\in \N}$ in $P_{\check{a}}$ converging to $\check{y}$ such that $0\in \partial^-W_{|P_{\check{a}}}(\bar{y}_k)$ for all $k\in \N$. Note that this property has to be understood as $0\in \partial^-W^{\check{a}}(\bar{t}_k)$ for all $k\in \N$, where the $\bar{t}_k$'s are defined by $\check{y}+\bar{t}_k\vec{v}=\bar{y}_k$. Proposition \ref{PROPSTEP3} shows that, in fact,  for any $k\in \N$ the function $W$ admits a support function from below at $\bar{y}_k$ which is Lipschitz $(2\check{K})$-Lipschitz on its domain. Thus, Proposition \ref{PROPLip-Lim} gives for every $k\in \N$ a co-vector $\bar{p}_k\in \partial^-_LW(\bar{y}_k)$ such that $|\bar{p}_k| \leq 2\check{K}$. By definition of $\partial^-_LW$ we infer that there is a sequence $\{y_k\}_{k\in \N}$ in $B_{r/2}(\hat{y})$ along with a sequence  $\{p_k\}_{k\in \N}$ in $(\R^n)^*$ such that 
\[
 \lim_{k\rightarrow +\infty}y_k = \check{y} \quad \mbox{and} \quad   p_k \in \partial^- W(y_k), \, |p_k|\leq 2\check{K}+1 \quad \forall k \in \N.
 \]
By taking a control $u_k$ in each $\Gamma_W^{y_k}$ and applying Lemma \ref{STEP4_LEM1}, we conclude that $\Theta$ is not empty. The compactness of $\Theta$ is an easy consequence of Lemma \ref{STEP4_LEM1}. 

Let us now prove the last part of Lemma \ref{STEP4_LEM2} and fix some $u\in \Theta$. By definition, there are sequences $\{y_k\}_{k\in \N}$, $ \{p_k\}_{k\in \N}$, $ \{u_k\}_{k\in \N}$ respectively in $B_{r/2}(\hat{y})$, $(\R^n)^*$ and $L^2([0,1],\R^m)$ such that 
\[
 \lim_{k\rightarrow +\infty}y_k = \check{y}, \,  \lim_{k\rightarrow +\infty}u_k = u \mbox{ and } u_k\in \Gamma_W^{y_k}, \, p_k \in \partial^- W(y_k), \, |p_k|\leq 2\check{K}+1 \, \forall k \in \N.
 \]
 Thus, for each $k\in \N$, there is a support function from below $\varphi_k: \mathcal{U}_k \rightarrow \R$ of class $C^1$ on its domain $\mathcal{U}_k \subset B_{r/2}(\hat{y})$ with $d\varphi_k(y_k)=p_k$ such that (we define $\hat{C}:U\rightarrow \R$ by $\hat{C}(u):=\|u-\hat{u}\|_{L^2}$ for all $u\in U$)
\begin{multline*}
C(u_k) + h\bigl( \hat{C}(u_k)\bigr) = W(y_k) = \varphi(y_k) \quad \\
\mbox{and} \quad C(u) + h\bigl( \hat{C}(u)\bigr) \geq W\left(E(u)\right) \geq \varphi \left(E(u)\right) \quad \forall u\in U_k,
\end{multline*}
where $U_k$ is an open neighborhood of $u_k$ in $U$ such that  $E(U_k)\subset \mathcal{U}_k$. Then, we infer that 
\[
p_k\cdot D_{u_k}E = D_{u_k}C +  h'\left(\|u_k-\hat{u}\|_{L^2}^2\right) \cdot D_{u_k}\hat{C} \qquad \forall k \in \N.
\]
By compactness (all $p_k$ satisfy $|p_k|\leq 2\check{K}+1$) and up to a subsequence, $\{p_k\}_{k\in \N}$ converges to some $p\in \partial^-_LW(\check{y})$ and in addition 
$u_k-\hat{u}$ converges in $L^2([0,1],\R^m)$ to $u-\hat{u}$ which satisfies $h'(\|u-\hat{u}\|_{L^2})=h(\|u-\hat{u}\|_{L^2})=0$ (by Proposition \ref{PROPSTEP2} (ii) and because $u\in \Gamma_W^{\check{y}}$ and $W(\check{y})=F(\check{y})$). The, by passing to the limit we obtain
 \[
 p\cdot D_uE = D_uC.
 \]
By Proposition \ref{PROPsub1}, we infer that $\check{y}$ belongs to the image of $\exp_0$ and since $\check{y}\notin \mathcal{C}_0$ the result follows. 
\end{proof}

The following result is an easy consequence of the Inverse Function Theorem and Lemma \ref{STEP4_LEM2}, its proof is left to the reader.

\begin{lemma}\label{STEP4_LEM3}
There are $\check{\delta}, \check{\rho} \in (0,1)$, $\check{M}>0$, a positive integer $N$, $N$ controls $u^1, \ldots, u^{N}$ in $\Theta$ such that the following properties are satisfied: For every $l\in \{1, \ldots N\}$, every $v\in U$ with $\|v-u^l\|_{L^2} < \check{\delta}$, there is a mapping $\mathcal{G}^{l,v}:B_{\check{\rho}}(\gamma^v(1)) \rightarrow U$ such that
\[
E\left( \mathcal{G}^{l,v} (y)\right) = y \qquad \forall y \in B_{\check{\rho}}(\gamma^v(1))
\]
and
\[
\bigl\| \mathcal{G}^{l,v} \bigr\|_{C^2} \leq \check{M}.
\]
\end{lemma}

We are ready to complete the proof of Proposition \ref{PROPSTEP4}. By construction of $\Theta$ there is $\mu>0$ such that for any $y\in  P_{\check{a}} \cap B_{\mu}(\check{y})$ for which  $\partial^-W_{|P_{\check{a}}}(y)$ admits a co-vector of norm $\leq 1$ we have that any $v$ in $\Gamma_W^y$ satisfies $\|v-u^l\|_{L^2} < \check{\delta}$ for some $l\in \{1, \ldots, N\}$. Therefore, if we consider $y\in  P_{\check{a}} \cap B_{\mu}(\check{y})$, $p \in \partial^-W_{|P_{\check{a}}}(y)$ with $|p|\leq 1$ and $v \in \Gamma_W^y$ then we have 
\[
W(z) \leq C\left( \mathcal{G}^{v}(z)\right) + h\left( \tilde{C}\left( \mathcal{G}^{v}(z)\right)\right) \qquad \forall z \in B_{\check{\rho}}(y)
\]
and moreover if $\varphi: \mathcal{I} \rightarrow \R$  is a support function from below for $W_{|P_{\check{a}}}$ at $y$, which is $C^1$ on an open segment containing $y$ in $P_{\check{a}}$ and verifies $d\varphi(y)=p$, then we also have 
\[
\varphi (z) \leq W_{|P_{\check{a}}}(z) \leq C\left( \mathcal{G}^{v}(z)\right) + h\left( \tilde{C}\left( \mathcal{G}^{v}(z)\right)\right) \qquad \forall z \in P_{\check{a}} \cap \mathcal{I} \cap B_{\check{\rho}}(y).
\]
We infer that the differential at $y$ of the function 
\[
z \longmapsto  C\left( \mathcal{G}^{v}(z)\right) + h\left( \tilde{C}\left( \mathcal{G}^{v}(z)\right)\right)
\]
is equal to $p$ and we conclude easily by noting that Lemma \ref{STEP4_LEM3} allows to obtain an upper bound for the $C^2$-norm of that function.

To prove (ii) we note that, by the property (ii) above, there is a sequence $\{t_k\}_{k\in \N}$ in $(-\mu,\mu)$ converging to $0$ such that $0\in \partial^-W^{\check{a}}(t_k)$ for all $k\in \N$ which by (i) yields
\[
W^{\check{a}}(s) \leq W^{\check{a}}(t_k) +\sigma (s-t_k)^2 \quad \forall s \in (-\mu,\mu), \, \forall k \in \N.
\]
Hence by passing to the limit we infer that we have (note that $W^{\check{a}}(0)=0$)
\begin{eqnarray}\label{27janv1}
W^{\check{a}}(s) \leq \sigma s^2 \quad \forall z \in P_{\check{a}} \cap B_{\mu}(\check{y}).
\end{eqnarray}
By Denjoy-Young-Saks' Theorem (see Remark \ref{RemDYS}), since $W^{\check{a}}$ is not differentiable at $\check{a}$, one of the following properties is satisfied:
\begin{itemize}
\item[(2)] $D^+W^{\check{a}}(0)=D^-W^{\check{a}}(0)=+\infty$ and $D_+W^{\check{a}}(0)=D_-W^{\check{a}}(0)=-\infty$,
\item[(3)] $D^+W^{\check{a}}(0)=+\infty, D_-W^{\check{a}}(0)=-\infty$ and $D_+W^{\check{a}}(0)=D^-W^{\check{a}}(0)\in \R$,
\item[(4)] $D^-W^{\check{a}}(0)=+\infty, D_+W{\check{a}}(0)=-\infty$ and $D_-W^{\check{a}}(0)=D^+W^{\check{a}}(0)\in \R$.
\end{itemize}
But  the properties (2)-(3) are prohibited by (\ref{27janv1}), so the proof is complete. 
\end{proof}
 
The final contradiction will be a consequence of the following:

\begin{proposition}\label{PROPSTEP5}
Let $\epsilon, \sigma>0$, $a,b \in \R$ with $b>a$ be such that
\begin{eqnarray}\label{9dec0}
\epsilon \geq \frac{(b-a)\sigma}{4}
\end{eqnarray}
and let $h:[a,b]\rightarrow \R$ with $h(a)=h(b)=0$ be a continuous function such that for any $s\in (a,b)$ and $p\in \partial^-h(s)$ there holds
\begin{eqnarray}\label{9dec1}
|p| \leq \epsilon \quad \Longrightarrow \quad \Bigl( h(s') \leq h(s) + p (s'-s)+\sigma (s'-s)^2 \quad \forall s \in [a,b]\Bigr).
\end{eqnarray} 
Then we have 
\begin{eqnarray}\label{9dec2}
h(s) \geq  D(s):= \max\Bigl\{-\epsilon (s-a), \epsilon (s-b)\Bigr\} \qquad \forall s \in [a,b].
\end{eqnarray}
\end{proposition}
\begin{proof}[Proof of Proposition \ref{PROPSTEP5}]
Suppose for contradiction that (\ref{9dec2}) does not hold and consider a global minimum $\bar{s} \in (a,b)$ of the function $h-D$ on $[a,b]$. So, we have $h(\bar{s})<D(\bar{s})$. If $\bar{s}$ belongs to $(0,(a+b)/2)$, then we have $0\in \partial^-(h-D)(\bar{s})=  \partial^-h(\bar{s})-\epsilon$, hence we infer that $\epsilon \in \partial^-h(\bar{s})$ which by (\ref{9dec1}) yields
\[
h(a) =0 \leq  h(\bar{s}) + \epsilon (a-\bar{s}) + \sigma (a-\bar{s})^2.
\]
Thus, we have
\[
 \epsilon (\bar{s}-a) - \sigma (\bar{s}-a)^2  \leq  h(\bar{s}) < D(\bar{s}) =  -\epsilon (\bar{s}-a),
\]
that is,
\[
2 \epsilon (\bar{s}-a) < \sigma (\bar{s}-a)^2 \quad \Longleftrightarrow \quad \epsilon < \frac{(\bar{s}-a)\sigma}{2} < \frac{(b-a)\sigma}{4}, 
\]
which contradicts (\ref{9dec0}). If $\bar{s}=(a+b)/2$, then we have 
\[
h(s)-D(s) \geq h(\bar{s})-D(\bar{s}) \qquad \forall s \in [a,b],
\]
which implies that
\[
h(s) \geq h(\bar{s})-D(\bar{s}) + D(s) \qquad \forall s \in [a,b],
\]
where the inequality becomes an equality for $s=\bar{s}$. Since the function $s\in [a,b] \mapsto h(\bar{s})-D(\bar{s}) + D(s)$ admits a minimum at $s=\bar{s}$, we infer that $0\in \partial^-h(\bar{s})$. Then,  (\ref{9dec1}) along with $h(a)=h(b)=0$ yield
\[
\left\{
\begin{array}{rcl}
h(a)  =  0 & \leq & h(\bar{s})  + \sigma (a-\bar{s})^2  = h(\bar{s})  + \sigma \frac{(b-a)^2}{4} \\
h(b)  =  0 & \leq & h(\bar{s})  + \sigma (b-\bar{s})^2 = h(\bar{s})  + \sigma \frac{(b-a)^2}{4}.
\end{array}
\right.
\]
Thus, we have 
\[
\left\{
\begin{array}{rcl}
 - \sigma \frac{(b-a)^2}{4} & \leq & h(\bar{s}) < D(\bar{s}) = -\frac{\epsilon(b-a)}{2}   \\
 - \sigma \frac{(b-a)^2}{4} & \leq & h(\bar{s}) <  D(\bar{s}) = -\frac{\epsilon(b-a)}{2}.
\end{array}
\right.
\]
We infer that
\[
- \sigma \frac{(b-a)}{4} +\frac{\epsilon}{2} < 0 < \sigma \frac{(b-a)}{4} -\frac{\epsilon}{2},
\]
a contradiction. The case $\bar{s} \in ((a+b)/2,b)$  is left to the reader. 
\end{proof}

Recall that there is a sequence $\{t_k\}_{k\in \N}$ in $(-\mu,\mu)$ converging to $0$ such that $0\in \partial^-W^{\check{a}}(t_k)$ for all $k\in \N$. We may assume without loss of generality that $\{t_k\}_{k\in \N}$ is contained in $(0,\mu)$ and is decreasing.  Proposition \ref{PROPSTEP4} (i) gives
\begin{eqnarray}\label{30janv}
0 = W^{\check{a}}(0) \leq W^{\check{a}}(t_k) +\sigma t_k^2 \quad   \forall k \in \N.
\end{eqnarray}
Let us distinguish two cases. \\

\noindent First case: There is $k\in \N$ such that $t_k \leq \sigma/4$ and $W^{\check{a}}(t_k)\geq 0$.\\
Since $W^{\check{a}}(0)=0$ and $W^{\check{a}}$ is continuous there is $b\in (0,t_k]$ such that $W^{\check{a}}(b)=0$. Therefore, by Proposition \ref{PROPSTEP4} (i), the function $h=W^{\check{a}}:[0,b] \rightarrow \R$ satisfies the assumptions of Proposition \ref{PROPSTEP5} (with $\epsilon=1$, $a=0$ and $(b-a)\sigma/4 \leq t_k\sigma/4\leq 1$)  and as a consequence we have 
\[
W^{\check{a}} (t) \geq \max\Bigl\{a-t, t-b\Bigr\} \qquad \forall t \in [a,b].
\]
This property contradicts the property $D_+W^{\check{a}}(0)=-\infty$ given by Proposition \ref{PROPSTEP4} (ii).\\

\noindent Second case: $W^{\check{a}}(t_k)< 0$ for all $k$ large.\\
For every $k$ large, we define the continuous function $h_k:[0,t_k] \rightarrow \R$ by 
\[
h_k(t) := W^{\check{a}}(t) - \frac{t}{t_k} \, W^{\check{a}}(t_k) \qquad \forall t \in [0,t_k].
\]
We have $h_k(0)=h_k(t_k)=0$ and moreover any $p$ in $\partial^-h_k(t)$ with $t\in [0,t_k]$ satisfies (by (\ref{30janv}))
\[
p + \frac{1}{t_k} W^{\check{a}}(t_k)  \in \partial^-W^{\check{a}}(t) \quad \mbox{with} \quad \left| \frac{1}{t_k} W^{\check{a}}(t_k) \right| = - \frac{1}{t_k} W^{\check{a}}(t_k) \leq \sigma t_k.
\]
Thus, by Proposition \ref{PROPSTEP4} (i), we infer that if $\sigma t_k\leq 1/2$ then we have for any $p$ in $\partial^-h_k(t)$ with $t\in [0,t_k]$ and $|p| \leq 1/2$,
\begin{eqnarray*}
 W^{\check{a}}(s) \leq W^{\check{a}}(t) + \left(p+  \frac{1}{t_k} W^{\check{a}}(t_k) \right) (s-t)+\sigma (s-t)^2 \quad \forall s \in [0,t_k].
\end{eqnarray*}
which gives
\begin{eqnarray*}
 h_k(s) \leq h_k(t) + p(s-t)+\sigma (s-t)^2 \quad \forall s \in [0,t_k].
\end{eqnarray*}
As in the first case, by applying Proposition \ref{PROPSTEP5}  we obtain a contradiction to the property $D_+W^{\check{a}}(0)=-\infty$. So, the proof of Theorem \ref{THM} is complete.

\section{Final comments}\label{SECComments}
 
\subsection{Domains of the limiting subdifferential of continuous functions}\label{SEC8jan} 
 
Given a continuous function $f:\mathcal{O} \rightarrow \R$ defined on an open set $\mathcal{O}$ of $\R^n$, we call domain of its limiting subdifferential, denoted by  $\mbox{dom} (\partial^-_L f)$, the set of $x\in \mathcal{O}$ where $\partial^-_Lf(x)$ is not empty. In Section \ref{SECAlter}, we proved that if $n=1$ then $\mbox{dom} (\partial^-_L f)$ has full Lebesgue measure in $\mathcal{O}$. But the we do not know the answer to the following:\\

\noindent {\bf Open question.} If $n\geq 2$, does $\mbox{dom} (\partial^-_L f)$ have full Lebesgue measure in $\mathcal{O}$?\\

We expect the answer to be No, even if $f$ is locally Hölder continuous. Nevertheless, since pointed sub-Riemannian distances are (locally) Hölder continuous (this is a consequence of Ball-Box Theorem, see {\it e.g.} \cite{abb17,bellaiche96,montgomery02}), a positive answer to the above question in the Hölder continuous case would have some interesting consequence on the image of sub-Riemannian exponential mappings, see (\ref{8janexp}) below.
 
\subsection{On the limiting subdifferentials of $f_x$}\label{SEClimitingmin}

Let $M$ be a smooth manifold equipped with a complete sub-Riemannan structure $(\Delta,g)$. The sub-Riemannnian Hamiltonian $H:T^*M \rightarrow \R$ canonically associated with $(\Delta,g)$ is defined by 
\[
(x,p) \, \longmapsto \frac{1}{2} \max \left\{ \frac{p(v)^2}{g_x(v,v)} \, \vert \, v \in \Delta(x) \setminus \{0\} \right\}
\] 
in local coordinates in $T^*M$. If we denote by $\phi^H_t$ the Hamiltonian flow (given by  $H$ w.r.t. the canonical symplectic 
structure on $T^*M$) then for every $x\in M$, the exponential mapping $\exp_x:T_x^*M \rightarrow M$ is defined by
\[
\exp_x(p) := \pi \left( \phi^H_1(x,p)\right) \qquad \forall p \in T_x^*M, 
\]
where $\pi:T^*M \rightarrow M$ stands for the canonical projection, and it satisfies 
\[
\exp_x(\lambda \, p) := \pi \left( \phi^H_{\lambda}(x,p)\right) \qquad \forall p \in T_x^*M, \, \forall \lambda \geq 0. 
\]
Let $x\in M$ be fixed, we define the set $\mathcal{P}_x^{min}\subset T_x^*M$, called {\it minimizing domain} of the exponential mapping $\exp_x$,  as
\[
\mathcal{P}_x^{min} := \left\{ p \in T^*_xM \, \vert \, d_{SR}\left(x, \exp_x(p)\right)^2 = 2H(x,p) \right\},
\]
it is the set of co-vectors $p\in T_x^*M$ for which the horizontal path $\gamma_p:[0,1] \rightarrow M$, given by $\gamma_p(t):=\pi(\phi_t^H(x,p))$ for all $t\in [0,1]$ is minimizing from $x$ to $\gamma_p(1)=\exp_x(p)$. Proposition \ref{PROPsub2} shows that we have
\begin{eqnarray}\label{8janexp}
\mbox{dom} \left(\partial^-_L f_x\right) \subset \exp_x \left( \mathcal{P}_x^{min}\right),
\end{eqnarray}
where $f_x:=d_{SR}^x(\cdot)^2/2$ and $\mbox{dom} (\partial^-_L f_x)$ denotes the set of $y\in M$ such that $\partial^-_L f_x(y)$ is non-empty. We do not know the answer to the following:\\

\noindent {\bf Open question.} Do we have $\mbox{dom} \left(\partial^-_L f_x\right) = \exp_x \left( \mathcal{P}_x^{min}\right)$ ?\\

It is worth to notice that the property "$\partial^-_Lf_x(y)\neq \emptyset$ for almost every $y\in M$" is not listed in Proposition \ref{PROPcharac}. We do not know either if this property is sufficient for the minimizing Sard conjecture to hold true. 

\subsection{Normal containers at infinity}

We keep here the same notations as in the previous section. Given $x\in M$, we denote by $\ell_x$ the set of all sequences $\{p_k\}_k$ in $T_x^*M$ such that 
\[
\lim_{r\rightarrow + \infty} |p_k|_{x} = +\infty \quad \mbox{and} \quad \lim_{k\rightarrow +\infty} H(x,p_k) = \frac{1}{2}. 
\]
By compactness, we can associate to each sequence $\{p_k\}_k$ in $\ell_x$ a set of horizontal paths starting from $x$. As a matter of fact, each $p_k$ gives rise to an horizontal path $\gamma_{p_k}:[0,+\infty) \rightarrow M$ (by setting $\gamma_{p_k}(t):=\pi(\phi_t^H(x,p_k))$ for all $t\geq 0$) and since $H(x,p_k)$ tends to $1/2$ all those curves are uniformly Lipschitz on each interval $[0,T]$ with $T>0$. So by Arzela-Ascoli's Theorem, the sequence $\{\gamma_{p_k}\}_k$ converges uniformly on compact sets, up to subsequences, to horizontal paths on $[0,+\infty)$ starting from $x$. We denote by $\Gamma_x^{\infty}$ the set of all such paths and we call it the {\it normal container at infinity} from $x$. By construction any path of $\Gamma_x^{\infty}$ is singular.



We call {\it minimizing normal container at infinity } from $x$,  the set of minimizing horizontal paths $\gamma:[0,1] \rightarrow M$ obtained as uniform limits of paths  $\gamma_{p_k}:[0,1] \rightarrow M$ where  $\{p_k\}_k$ is a sequence in $T_x^*M$ such that 
 \[
 p_k \in \mathcal{P}_x^{min} \quad \forall k \quad \mbox{and} \quad \lim_{r\rightarrow + \infty} |p_k|_{x} = +\infty.
\]
By construction, we have 
\[
\Gamma_x^{\infty,min}([0,1])  \subset \Gamma_x^{\infty} \left([0,+\infty)\right).
\] 
By Proposition \ref{PROPcharac}, the set $\mbox{Abn}^{min}(x)$ has Lebesgue measure zero in $M$ if and only if there holds $\partial^-_{PL}f_x(y)= \emptyset$ for almost every $y\in M$. Moreover, we infer easily from Proposition \ref{PROPGohLip} that for any $(y,p) \in T^*M$ with $p \in \partial^-_{PL}f_x(y)$, there is a minimizing horizontal path $\gamma \in \Gamma_x^{\infty,min}$ such that $\gamma(1)=y$ which satisfies the Goh condition. Those results suggest that a fine study of normal containers at infinity $\Gamma_x^{\infty}$ and $\Gamma_x^{\infty,min}$ may help in the understanding of the minimizing Sard conjecture. 

\subsection{Measure contraction properties}

Measure contraction properties consist in comparing the contraction of volumes along geodesics from a given point with what happens in classical model of Riemannian geometry. Unlike other notions of Ricci curvature (bounded from below) on measured metric spaces which are not relevant in sub-Riemannian geometry (see \cite{juillet21}), measure contraction properties have been shown to be satisfied for several types of sub-Riemannian structures (see \cite{juillet09,rifford13,al14,lee16,lcz16,rizzi16,br18,br20}), all of which do not admit strictly abnormal minimizing horizontal paths. The present paper provides new examples of sub-Riemannian structures which may have strictly abnormal minimizing horizontal paths and for which Ohta's definition of measure contraction property makes sense (see \cite{ohta07,rifford13}), it is thus natural to wonder whether they might enjoy measure contraction properties.

\appendix

\section{A second-order condition for local openness at second-order}\label{SECSecondOrder}

We state and prove in this section the result of local openness that we apply in the proof of Lemma \ref{STEP2_LEM3}. For this, we consider a Banach space $(X,\|\cdot\|)$ (whose open ball centered at $u\in X$ of radius $r>0$ will be denoted by $B_X(u,r)$), a positive integer $N$, an open subset $U$ of $X$ and a mapping $F: U \rightarrow \R^N$ which is assumed to be of class $C^2$ on $U$, which means that it satisfies the following properties (the usual Euclidean norm in $\R^N$ is denoted by $|\cdot|$):
\begin{itemize}
\item[(i)] the function $F$ is (Fréchet) differentiable at every $u\in U$, that is, there is a bounded linear operator $D_uF:X \rightarrow \R^N$ such that 
\[
\lim_{h \rightarrow 0} \frac{\left|  F(u+h) - F(u) - D_uF (h)   \right| }{ \| h \|} =0,
\]
\item[(ii)] the mapping $u \mapsto D_uF$ is continuous from $U$ to the set $L(X,\R^N)$  of bounded linear operators from $X$ to $\R^N$ equipped with the operator norm $\|\cdot\|$, 
\item[(iii)] the function $u\in U \mapsto D_uF \in L(X,\R^N)$ is of class $C^1$ as a function from $(X,\|\cdot \|)$ to $(L(X,\R^N),\|\cdot\|)$ (note that (i)-(ii) above can be adapted to functions valued in a Banach space instead of $\R^N$) which means that 
\[
\lim_{h \rightarrow 0} \frac{\left|  F(u+h) - F(u) - D_uF (h) - \frac{1}{2} D_u^2 F (h)  \right| }{ \| h \|^2} =0 \qquad \forall u \in U,
\]
where for every $u\in U$, $D_u^2F:X \rightarrow \R^N$ stands for the quadratic form defined by the (symmetric) bilinear form $(h,k) \mapsto D_u^2F(h,k)$ given by the derivative of $u\mapsto D_uF(h)$ in $u$ (along $k$)  and where the mapping $u\in U\mapsto D_u^2F\in L^2(X,\R^n)$ is continuous. 
\end{itemize}

We refer for example the reader to the monograph \cite{hamilton82} for further detail on differential calculus in infinite dimensions. 

By the Inverse Function Theorem, $F$ is locally open \guillemotleft{} at first order \guillemotright{} at any point where $F$ is a submersion, that is, where $D_uF$ is surjective. The second-order theory developed by Agrachev-Sachkov \cite{as04} and Agrachev-Lee \cite{al09} allows to give sufficient conditions for a local openness property \guillemotleft{} at second-order \guillemotright{} as we now show.  Given a critical point $u \in U$, that is, a point where $D_uF:X \rightarrow \R^N$ is not surjective, we define the co-rank of $u$ by
\[
\mbox{corank}_F (u) := N - \mbox{dim} \left(  \mbox{Im} \bigl( D_uF  \bigr)\right) \in [1,N]
\]
and we recall that the negative index of a quadratic form  $Q: X \rightarrow \R$ (that is $Q$ is defined by $Q(v):=B(v,v)$ with $B:X \times X \rightarrow \R$ a symmetric bilinear form) is defined by
\[
\mbox{ind}_- (Q) := \max \Bigl\{ \mbox{dim} (L) \ \vert \ Q_{\vert L \setminus \{0\}} <0 \Bigr\},
\]
where $Q_{\vert L \setminus \{0\}} <0$ means 
\[
Q(u) <0 \qquad \forall u \in L \setminus \{0\}.
\]
The following result provides a refinement of \cite[Theorem B.3  p.128]{riffordbook} which was itself obtained as an application of the second-order theory developped in Agrachev-Sachkov \cite[Chapter 20]{as04} and Agrachev-Lee \cite[Section 5]{al09} (see also \cite[Chapter 12]{abb17}) (given a vector space $V\subset \R^N$, $V^{\perp}$ stands for the set of linear forms on $\R^N$ which annihilate $V$):

\begin{theorem}\label{THMopenquant}
Let $F: U \rightarrow \R^N$ be a mapping of class $C^2$ on an open set $U \subset X$, $\bar{u} \in U$ be a critical point of $F$ of co-rank $r$ and let $G :U \rightarrow \R^d$, with  $d\in \N^*$, be a mapping of class $C^1$ on $U$. If there holds
\begin{eqnarray}\label{ind}
\mbox{\rm ind}_- \left( \lambda \cdot \left( D^2_{\bar{u}} F \right)_{\vert \mbox{\rm Ker} (D_{\bar{u}} F)}  \right) \geq N +d \qquad \forall \lambda \in \left( \mbox{\rm Im} \bigl( D_{\bar{u}} F \bigr)\right)^{\perp} \setminus \{0\},
\end{eqnarray}
then there exist $(\delta, \rho) \in (0,1)$ and $K>0$  such that the following property holds: For every $u \in U$, $x\in \R^N$ with
\[
\| u - \bar{u}\| < \delta \quad  \mbox{and} \quad |x-F(u)|<\rho,
\]
there are $w_1, w_2 \in X$ such that $u+w_1+w_2 \in U$,
\[
x=F(u+w_1+w_2),
\]
\[
w_1 \in \mbox{\rm Ker}\left(D_uF\right) \cap \mbox{\rm Ker}\left(D_uG\right)
\]
and
\[
 \|w_1\|<K \sqrt{|x-F(u)|}, \quad \|w_2\|<K |x-F(u)|.
\]
\end{theorem}

\begin{proof}[Proof of Theorem \ref{THMopenquant}]
Let $F: U \rightarrow \R^N$, $\bar{u} \in U$ and $G:U \rightarrow \R^d$ as in the statement be fixed such that (\ref{ind}) is satisfied. The following result will allow us to work on spaces of finite dimension, it is a consequence of (\ref{ind}).

\begin{lemma}\label{11jan1}
There are a vector space $W\subset X$ of dimension $D$ and a vector space $V\subset W$ of dimension $N-r$ such that the restriction $\tilde{F}:W_{\bar{u}} \rightarrow \R^N$ of $F$ to $W_{\bar{u}} := \{\bar{u}\} + W$ satisfies the following properties: There holds
\begin{eqnarray}\label{13janv1}
W = V \oplus \mbox{\rm Ker}\bigl(D_{\bar{u}}\tilde{F}\bigr), \quad \mbox{\rm Im} \bigl( D_{\bar{u}} \tilde{F} \bigr)=\mbox{\rm Im} \bigl( D_{\bar{u}} \tilde{F}_{\vert V} \bigr)=\mbox{\rm Im} \bigl( D_{\bar{u}} F \bigr)
\end{eqnarray}
and for every vector space $Z\subset  \mbox{\rm Ker}(D_{\bar{u}}\tilde{F})$ of dimension $\geq D+r-N-d$, 
\begin{eqnarray}\label{EQ11jan1}
\mbox{\rm ind}_- \left( \lambda \cdot \bigl( D^2_{\bar{u}} \tilde{F} \bigr)_{\vert Z}  \right) \geq r \qquad \forall \lambda \in \left( \mbox{\rm Im} \bigl( D_{\bar{u}} \tilde{F} \bigr)\right)^{\perp} \setminus \{0\}.
\end{eqnarray}
\end{lemma}

\begin{proof}[Proof of Lemma \ref{11jan1}]
Consider the $(N-1)$-dimensional sphere $S \subset (\R^N)^*$ defined by
\[
S:= \left\{ \lambda \in \left( \mbox{Im} \bigl( D_{\bar{u}} F  \bigr)\right)^{\perp} \ \vert \ | \lambda|=1 \right\} \subset \left(\R^N\right)^*.
\] 
By (\ref{ind}), for every $\lambda \in S$, there is a subspace $E_{\lambda}\subset \mbox{Ker} \left( D_{\bar{u}}F\right)$ of dimension $N+d$ such that 
\[
\lambda \cdot \bigl( D^2_{\bar{u}} F\bigr)_{\vert E_{\lambda} \setminus \{0\}} <0
\]
and moreover by continuity of the mapping $\nu \mapsto \nu \cdot \left( D^2_{u} F\right)_{\vert E_{\lambda}}$, there is indeed an open set $\mathcal{O}_{\lambda}\subset S$ containing $\lambda$ such that 
\[
\nu \cdot \bigl( D^2_{u} F \bigr)_{\vert E_{\lambda} \setminus \{0\}} <0 \qquad \forall \nu \in \mathcal{O}_{\lambda}.
\]
Therefore, by compactness of $S$ there are finitely many open sets $\mathcal{O}_{\lambda_1}, \ldots, \mathcal{O}_{\lambda_I}$ in $S$ such that
\[
S = \bigcup_{i=1}^I \mathcal{O}_{\lambda_i}.
\]
Pick now a finite dimensional space $V\subset X$ of dimension $N-r$ such that (note that $V\cap \mbox{Ker}(D_{\bar{u}}F)=\{0\}$)
$$
\mbox{Im} \left( D_{\bar{u}}F_{\vert V}\right) = \mbox{Im} \bigl( D_{\bar{u}} F\bigr)
$$
and define the finite dimensional vector space $W\subset X$, say of dimension $D$, by 
\[
 W := V \oplus 
 \left(\sum_{i=1}^I E_{\lambda_i} \right). 
\]
By construction, the restriction $\tilde{F}:W_{\bar{u}} \rightarrow \R^N$ of $F$ to $W_{\bar{u}} := \{\bar{u}\} + W$ satisfies (\ref{13janv1}) and 
\begin{eqnarray}\label{13jan1}
\mbox{ind}_- \left( \lambda \cdot \bigl( D^2_{\bar{u}} \tilde{F} \bigr)_{\vert \mbox{Ker} (D_{\bar{u}} \tilde{F})}  \right) \geq N+d \qquad \forall \lambda \in \left( \mbox{Im} \bigl( D_{\bar{u}} \tilde{F} \bigr)\right)^{\perp} \setminus \{0\}.
\end{eqnarray}
Furthermore, if $Z\subset \mbox{Ker} (D_{\bar{u}} \tilde{F})$ is a vector space of dimension $\geq D+r-N-d$ and $\lambda$ belongs to $( \mbox{Im} ( D_{\bar{u}} \tilde{F}))^{\perp} \setminus \{0\}$ then thanks to (\ref{13jan1}) there is a vector space $E\subset \mbox{Ker} (D_{\bar{u}} \tilde{F})$ of dimension $N+d$ such that 
\[
\lambda \cdot \bigl( D^2_{\bar{u}} \tilde{F} \bigr)_{\vert E \setminus \{0\}} <0
\]
and in addition we have 
\begin{eqnarray*}
\dim (Z\cap E) & = & \dim (Z) + \dim(E) -  \dim (Z+E) \\
& \geq &  (D+r-N-d) + (N+d) - D=r.
\end{eqnarray*}
This proves (\ref{EQ11jan1}).
\end{proof}

For every vector space $Z\subset  \mbox{Ker} (D_{\bar{u}} \tilde{F})$ of dimension $\geq D+r-N-d$, we define the vector space 
\[
X_Z :=  V \oplus Z  \subset X
\]
and the mapping $H_Z :  X_Z \longrightarrow \R^N$
by
\[
H_Z(v,z) := D_{\bar{u}}\tilde{F} (v) + \frac{1}{2}     \bigl( D^2_{\bar{u}} \tilde{F} \bigr) (z)        \qquad \forall (v,z) \in  V \times Z \simeq X_Z,
\]
and we set
\[
\mathcal{B}_Z (a) := \Bigl\{ (v,z) \in X_Z \, \vert \, \|v+z\| < a\Bigr\} \qquad \forall a>0. 
\]
The proof of the following lemma is moreorless the same as the proof of \cite[Lemma B.7 p. 135]{riffordbook}, we give it for the sake of completeness. 

\begin{lemma}\label{LEM13jan}
For every vector space $Z\subset  \mbox{\rm Ker} (D_{\bar{u}} \tilde{F})$ of dimension $\geq D+r-N-d$, there are $\mu_Z, c_Z>0$ such that  the image of any continuous mapping $H : \mathcal{B}_Z(1) \rightarrow \R^N$ with
\begin{eqnarray}\label{COND23july}
\sup \Bigl\{ \left| H(v,z)-H_Z(v,z) \right| \, \vert \, (v,z) \in \mathcal{B}_Z(1) \Bigr\} \leq \mu_Z
\end{eqnarray}
contains the ball $\bar{B}(0,c_Z)\subset \R^N$. 
\end{lemma}

\begin{proof}[Proof of Lemma \ref{LEM13jan}]
Let  $Z\subset  \mbox{Ker} (D_{\bar{u}} \tilde{F})$ a vector space of dimension $\geq D+r-N-d$ be fixed. Denote by $\mathcal{K}$ the orthogonal complement of $ \mbox{Im} ( D_{\bar{u}} \tilde{F})=\mbox{Im} ( D_{0} H_Z)$ in $\R^N$ which is a vector subspace of $\R^N$ of dimension $r$  and define the quadratic mapping $Q_Z: Z \rightarrow \mathcal{K}$ by 
\[
Q_Z(z) :=  \mbox{Proj}^{\perp}_{\mathcal{K}} \left[ \bigl( D^2_{\bar{u}} \tilde{F} \bigr)(z) \right] \qquad \forall z \in Z,
\]
where $\mbox{Proj}^{\perp}_{\mathcal{K}} : \R^N \rightarrow \mathcal{K}$ stands for the orthogonal projection to $\mathcal{K}$. By (\ref{13janv1})-(\ref{EQ11jan1}), we have 
\[ 
\mbox{\rm ind}_- \left( \Lambda^* \cdot Q_Z  \right) \geq r \qquad \forall \Lambda \in \mathcal{K} \setminus \{0\}.
\]
Hence by \cite[Lemma 20.8 p. 301]{as04} or \cite[Lemma B.6 p. 130]{riffordbook}, $Q_Z$ admits a regular zero $\bar{z}\in Z$. Thus, the point $\bar{z}\in  X_Z$ satisfies 
\[
\tilde{z} \in \mbox{Ker}(D_0H_Z) = Z,
\]
\[
 D_0^2H_Z \left(\bar{z},\bar{z}\right) \in  \mbox{Im} ( D_{0} H_Z) \quad (\mbox{because } Q_Z(\bar{z})=0 \Leftrightarrow D^2_{\bar{u}}\tilde{F} (\bar{z})\in \mathcal{K}^{\perp}=\mbox{Im} ( D_{0} H_Z) )
\]
and the linear mapping 
\[
 (v,z) \in \mbox{Ker}(D_0H_Z) \, \longmapsto \,  \mbox{Proj}^{\perp}_{\mathcal{K}}\left[ \bigl( D^2_{0} H_Z) \bigl(\bar{z},(v,z) \bigr)\right] \in \mathcal{K} =\left( \mbox{Im} \bigl( D_{0} H_Z\bigr)\right)^{\perp}
\]
is surjective, so by \cite[Lemma B.5 p. 129]{riffordbook} we infer that  there is a sequence $\{(v_i,z_i)\}_i$ in $X_Z$ converging to $0$ (w.r.t. $\|\cdot\|$) such that $H_Z(v_i,z_i)=0$ and $D_{(v_i,z_i)}H_Z$ is surjective for all $i$. Let $i$ be large enough such that $u_i:=(v_i,z_i)$ belongs to $\mathcal{B}_Z(1/4)$. Since $D_{u_i}H_Z$ is surjective, there is a affine space $Y \subset X_Z$ of dimension $N$ containing $u_i$ such that $D_{u_i}(H_Z)_{\vert Y}$ is invertible. So, by the Inverse Function Theorem, there is an open ball $\mathcal{B}:=B_X(u_i,\rho)\cap Y \subset \mathcal{B}_Z(1)$ centered at $u_i$ in $Y$  
such that the mapping
\[
(H_Z)_{\vert Y} \, : \, \mathcal{B}  \, \longrightarrow \,  (H_Z)_{\vert Y} (\mathcal{B}) \subset \R^N
\]
is a diffeomorphism. Denoting by $\mathcal{H} : (H_Z)_{\vert Y} (\mathcal{B}) \rightarrow \mathcal{B}$ its inverse, we pick some $c_Z>0$ such that 
\[
\bar{B}(0,c_Z) \subset (H_Z)_{\vert Y} (\mathcal{B}) \quad \mbox{and} \quad   
\mathcal{H}\left(\bar{B}(0,c_Z)\right) \subset B_X\bigl( u_i,\rho/4\bigr),
\]
and moreover we consider some $\mu_Z>0$ small enough such that any continuous mapping $H : \mathcal{B}_Z(1) \rightarrow \R^N$ verifying (\ref{COND23july}) satisfies
\[
H(u) \in (H_Z)_{\vert Y} (\mathcal{B}) \qquad \forall u \in B_X(u_i,\rho/2) \cap Y
\]
and
\[
\left|( \mathcal{H} \circ H)(u) -u \right| \leq \frac{\rho}{4} \qquad  \forall u\in B_X(u_i,\rho/2) \cap Y.
\]
We claim that by construction the image of any continuous mapping $H : \mathcal{B}_Z(1) \rightarrow \R^N$ verifying (\ref{COND23july}) contains the ball $\bar{B}(0,c_Z)$. As a matter of fact, for every $x\in \bar{B}(0,c_Z)$, the above construction implies that the function 
\[
\Psi : \bar{B}_X(\mathcal{H}(x),\rho/4) \cap Y \, \longrightarrow \, \bar{B}_X(\mathcal{H}(x),\rho/4)\cap Y
\]
defined by 
\[
\Psi (u) := u -( \mathcal{H} \circ H) (u) + \mathcal{H}(x) \qquad \forall u \in \bar{B}_X(\mathcal{H}(x),\rho/4) \cap Y,
\]
is continuous from the $N$-dimensional ball $\bar{B}_X(\mathcal{H}(x),\rho/4)\cap Y$ into itself. Therefore, by Brouwer's Theorem, $\Psi$ has a fixed point, that is, there is $u \in \bar{B}_X(\mathcal{H}(x),\rho/4)\cap Y$ such that 
\[
\Psi(u) = u \quad \Longleftrightarrow \quad \mathcal{H}\bigl(H(u)\bigr) = \mathcal{H}(x) \quad \Longleftrightarrow \quad H(u)=x,
\]
which concludes the proof of the lemma.
\end{proof}

Let $\bar{\epsilon}>0$ be such that $B_X(\bar{u},10\bar{\epsilon})\subset U$. For every vector space $Z\subset  \mbox{Ker} (D_{\bar{u}} \tilde{F})$ of dimension $e \in [D+r-N-d,D+r-N]$ (note that $ \mbox{Ker} (D_{\bar{u}} \tilde{F})$ has dimension $D+r-N$) and every vector space $Z'\subset W$ of dimension $e$, we denote by $D_H(Z,Z')$ the Hausdorff distance between $Z$ and $Z'$ over the unit ball, that is, (we set $B_X^1:=B_X(0,1)$)
\[
D_H(Z,Z') := \max \left\{ \sup_{u\in Z\cap B_X^1} \inf _{u\in Z\cap B_X^1} \left\{ \|u-u'\|\right\},  \sup_{u'\in Z'\cap B_X^1} \inf _{u\in Z\cap B_X^1} \left\{ \|u-u'\| \right\}\right\}
\]
and we denote by $\pi_{Z'} : W \rightarrow Z'$ the orthogonal projection to $Z'$ with respect to a fixed Euclidean metric in $W\simeq \R^D$. We note that, since norms in finite dimension are equivalent, there is $K>0$ which does not depend upon $Z, Z'$ such that there holds
\begin{eqnarray}\label{14janvroma}
\left\|\pi_{Z'} (z) -z \right\| \leq K d_H(Z,Z') \, \|z\| \qquad \forall z \in Z.
\end{eqnarray}
Then, for every $\epsilon \in (0,\bar{\epsilon})$ and every $u\in U$ with $\|u-\bar{u}\|<\bar{\epsilon}$, we define the function $\Phi_{Z,Z',u}^{\epsilon}:X_Z \rightarrow \R^N$ of class $C^2$ by
\[
\Phi_{Z,Z',u}^{\epsilon} (v,z) :=\frac{1}{\epsilon^2} \Bigl( F \bigl(u + \epsilon^2 v +\epsilon \pi_{Z'}(z) \bigr)-F(u) \Bigr)  \qquad \forall (v,z) \in X_Z.
\]
The following lemma follows from Taylor's formula at second-order (iii) above and Lemma \ref{LEM13jan}:

\begin{lemma}\label{13janvsoir}
For every vector space $Z\subset  \mbox{\rm Ker} (D_{\bar{u}} \tilde{F})$ of dimension $e\in [D+r-N-d,D+r-N]$, there is $\epsilon_Z>0$ such that  for every vector space $Z'\subset W$ of dimension $e$ satisfying 
\[
D_H(Z,Z') <\epsilon_Z,
\]
and every $u\in U$ with $\|u-\bar{u}\|<\epsilon_Z$, we have 
\[
\bar{B}(0,c_Z)\subset \Phi_{Z,Z',u}^{\epsilon} \left( \mathcal{B}_Z(1)\right) \qquad \forall \epsilon \in (0,\epsilon_Z).
\]
\end{lemma}

\begin{proof}[Proof of Lemma \ref{13janvsoir}]
Let  $Z\subset  \mbox{Ker} (D_{\bar{u}} \tilde{F})$ a vector space of dimension $e\in [D+r-N-d,D+r-N]$ be fixed. We claim that
\[
 \lim_{u \rightarrow \bar{u}, Z' \rightarrow Z, \epsilon\rightarrow 0} \, \sup \Bigl\{ \left| \Phi_{Z,Z',u}^{\epsilon} (v,z) -H_Z(v,z) \right| \, \vert \, (v,z) \in \mathcal{B}_Z(1) \Bigr\} =0.
\]
As a matter of fact, for every vector space $Z'\subset W$ of dimension $e$ and every $\epsilon \in (0,\bar{\epsilon})$, we have for every $(v,z)\in \mathcal{B}_Z(1)$,
\begin{multline*}
 \Phi_{Z,Z',u}^{\epsilon} (v,z) -H_Z(v,z) = \left( D_uF (v) +\frac{1}{2} D_u^2F \bigl(  \pi_{Z'}(z)\bigr)\right)\\ -\left( D_{\bar{u}}\tilde{F} (v) + \frac{1}{2}     \bigl( D^2_{\bar{u}} \tilde{F} \bigr) (z) \right)
 + \Theta^{\epsilon}_{Z,Z',u}(v,z),
\end{multline*}
with 
\begin{multline*}
 \Theta^{\epsilon}_{Z,Z',u}(v,z) = \frac{1}{\epsilon^2} \left( F \bigl(u + \epsilon^2 v +\epsilon \pi_{Z'}(z) \bigr)-F(u) -  D_uF \bigl(\epsilon^2 v\bigr) \right)- \frac{1}{2} D^2_uF\bigl( \pi_{Z'}(z) \bigr).
\end{multline*}
By $C^1$ regularity of $F$, we have 
\[
\lim_{u\rightarrow \bar{u}} \, \sup \Bigl\{ \left| D_uF(v) - D_{\bar{u}}\tilde{F}(v) \right| \, \vert \, (v,z) \in \mathcal{B}_Z(1) \Bigr\} =0
\]
and by $C^2$ regularity of $F$ along with (\ref{14janvroma}) we have  
\[
\lim_{u\rightarrow \bar{u}, Z' \rightarrow Z} \, \sup \left\{ \left| D_u^2F \bigl(  \pi_{Z'}(z)\bigr) - \bigl( D^2_{\bar{u}} \tilde{F} \bigr) (z)  \right| \, \vert \, (v,z) \in \mathcal{B}_Z(1) \right\} =0.
\]
Moreover, we can write for every  $(v,z)\in \mathcal{B}_Z(1)$,
\begin{eqnarray*}
&\quad & F \bigl(u + \epsilon^2 v +\epsilon \pi_{Z'}(z) \bigr)-F(u) -  D_uF \bigl(\epsilon^2 v\bigr)  \\
 &= &F \bigl(u + \epsilon^2 v +\epsilon \pi_{Z'}(z) \bigr)-F(u) -  D_uF \bigl(\epsilon^2 v+\epsilon \pi_{Z'}(z)\bigr)\\
& = & \frac{1}{2} \int_0^1 D^2_{tu+(1-t)(\epsilon^2 v+\epsilon \pi_{Z'}(z))} F \bigl(\epsilon^2 v+\epsilon \pi_{Z'}(z)\bigr) \, dt\\
& = & \frac{\epsilon^2}{2} \int_0^1 D^2_{tu+(1-t)(\epsilon^2 v+\epsilon \pi_{Z'}(z))} F \bigl(\epsilon v+ \pi_{Z'}(z)\bigr) \, dt,
\end{eqnarray*}
so by $C^2$ regularity of $F$ and (\ref{14janvroma}) we also infer that
\[
 \lim_{u \rightarrow \bar{u}, Z' \rightarrow Z, \epsilon\rightarrow 0} \, \sup \Bigl\{ \left|  \Theta^{\epsilon}_{Z,Z',u}(v,z) \right| \, \vert \, (v,z) \in \mathcal{B}_Z(1) \Bigr\} =0.
\]
Consequently, the claim is proved and Lemma \ref{LEM13jan} completes the proof. 
\end{proof}

We are ready to conclude the proof of Theorem \ref{THMopenquant}. First, we observe that for every integer $e\in [D+r-N-d,D+r-N]$, the set of vector spaces $Z \subset  \mbox{Ker} (D_{\bar{u}} \tilde{F})$ of dimension $e$ is compact with respect to the metric $D_H$. Hence, by Lemma \ref{13janvsoir}, we infer that there are $\hat{c}, \hat{\epsilon}\in (0,1)$ such that for every vector space $Z\subset  \mbox{\rm Ker} (D_{\bar{u}} \tilde{F})$ of dimension $e\geq [D+r-N-d,D+r-N]$, for every vector space $Z'\subset W$ of dimension $e$ with $D_H(Z,Z')< \hat{\epsilon}$ and for any $u\in U$ with $\|u-\bar{u}\|<\hat{\epsilon}$, there holds 
\begin{eqnarray}\label{flightEthiopian}
\bar{B}(0,\hat{c})\subset \Phi_{Z,Z',u}^{\epsilon} \left( \mathcal{B}_Z(1)\right) \qquad \forall \epsilon \in (0,\hat{\epsilon}).
\end{eqnarray}
Second, we note that for every $u \in U$ close enough to $\bar{u}$, the two vector spaces $Z^{\prime}_u, Z_u\subset W$ defined by ($\mbox{Proj}^{\perp}_{E}$ stands for the orthogonal projection to a vector space $E\subset W$ with respect to the fixed Euclidean metric in $W\simeq \R^D$)
\[
Z^{\prime}_{u}:=\mbox{Ker} \left( D_uF\right)  \cap \mbox{Ker} \left( D_uG\right) \cap W \quad \mbox{and} \quad Z_{u}:= \mbox{Proj}^{\perp}_{\mbox{Ker}(D_{\bar{u}}\tilde{F})} \bigl( Z^{\prime}_u\bigr)
\]
have the same dimension in $ [D+r-N-d,D+r-N]$. Thus, there is $\delta \in (0,1)$ such that (\ref{flightEthiopian}) is satisfied for any $u\in U$ with $\|u-\bar{u}\|<\delta$, $Z=Z_{u}$ and $Z'=Z^{\prime}_{u}$; furthermore we note that we can assume also that $\|\pi_{Z'}(z)\|<C\|z\|$ for some $C\geq 1$. Let us now show that the conclusions of Theorem \ref{THMopenquant} are satisfied with $K:=C/\sqrt{\hat{c}}$ and $\rho:=\hat{\epsilon}^2\hat{c}$. Given $u\in U$ with $\|u-\bar{u}\|<\delta$, if  $x\in \R^N$ verifies $|x-F(u)|< \rho$ then we have  
\[
 \epsilon:= \sqrt{ \frac{|x-F(u)|}{\hat{c}}} < \hat{\epsilon} \quad \mbox{and} \quad \frac{1}{\epsilon^2} \bigl(x-F(u)\bigr) < \bar{c}.
\]
So, by (\ref{flightEthiopian}) (with  $Z=Z_{u}$ and $Z'=Z^{\prime}_{u}$), there are $(v,z)\in \mathcal{B}_Z(1)$ such that $ \Phi_{Z,Z',u}^{\epsilon} (v,z)=(x-F(u))/\epsilon^2$ which yields 
\[
x=F(u+w_1+w_2) \quad \mbox{with} \quad w_1:=\epsilon \pi_{Z'}(z) \mbox{ and } w_2:= \epsilon^2 v,
\]
where 
\[
w_1 \in \mbox{Ker}\left(D_uF\right) \cap \mbox{Ker}\left(D_uG\right) 
\]
\[
\mbox{and} \quad \|w_1\|<K \sqrt{|x-F(u)|}, \quad \|w_2\|<K |x-F(u)|
\]
are satisfied by construction. 
\end{proof}


\end{document}